%% file: HierarchicalMain.tex
\documentclass[12pt,fleqn,reqno]{article}
\usepackage{epsfig,amssymb,latexsym,amsmath,pifont,multicol,mathtools,amsmath,verbatim,amsthm,float,lipsum}
\usepackage{caption} 
\usepackage[utf8]{inputenc}
\usepackage[T1]{fontenc}

\usepackage[varg]{txfonts}
\let\mathbb=\varmathbb

\usepackage[sans]{dsfont}

\usepackage[left=2.5cm, right=2.5cm, top=2.5cm, bottom=2.5cm]{geometry}

\usepackage[%
cal=euler,
bb=fourier,
scr=zapfc,
]
{mathalfa}

\usepackage[font=small,labelfont=bf]{caption}
\usepackage{subfigure}
\usepackage{graphicx}
\graphicspath{{Figures/}} 
\usepackage{acronym}
\usepackage{latexsym}
\usepackage{paralist}
\usepackage{wasysym}
\usepackage{xspace}
\usepackage{framed}
\usepackage{palatino,pxfonts}
\usepackage{authblk}
\usepackage{enumitem}
\usepackage[numbers,sort&compress]{natbib}

\usepackage[dvipsnames,svgnames]{xcolor}
\colorlet{MyBlue}{DodgerBlue!75!Black}
\colorlet{MyGreen}{DarkGreen!95!Black}

\usepackage{hyperref}
\hypersetup{
colorlinks=true,
linktocpage=true,
pdfstartview=FitH,
breaklinks=true,
pdfpagemode=UseNone,
pageanchor=true,
pdfpagemode=UseOutlines,
plainpages=false,
bookmarksnumbered,
bookmarksopen=false,
bookmarksopenlevel=1,
hypertexnames=true,
pdfhighlight=/O,
urlcolor=MyBlue!60!black,linkcolor=MyBlue!70!black,citecolor=DarkGreen!70!black, 
pdftitle={},
pdfauthor={},
pdfsubject={},
pdfkeywords={},
pdfcreator={pdfLaTeX},
pdfproducer={LaTeX with hyperref}
}

\numberwithin{equation}{section}  
\usepackage[sort&compress,capitalize,nameinlink]{cleveref}
\crefname{example}{Ex.}{Exs.}

\crefrangeformat{equation}{\upshape(#3#1#4)\textendash(#5#2#6)}

\usepackage{mathtools}
\usepackage[para,online,flushleft]{threeparttable}
\usepackage{multirow}
\newcommand{\pmat}[1]{\begin{pmatrix} #1 \end{pmatrix}}
\usepackage{physics}
\usepackage{float}
\newcommand{\eps}{\varepsilon}
\DeclareMathOperator*{\argmin}{argmin}

\DeclareMathOperator{\dom}{dom}

\DeclareMathOperator{\gr}{graph}

\DeclareMathOperator{\Id}{Id}

\newcommand{\const}{\mathtt{c}}

\newcommand{\ca}{\mathtt{a}}

\newcommand{\ce}{\mathtt{e}}

\newcommand{\bA}{{\mathbf A}}

\newcommand{\bY}{\mathbf{Y}}

\newcommand{\bx}{\mathbf{x}}
\newcommand{\by}{\mathbf{y}}
\newcommand{\bu}{\mathbf{u}}
\newcommand{\bs}{\mathbf{s}}
\newcommand{\bz}{\mathbf{z}}

\newcommand{\bZ}{\mathbf{Z}}

\newcommand{\bw}{\mathbf{w}}
\newcommand{\bW}{\mathbf{W}}
\newcommand{\bv}{\mathbf{v}}

\renewcommand{\iff}{\Leftrightarrow}

\renewcommand{\emptyset}{\varnothing}
\newcommand{\eqdef}{\triangleq}

\newcommand{\scrA}{\mathcal{A}}
\newcommand{\scrB}{\mathcal{B}}

\newcommand{\scrD}{\mathcal{D}}

\newcommand{\scrF}{\mathcal{F}}
\newcommand{\scrG}{\mathcal{G}}

\newcommand{\scrI}{\mathcal{I}}

\newcommand{\scrK}{\mathcal{K}}
\newcommand{\scrL}{\mathcal{L}}

\newcommand{\scrO}{\mathcal{O}}

\newcommand{\scrS}{\mathcal{S}}

\newcommand{\scrX}{\mathcal{X}}
\newcommand{\scrY}{\mathcal{Y}}

\newcommand{\Nscr}{{\cal N}}
\newcommand{\Tscr}{{\cal T}}
\newcommand{\Kscr}{{\cal K}}
\newcommand{\Fscr}{{\cal F}}
\newcommand{\Real}{\mathbb{R}}
\renewcommand{\Pr}{\mathbb{P}}
\newcommand{\Ex}{\mathbb{E}}

\newcommand{\F}{{\mathbb{F}}}
\newcommand{\Normal}{\mathsf{N}}
\newcommand{\Uniform}{\mathsf{U}}
\newcommand{\Vol}{{\mathbf{Vol}}}

\newcommand{\1}{\mathbf{1}}
\newcommand{\Rn}{\R^n}

\newcommand{\R}{\mathbb{R}}

\newcommand{\N}{\mathbb{N}}

\newcommand{\sphere}{\mathbb{S}}
\DeclareMathOperator{\NC}{\mathsf{NC}}

\newcommand{\ball}{\mathbb{B}}
\newcommand{\nash}{\mathbf{NE}}

\DeclareMathOperator{\VI}{VI}

\DeclareMathOperator{\MVI}{MVI}
\DeclareMathOperator{\SOL}{SOL}

\DeclareMathOperator{\OC}{\mathsf{OC}}

\DeclareMathOperator{\SFBF}{\texttt{SFBF}}
\DeclareMathOperator{\ISFBF}{\texttt{ISFBF}}
\DeclareMathOperator{\VRHGS}{\texttt{VRHGS}}

\usepackage[ruled,vlined]{algorithm2e}
\usepackage{algpseudocode}								

\theoremstyle{plain}
\newtheorem{theorem}{Theorem}

\newtheorem*{corollary*}{Corollary}
\newtheorem{lemma}[theorem]{Lemma}
\newtheorem{proposition}[theorem]{Proposition}
\newtheorem{fact}[theorem]{Fact}

\theoremstyle{definition}
\newtheorem{definition}[theorem]{Definition}
\newtheorem*{definition*}{Definition}
\newtheorem{assumption}{Assumption}

\renewcommand{\qed}{\hfill{\footnotesize$\blacksquare$}}

\theoremstyle{remark}
\newtheorem{remark}{Remark}
\newtheorem*{remark*}{Remark}
\newtheorem*{notation*}{Notational remark}

\numberwithin{theorem}{section}
\numberwithin{remark}{section}
\numberwithin{example}{section}

\DeclarePairedDelimiter{\inner}{\langle}{\rangle}

\title{A regularized variance-reduced modified extragradient method for stochastic hierarchical games}
\date{\today}

\author[1]{\small Shisheng Cui}
\author[2]{\small Uday V. Shanbhag}
\author[3]{\small Mathias Staudigl\thanks{Corresponding Author}}

\affil[1]{\footnotesize School of Automation, Beijing Institute of Technology\\
Beijing 100081, China\\
(\href{mailto:css@bit.edu.cn}{css@bit.edu.cn})
}

\affil[2]{\footnotesize Department of Industrial and Manufacturing Engineering \\
Pennsylvania State University, University Park, USA \\
(\href{mailto:udaybag@psu.edu}{udaybag@psu.edu})}

\affil[3]{\footnotesize Department of Mathematics, Universität Mannheim, \\
B6, 68159, Mannheim, Germany\\
(\href{mailto:mathias.staudigl@uni-mannheim.de}{mathias.staudigl@uni-mannheim.de})}

\begin{document}

\maketitle

\begin{abstract}%
We consider an $N$-player hierarchical game in which the $i$th player's
objective comprises of an expectation-valued term, parametrized by rival
    decisions, and a hierarchical term. Such a framework allows for capturing a
    broad range of stochastic hierarchical optimization problems, Stackelberg
    equilibrium problems, and leader-follower games. We develop an iteratively regularized and smoothed
    variance-reduced modified extragradient framework for iteratively approaching hierarchical equilibria in a stochastic setting. We equip our analysis
    with rate statements, complexity guarantees, and almost-sure convergence
    results. We then extend these statements to settings where the lower-level problem is solved inexactly and provide the corresponding rate and complexity statements. Our model framework encompasses many game theoretic equilibrium problems studied in the context of power markets. We present a realistic application to the virtual power plants, emphasizing the role of hierarchical decision making and regularization.
    \end{abstract}

\section{Introduction}
\label{sec:intro}
\input{Introduction}
\section{Preliminaries}\label{sec:prelims}
\input{prelims}

\section{A variance reduced forward-backward-forward algorithm for hierarchical games}
\label{sec:algorithm}
\input{FBF}
\section{Main Results}
\label{sec:analysis}
\input{analysis}

\section{Hierarchical games in power markets}
\label{sec:application}
\input{application}

\section{Proof of the Main Theorem}
\label{sec:proofs}
\subsection{Analysis of the exact scheme}
\label{sec:exact}
\input{Appendix_Analysis}

\subsection{Analysis of the inexact scheme}
\label{sec:analysisinexact}
\input{Appendix_Inexact}

\section{Conclusion}
\label{sec:conclusion}
\input{Conclusion}

\subsubsection*{acknowledgements}
The research of MST benefited from the support of the FMJH Program Gaspard Monge for optimization and operations research and their interactions with data science. CS acknowledges supported in part by NSFC under Grant 62373050.

\begin{appendix}
\section{Auxiliary Facts}
\subsection{Generalities}
\label{app:general}
\input{Appendix_General}

\subsection{Variational inequalities}
\label{app:VIs}
\input{Appendix-VI}

\subsection{Smoothing}
\label{app:smoothing}
\input{smoothing}

\subsection{Inexact Implementation}
\label{app:inexact}
\input{inexactImplementation}
\end{appendix}

\bibliographystyle{plain}
\bibliography{mybib,uvsbib}

\end{document}

%% file: Introduction.tex
%
In this paper we consider a class of stochastic hierarchical optimization problems and games, generalizing many learning problems involving sequential optimization. Consider a collection of $N$-agents, where the $i$-th agent solves the optimization problem parametrized by rival decisions $\bx_{-i}$:
\begin{equation}\label{eq:Opt}\tag{P}
\min_{\bx_{i},\by_{i}}\, \{\ell_{i}(\bx_{i},\bx_{-i},\by_{i})\eqdef f_{i}(\bx_{i},\bx_{-i})+g_{i}(\bx_{i},\by_{i})\},\text{ s.t. }\bx_{i}\in\scrX_{i},\by_{i}\in\SOL(\phi_{i}(\bx_{i},\cdot),\scrY_{i}).
\end{equation}
We let $i\in\scrI\eqdef\{1,\ldots,N\}$ represent a set of \emph{leaders},
characterized by two loss functions: (i) $f_{i}(\bx)\eqdef\Ex_{\xi}[F_{i}(\bx,\xi)]$, which depends on the entire action profile $\bx\eqdef (\bx_{i},\bx_{-i})=(\bx_{1},\ldots,\bx_{N})\in\scrX\eqdef\prod_{i\in\scrI}\scrX_{i}$;  (ii) $g_{i}(\bx_{i},\by_{i})$ is a deterministic function, jointly controlled
by leader $i$'s decision variable $\bx_{i}\in\scrX_{i}$ and a follower's
decision variable $\by_{i}\in\scrY_{i}$. Each leader's optimization problem
exhibits two sets of  private constraints, the first given by $\bx_{i}\in\scrX_{i}\subseteq\R^{n_{i}}$, while the second are
equilibrium constraints represented by the solution set of a parameterized
variational inequality $\VI(\phi_{i}(\bx_{i},\cdot),\scrY_{i})$, which reads as
\begin{equation}\label{eq:VI}
\text{Find }\by_{i}(\bx_{i})\in\scrY_{i}\text{ satisfying }\inner{\phi_{i}(\bx_{i},\by_{i}(\bx_{i})),\by_{i}-\by_{i}(\bx_{i})}\geq 0\quad \forall \by_{i}\in\scrY_{i}.
\end{equation}
We denote the set of points $\by_{i}$ satisfying this condition by $\SOL(\phi_{i}(\bx_{i},\cdot),\scrY_{i})$.
This VI is defined in terms of a closed convex set
$\scrY_{i}\subseteq\R^{m_{i}}$ and an expectation-valued mapping
$\phi_{i}(\bx,\by_{i})\eqdef\Ex_{\xi}[\Phi_{i}(\bx,\by_{i},\xi)]$. All the
problem data are affected by random noise represented by a random variable
$\xi:\Omega\to\Xi$, defined on a probability space $(\Omega,\scrF,\Pr)$ and
taking values in a measurable space $\Xi$. Such hierarchical optimization
problems traditionally play a key role in operations research and engineering,
where they are deeply connected to bilevel programming
\citep{dempe02foundations} and mathematical programs under equilibrium
constraints (MPEC) \citep{Luo:1996td}. In fact, the canonical MPEC formulation
is obtained from \eqref{eq:Opt} when $N=1$. The multi-agent formulation
\eqref{eq:Opt} also relates to leader-follower and Stackelberg games
(cf.~\citep{sherali1984multiple,pang05quasi,DeMiguel:2009aa,kulkarni15existence})
which are a traditional model in economics, and also have received increased
attention in machine learning recently
\citep{BilevelLearningIEEE,Fiez:2020aa,NEURIPS2021_d8211837,Blau:2022aa}.
Economic equilibria in power markets have been extensively studied using a
complementarity framework
(cf.~\cite{hobbs01linear,hobbspang,hobbs07nash,gabriel2013optimality}). More
recently, stochastic generalizations have been examined where uncertainty in
price and cost functions have been
addressed~\cite{shanbhag11complementarity,kannan2011strategic}.  As an
immediate application of our algorithmic framework, we present in Section
\ref{sec:application} a model inspired by Hobbs and Pang~\cite{hobbs07nash},
but suitably modified to account for uncertainty in prices and costs,
multi-period settings with ramping constraints, and the  incorporation of
virtual power plants (VPPs) (see~\cite{NAVAL2021111393,DOE_VPPLiftOff_2023} for
a review of VPPs and power markets).

\subsection{Our contributions and related work}

\paragraph{Hierarchical optimization, games and uncertain generalizations.}

To date, hierarchical optimization has been studied under the umbrella of
bilevel programming~\citep{dempe02foundations,dempe20bilevel} and mathematical programs with
equilibrium constraints (MPECs)~\citep{Luo:1996td,outrata98nonsmooth}.
Algorithmic schemes for resolving MPECs where the lower-level problem
is an optimization problem, or a variational inequality, have largely emphasized either implicit
approaches~\citep{outrata98nonsmooth} or regularization/penalization-based
techniques~\citep{Luo:1996td,leyffer2006interior,xu04convergence}.
Yet, there appears to have been a glaring lacuna in non-asymptotic rate and
complexity guarantees for resolving hierarchical optimization and their
stochastic and game-theoretic variants. This gap has been
partially addressed in the recent papers \cite{Cui:2022aa} and \cite{Cui2022}. Both papers
present variance-reduced solution strategies for various versions of
hierarchical optimization problems and games, respectively, relying on variance
reduction via a sequence of increasing mini-batches. Finite-time and almost
sure convergence to solutions is proved under convexity/monotonicity
assumptions on the problem data. However, the framework for monotone games
in~\cite{Cui2022} requires exact solutions of lower-level problems,
significantly impacting its efficient implementation in large-scale settings.
We complement this literature by developing a novel regularized
smoothed variance reduction method for the family of hierarchical games
\eqref{eq:Opt}, building on a disciplined operator splitting approach.
Notably, we provide an inexact generalization allowing for random
error-afflicted lower-level solutions, addressing a significant shortcoming
in~\cite{Cui2022}.  Specifically, in this paper we improve \cite{Cui2022} along
two important dimensions: (i) First, we allow for inexact resolution of
lower-level problems to accommodate large-scale stochastic follower problems;
(ii) Second, we provide a novel variance-reduction framework for addressing
this problem. Despite the need for inexactness, our statements match the
state-of-the-art both in terms of rate and oracle complexity. Algorithmically,
these advancements are achieved via a novel stochastic operator splitting
approach that combines ideas from iterative smoothing and regularization
\cite{YouNedSha17}, with modern variance reduction approaches originating in
machine learning \cite{Gower:2020ty}.

\paragraph{Zeroth-order optimization, smoothing, and regularization.}
Zeroth-order (gradient-free) optimization is being increasingly embraced for
solving machine learning problems where explicit expressions of the gradients
are difficult or infeasible to obtain. In hierarchical optimization problems,
this is particularly relevant when solutions of the lower level problem are
injected into the leader's upper-level problem. In machine learning, this
problem is known as approximating the hypergradient. Various techniques for
estimating this object have been studied recently, ranging from truncated von
Neuman series \cite{Ghadimi:2018aa} and fully first-order methods
\cite{Kwon:2023aa}. Instead of computationally expensive first-order (or
higher) information about the problem data, we develop an online stochastic
approximation approach based solely on function evaluations in order to
approximate the directional derivative of the coupling function between the
upper and the lower level. Approximating the directional derivative of the thus
obtained implicit function has a long history~\citep{steklov1907} and has been
employed for resolving stochastic optimization
~\cite{hariharan08decentralized,YousNedSha10,Farzad1} and variational
inequality problems~\cite{yousefian2013regularized,YouNedSha17}. 


\paragraph{Variance Reduction.}
Variance reduction is a commonly employed method exploiting the finite-sum
structure of variational problems arising in machine learning and engineering.
The classical stochastic variance reduced gradient (SVRG)
\citep{johnson13accelerating} is embedded within a double loop structure and
tailored to the prototypical finite-sum structure in empirical risk
minimization. Indeed, in the classical SVRG formulation full gradients are
computed ``from time-to-time'' in the outer loop while cheap variance reduced
gradients are used in the frequently activated inner loop subroutine. This
construction has been extended to saddle-point problems and stochastic monotone
inclusions in \cite{PalBac16}. Extensions to monotone \emph{mixed}-variational
inequality problems were recently provided in
\citep{Alacaoglu:2021aa,Alacaoglu:2022aa} and \cite{ChavGidelNips19}. In
contrast, we consider stochastic hierarchical games over general sample spaces,
complicated by the presence of nested optimization problems embodied by the
interaction between leaders and followers. While the assumptions we make in
this work allow us to recast the problem as a mixed variational inequality,
several challenges persist. First, subgradients of $g_i(\cdot,\by_i(\cdot))$
are unavailable; Second, enlisting smoothing approaches requires
$\by_i(\bx_i)$, unavailable in closed form; Third, we are not restricted to
finite-sum regimes and allow for general sample spaces by employing
increasingly large batch-sizes to approximate the gradient in the outer loop.

%

%% file: prelims.tex
%

In this section, we articulate the standing assumptions employed in this paper and introduce our notation. The decision set of leader $i$ is a subset $\scrX_{i}$ in $\R^{n_{i}}$. We let $\scrX\eqdef\prod_{i\in\scrI}\scrX_{i}$ represent the set of strategy profiles of the leaders and identify it with a subset of $\Rn$, where $n\eqdef\sum_{i\in\scrI}n_{i}$. For any $d\geq 1$, we let $\ball_{d}\eqdef\{\bx\in\R^{d}\vert\;\norm{\bx}\leq 1\}$ denote the unit ball in $\R^{d}$.
We start by introducing a basic assumption on the follower's problem. It bears reminding that the VI representation of the follower problem allows for capturing a range of problems, ranging from smooth convex optimization problems to more intricate smooth convex games and equilibrium problems; see~\cite{FacPan03}.\footnote{Appendix \ref{app:VIs} explains the terminology related to VIs.}
\begin{assumption}\label{ass:LLunique}
   For $i \, \in \, \scrI$, $\scrY_i\subseteq\R^{m_{i}}$ is closed and convex set, and for all $\bx_{i}\in\scrX_{i}$, the mapping $\phi_i(\bx_i,\cdot):\scrY_{i}\to\R^{m_{i}}$ is strongly monotone and Lipschitz continuous.
\end{assumption}
By Assumption~\ref{ass:LLunique}, for any $i \, \in \, \scrI$, the set of solutions to $\VI(\phi_{i}(\bx_{i},\cdot),\scrY_{i})$, denoted by $\SOL(\phi_{i}(\bx_{i},\cdot),\scrY_{i})$, is single-valued with unique element $\by_{i}(\bx_{i})$. Moreover, $\phi_i(\bx_i,\by_i) = \mathbb{E}_{\xi} \left[ \Phi_i(\bx_i,\by_i,\xi)\right]$ for any $(\bx_i,\by_i) \in \scrX_i \times\scrY_i$.
\begin{assumption}\label{ass:standing}
The following assumptions hold for each leader $i\in\scrI$:
\begin{itemize}
    \item[{(i)}] The set $\scrX_{i}\subset\R^{n_{i}}$ is nonempty, compact, and convex. In particular, there exists $C_{i}>0$ such that $\sup_{\bx_{i},\bx'_{i}\in\scrX_{i}}\norm{\bx_{i}-\bx'_{i}}\leq C_{i}$ for all $i\in\scrI$.  
\item[{(ii)}] For some $\delta_{0}>0$, the mapping $\bx_{i}\mapsto g_{i}(\bx_{i},\by_{i}(\bx_{i}))$ is $L_{1,{i}}$-Lipschitz on $\scrX_{i,\delta_{0}}\eqdef\scrX_{i}+\delta_{0}\ball_{n_{i}}$.
\item[{(iii)}] $\bx_{i}\mapsto F_{i}((\bx_{i},\bx_{-i}),\xi)$ is convex and continuously differentiable over an open set containing $\scrX_{i}$, uniformly for all $\bx_{-i}\in\scrX_{-i}$ and almost every $\xi\in\Xi$. 
\item[{(iv)}] The mapping $\by_{i}\mapsto g_{i}(\bx_{i},\by_{i})$ is $L_{2,i}$-Lipschitz continuous for all $\bx_{i}\in\scrX_{i,\delta_{0}}$.
\item[{(v)}] The operator $V:\Rn\to\R^{n},$ defined by $V(\bx)=(V_{i}(\bx))_{i\in\scrI}$ and $V_{i}(\bx)\eqdef\nabla_{\bx_{i}}f_{i}(\bx)$, is $L_{f}$-Lipschitz continuous and monotone on $\scrX_{\delta_{0}}=\prod_{i\in\scrI}\scrX_{i,\delta_{0}}$.
\end{itemize}
\end{assumption}
Assumption \ref{ass:LLunique} is commonly employed in hierarchical optimization problems. Indeed, in the special case where the VI captures the optimality conditions of a parametrized convex optimization problem solved by the follower, then strong monotonicity of $\phi_i(\bx_{i},\cdot)$, is equivalent to strong convexity of the follower's cost function, an assumption that is known in bilevel optimization literature as the \emph{lower level uniqueness property} \cite{BilevelLearningIEEE}.\\

Given Assumption \ref{ass:LLunique} we define the implicit loss function $L_{i}:\scrX\to\R\cup\{\infty\}$ by
\begin{equation}\label{eq:L}
L_{i}(\bx_{i},\bx_{-i})\eqdef \ell_{i}(\bx_{i},\bx_{-i},\by_{i}(\bx_{i}))=\Ex_{\xi}[F_{i}((\bx_{i},\bx_{-i}),\xi)]+h_{i}(\bx_{i}), 
\end{equation}
where 
$
h_{i}(\bx_{i})\eqdef g_{i}(\bx_{i},\by_{i}(\bx_{i})).
$
In terms of the implicit loss function \eqref{eq:L}, we convert the hierarchical game \eqref{eq:Opt} into a \emph{stochastic Nash equilibrium problem} in which each player solves the loss minimization problem 
\begin{equation}\label{eq:upperlevel}
(\forall i\in\scrI):\quad \min_{\bx_{i}\in\scrX_{i}}L_{i}(\bx_{i},\bx_{-i}).
\end{equation}
We refer to \eqref{eq:upperlevel} as the \emph{upper-level problem}, and summarize it as the tuple $\scrG^{\text{upper}}\eqdef\{L_{i},\scrX_{i}\}_{i\in\scrI}$. Let $\nash(\scrG^{\text{upper}})$ denote the set of Nash equilibria of the game $\scrG^{\text{upper}}$ and $\scrY = \prod_{i\in\scrI}\scrY_{i}$.
\begin{definition}
    A $2N$-tuple $(\bx^{\ast},\by^{\ast})\in\scrX\times\scrY$ is called a \emph{hierarchical equilibrium} if $\bx^{\ast}\in\nash(\scrG^{\text{upper}})$ and, for all $i\in\scrI$, $\by_{i}^{\ast}=\by_{i}(\bx_{i}^{\ast})$ the unique solution of $\VI(\phi_{i}(\bx^{\ast}_{i},\cdot),\scrY_{i})$. 
\end{definition}
Typical online learning approaches in game theory employ stochastic
approximation (SA) for iteratively approaching a Nash equilibrium of the game
$\scrG^{\text{upper}}$. These iterative methods rely on the availability of a stochastic oracle revealing (noisy) first-order information about the operators involved (i.e. samples of pseudo-gradients of the objective  for each individual player). Such direct methods are complicated in hierarchical optimization since the required subgradient is an element of the subdifferential of the sum of two Lipschitz continuous functions. In the current setting, this task is even more complicated since the upper level objective is defined by a function which is available only in an implicit form, as it depends on the solution of the lower level problem $\by_{i}(\bx_{i})$, and another function given in terms of an expected value. Thus, even if a sum-rule for a subdifferential applies \cite{Cla98}, it would read as 
$
\partial_{\bx_{i}}L(\bx_{i},\bx_{-i})=\partial_{\bx_{i}}f_{i}(\bx)+\partial_{\bx_{i}}(g_{i}\circ(\Id,\by_{i}(\cdot)))(\bx_{i}).
$
Hence we would need to invoke a non-smooth chain rule for our chosen version of a subdifferential, in order to evaluate
$\partial_{\bx_{i}}(g_{i}\circ(\Id,\by_{i}(\cdot)))(\bx_{i})$. In addition, we would require access to the subdifferential of $\by_i(\bullet)$ at $\bx_i$.  We circumvent this
computationally challenging step by developing a random search procedure based on a finite difference approximation. To develop such scheme, recall that we defined
$h_{i}(\bx_{i})\eqdef g_{i}(\bx_{i},\by_{i}(\bx_{i}))$ as the loss function
coupling the leader and the follower. The following fact can be found in
Proposition 1 in \cite{Cui:2022aa}.
\begin{lemma}\label{lem:hisemismooth}
   Let Assumptions \ref{ass:LLunique}-\ref{ass:standing} hold. Then $h_{i}(\bx_{i})=g_{i}(\bx_{i},\by_{i}(\bx_{i}))$ is $L_{h_{i}}$-Lipschitz continuous on $\scrX_{i}$ and directionally differentiable.
\end{lemma}
To proceed, we impose a convexity requirement on $h_i(\cdot,\by_i(\cdot))$.
\begin{assumption}\label{ass:implicit}
The implicitly defined function $\bx_{i}\mapsto h_{i}(\bx_{i})\eqdef g_{i}(\bx_{i},\by_{i}(\bx_{i}))$ is convex on $\scrX_{i}$.
\end{assumption}
\begin{remark}
Several papers in the literature provided conditions under which the implicit function $h_{i}$ is indeed convex in hierarchical settings. A structural model framework where convexity provably holds is described in \cite{doi:10.1137/120863873}, and \cite{DeMiguel:2009aa}. 
\end{remark}
\begin{remark}
 Under Assumptions \ref{ass:standing} and \ref{ass:implicit}, an equilibrium exists using classical results due to \cite{Glicksberg:1952aa}. Indeed, since the lower level problem is assumed to have a unique solution $\by_{i}(\bx_{i})$ and Lemma \ref{lem:hisemismooth} guarantees that the noncooperative game $\scrG^{\text{upper}}$ satisfies conditions of \cite[Th.~I.4.1]{Mertens:2015aa}, implying $\nash(\scrG^{\text{upper}})\neq\emptyset$. 
\end{remark}
Assumption~\ref{ass:standing}(v), Assumption \ref{ass:implicit} and Lemma~\ref{lem:hisemismooth}, yield a variational characterization of elements of $\nash(\scrG^{\text{upper}})$ in terms of an expectation-valued \emph{mixed-variational inequality} (cf. Appendix~\ref{app:VIs}). 
\begin{lemma}\label{lem:VI-hierarchical}
Let $h:\scrX\to\R$ be defined by $h(\bx)\eqdef\sum_{i\in\scrI}h_{i}(\bx_{i})$. Then, $\bx^{\ast}\in\nash(\scrG^{\text{upper}})$ if and only if $\bx^{\ast}$ solves the mixed variational inequality $\MVI(V,h)$:
\begin{equation}\label{eq:MVImain}
\inner{V(\bx^{\ast}),\bx-\bx^{\ast}}+h(\bx)-h(\bx^{\ast})\geq 0\qquad\forall \bx\in\scrX.
\end{equation}
\end{lemma}
Solutions to mixed VIs with expectation-valued operators
have been developed recently in cases where the random variable takes values in
a finite set, and/or when the VI is derived from a zero-sum game displaying a
finite-sum structure \citep{Alacaoglu:2021aa,Alacaoglu:2022aa}. The standard
algorithmic approach to iteratively approximate a solution to such
structured VIs are extragradient type of methods. A direct application of these
methods to the mixed VI \eqref{eq:MVImain} is complicated because of the
following facts: 

\noindent (i) \em $L$-smoothness of $h_i$. The assumptions made thus far do not
guarantee the differentiability of $h_i$ with a Lipschitz continuous gradient.
Hence, a direct application of gradient, or extragradient methods, is a
difficult task in our setting. To cope with this technical difficulty, we
develop a smoothing approach, yielding a family of approximating models
enjoying the typical Lipschitz smoothness requirements.\\ 
(ii)  \em Randomness in the operator $V$: Since the operator $V$ is only available in
terms of an expected value, in general, we cannot tractably evaluate it.
Instead we have to use simulation-based methods to obtain random estimators of
this mathematical expectation. To keep this simulation task within a feasible
computational budget, iterative variance reduction ideas are frequently used in
iterative methods for generating the input data. Again, these standard variance
reduction techniques rely on smoothness of the data. In our case,
non-smoothness is present in terms of the implicit function $h_{i}(\cdot)$. In
principle, one could apply splitting techniques to deal with the non-smooth
function via a proximal smoothing.  However, this approach requires
$h_{i}(\cdot)$ to be proximable for which we have no a-priori guarantee since
it is the value function of the leader, derived
from the solution of the follower. With these preparatory remarks in mind, we now explain the design of our algorithmic solution strategy for the hierarchical game problems \eqref{eq:Opt}.

%% file: FBF.tex
%

In this section, we present our algorithm for computing
an equilibrium of the hierarchical game~\eqref{eq:Opt}. As in the seminal
SVRG formulation, our method runs in two loops. Each loop requires as inputs data that are computed in the outer loop. The inputs of the inner and outer loops are
constructed as follows.  For  $\eta>0$, we define the (Tikhonov) regularized vector field $V^{\eta}:\scrX\to\R^{n}$ by
\begin{equation}\label{eq:Vreg}
    V^{\eta}(\bx)=(V_{i}^{\eta}(\bx))_{i\in\scrI},\mbox{ where } V^{\eta}_{i}(\bx)\eqdef V_{i}(\bx)+\eta \bx_{i}\quad\forall i\in\scrI. 
\end{equation}
Tikhonov regularization is a classical tool to obtain stronger convergence results in numerical schemes. It has been examined for deterministic \citep{kannan12distributed} and stochastic equilibrium problems \citep{koshal13regularized,YouNedSha17}.\\

Our next assumption is concerned with the nature of the stochastic oracle which generated random estimators on the expectation-valued operator $V$ when queried at a given point $\bx$.
\begin{assumption}\label{ass:SO_V}
The operator $V$ has a stochastic oracle $\hat{V}(\cdot,\xi)$ that is
\begin{enumerate}
\item unbiased: $V(\bx)=\Ex_{\xi}[\hat{V}(\bx,\xi)]$ for all $\bx\in\scrX$;
\item $\scrL_{f}(\xi)$-Lipschitz for almost every $\xi\in\Xi$: $\norm{\hat{V}(\bx',\xi)-\hat{V}(\bx,\xi)}\leq \scrL_{f}(\xi)\norm{\bx'-\bx}$ for all $\bx,\bx'\in\scrX$.
The random variable $\scrL_{f}(\xi)$ is positive and integrable with $\Ex_{\xi}[\scrL(\xi)]=L_{f}$. 
\end{enumerate}
\end{assumption} 
To retrieve in-play information about the value of the implicit loss function $h_{i}(\cdot)$, we employ a smoothing-based approach, which necessitates 
defining another sampling mechanism. We follow the gradient sampling strategy of
\cite{FlaKalMcM05}, though alternative random estimation strategies are
certainly possible (see, e.g.  \cite{Duvocelle:2022aa,Berahas:2022vu,Kozak:2022aa}). Specifically, given $\delta>0$, we denote the finite difference
approximation of the directional derivative of $h_{i}$ in direction $\bw_{i}\in\R^{n_{i}}$ as
$$
\nabla_{(\bw_{i},\delta)} h_{i}(\bx_{i})\, \eqdef\, \frac{h_{i}(\bx_{i}+\delta\bw_{i})-h_{i}(\bx)}{\delta}.
$$
Let $\bW_{i}$ be a random vector uniformly distributed on the unit sphere $\sphere_{i}\eqdef\{\bx_{i}\in\R^{n_{i}}\vert\norm{\bx_{i}}= 1\}$.\footnote{See Appendix \ref{sec:sampling} for the explicit construction of such an oracle.} We then define the random vector $H_{i,\bx_{i}}^{\delta}(\bW_{i})$ as a randomized and suitably rescaled version of the finite difference approximator reading as 
\begin{equation}\label{eq:estimator}
 H_{i,\bx_{i}}^{\delta}(\bW_{i})\eqdef n_{i}\bW_{i} \nabla_{(\bW_{i},\delta)} h_{i}(\bx_{i})\in\R^{n_{i}}.
\end{equation}
From eq. \eqref{eq:gradsmoothing} in App.~\ref{app:smoothing}, we know that $H^{\delta_{i}}_{i,\bx_{i}}(\cdot)$ is an unbiased estimator of the gradient of the smoothed function 
$$
h^{\delta}_{i}(\bx_{i})\eqdef \frac{1}{\Vol_{n}(\delta\ball_{n_{i}})} \int_{\delta \ball_{n_{i}}}h_{i}(\bx_{i}+\bf{u})\dd\bf{u},
$$
where $\ball_{d}\eqdef\{\bx\in\R^{d}\vert\norm{\bx}\leq 1\}$ for any dimension $d\geq 1$. Furthermore, we discuss in Appendix \ref{app:smoothing} that the function $h^{\delta}_{i}$ is continuously differentiable with gradient 
\[
\nabla h^{\delta}_{i}(\bx_{i})=\frac{n_{i}}{\delta}\Ex_{\bW_{i}\sim\Uniform(\sphere_{n_{i}})}[\bW_{i}\left( h_{i}(\bx_{i}+\delta\bW_{i})-h_{i}(\bx_{i})\right)]=\Ex_{\bW_{i}\sim\Uniform(\sphere_{n_{i}})}[H^{\delta}_{i,\bx_{i}}(\bW_{i})],
\]
and 
$$
\norm{\nabla h_{i}^{\delta}(\bx_{i})-\nabla h_{i}^{\delta}(\by_{i})}\leq \frac{L_{h_{i}}n_{i}}{\delta}\norm{\bx_{i}-\by_{i}},\qquad\forall \bx_{i},\by_{i}\in\R^{n_{i}},\delta>0
$$
where  $\bW_{i}\sim\Uniform(\sphere_{n_{i}})$ indicates that $\bW_i$ is uniformly distributed on the surface of a unit sphere $\sphere_{n_i}$.

\subsection{Iterative Regularization methods}
Lemma \ref{lem:VI-hierarchical} shows that the equilibria of our hierarchical game are entirely captured by the solution set of problem $\MVI(V,h)$. This is a rich class of
variational problems for which the number of
contributions is so numerous that we just point the reader to the monographs \cite{FacPan03} and \cite{Ryu:2022aa}. Deducing convergence results on the
\emph{last iterate} for standard algorithmic schemes is an important requirement for
game-theoretic learning algorithms, but typically is a rare commodity: Despite
some special classes of games \citep{Azizian:2021aa,Giannou:2021aa}, first-order
methods give only guarantees on a suitably constructed ergodic average. To
obtain last iterate convergence results, we develop an iterativeTikhonov regularization approach. This leads us to consider the regularized problem $\MVI(V^{\eta},h)$, which requires to find
$\bs(\eta)\in\scrX$ satisfying 
\begin{equation}\label{eq:MVI-Tik}
\inner{V^{\eta}(\bs(\eta)),\bx-\bs(\eta)}+h(\bx)-h(\bs(\eta))\geq 0\qquad\forall \bx\in\scrX. 
\end{equation}
Naturally, we would like to understand the nature of the accumulation points of the sequence $\{\bs_{t}\}_{t\in\N}$, where $\bs_{t}\equiv\bs(\eta_{t})$ and $\eta_{t}\downarrow 0$. This sequence can be studied in quite some detail, and we summarize some well-known facts in Proposition \ref{prop:Tikhonov} below. As those results are rather scattered in the literature, we provide a self-contained proof in Appendix \ref{app:Tik}. 
\begin{proposition}\label{prop:Tikhonov}
    Let Assumptions \ref{ass:LLunique}, \ref{ass:standing} and \ref{ass:implicit} hold true. Consider the problem $\MVI(V,h)$ with nonempty solution set $\SOL(V,h)$. Then the following apply:
\begin{enumerate}
\item[(a)] For all $\eta>0$, the set $\SOL(V^{\eta},h)$ is a singleton with unique element denoted by $\bs(\eta)$; 
\item[(b)] $(\forall\eta>0): \norm{\bs(\eta)}\leq \inf\{\norm{\bx}:\bx\in\SOL(V,h)\}$;
\item[(c)] Let $\{\eta_{t}\}_{t\in\N}$ be a positive sequence with $\eta_{t}\downarrow 0$. Then, the sequence $\{\bs(\eta_{t})\}_{t\in\N}$ converges to the least norm solution $\arg\min\{\norm{\bx}:\bx\in\SOL(V,h)\}$;
\item[(d)] For any positive sequence $\{\eta_{t}\}_{t\in\N}$ satisfying $\eta_{t}\downarrow 0$, we have 
\begin{equation}\label{eq:TikEvolv}
\left(\frac{\eta_{t}-\eta_{t-1}}{\eta_{t}}\right)\inf_{\bx\in\SOL(V,h)}\norm{\bx}\geq\norm{\bs(\eta_{t})-\bs(\eta_{t-1})}.
\end{equation}
\end{enumerate}
\end{proposition}

\subsection{The Algorithm} 
Our algorithm for the hierarchical game
setting consists of a double-loop structure: The outer loop allows the
$N$  players to make multiple independent queries of the stochastic oracle
$\hat{V}_1(\cdot,\xi), \ldots, \hat{V}_N(\cdot,\xi),$ and draw multiple independent samples from the surface of a unit sphere, allowing for the simulation of the random estimator \eqref{eq:estimator}. However, since the multiple calls are a negative entry
on the oracle complexity of the method, we impose some control on the number of
mini-batches to be constructed by the agents. Within the inner-loop subroutine, the agents only receive single samples from their stochastic oracles, and employ this new information in an extragradient-type algorithm. We give a precise construction in the following paragraphs.

\subsubsection{The outer loop}
Let $t=0,\ldots,T-1$ be the iteration counter for the outer loop. We denote by $b_{t}\in\N$ the pre-defined sample rate defining the number of random variables each player is allowed to generate in round $t$. Specifically, each player generates an iid sample  $\xi^{1:b_{t}}_{i,t}\eqdef \{\xi_{i,t}^{(s)};1\leq s\leq b_{t}\}$ and constructs the mini-batch estimator
\begin{equation}\label{eq:Vregestimator}
\bar{V}^{t}_{i}\eqdef\frac{1}{b_{t}}\sum_{s=1}^{b_{t}}\hat{V}_{i}(\bx^{t},\xi_{i,t}^{(s)}),
\end{equation}
Let $\bar{V}^{t}=(\bar{V}_{1}^{t},\ldots,\bar{V}_{N}^{t})$. A similar 
assumption is made to obtain point-estimators of the gradient of the implicit
function $h_{i}(\cdot)$. Hence, 
$\bW^{1:b_{t}}_{i,t}\eqdef\{\bW^{(s)}_{i,t};1\leq s\leq b_{t}\}$ denotes an i.i.d.
sample of $b_{t}$ random vectors drawn uniformly at random from $\sphere_{i}$
and define the mini-batch estimator $H_{i,\bx_{i}}^{\delta_{t},b_{t}}$ as 
\begin{equation}\label{eq:batchestimator}
H_{i,\bx_{i}}^{\delta_{t},b_{t}}\eqdef \frac{1}{b_{t}}\sum_{s=1}^{b_{t}}H_{i,\bx_{i}}^{\delta_{t}}(\bW_{i,t}^{(s)}),\quad H^{\delta_{t},b_{t}}_{\bx^{t}}\eqdef(H^{\delta_{t},b_{t}}_{1,\bx^{t}_{1}},\ldots,H^{\delta_{t},b_{t}}_{N,\bx^{t}_{N}}),
\end{equation}
where $\delta_t$ denotes a positive smoothing parameter. Equipped with these estimators, each player enters the procedure $\SFBF(\bx^{t},\bar{V}^{t},H^{\delta_{t},b_{t}}_{\bx^{t}},\gamma_{t},\eta_{t},\delta_{t},K)$, that relies on 
steplength $\gamma_{t}$ and regularization parameter $\eta_{t}$, whose role is explained in the description of the inner loop.

\subsubsection{The inner loop}
 Given the inputs $(\bx^{t},\bar{V}^{t},H^{\delta_{t},b_{t}}_{\bx_{t}},\gamma_{t},\eta_{t})$ prepared in the outer loop, the inner loop of our method is based on a stochastic version of Tseng's modified extragradient method \cite{Tse00}, using one-shot estimators of the relevant data. To be precise, given the current iterate $\bx^{t}$, each player $i$ produces a trajectory $\{\bz^{(t)}_{i,k}\}_{k\in\{0,1/2,1,\ldots,K\}}$. These interim strategy profiles are updated recursively by the procedure $\SFBF(\bx^{t},\bar{V}^{t},H^{\delta_{t},b_{t}}_{\bx^{t}},\gamma_{t},\eta_{t},\delta_{t},K)$ described in Algorithm \ref{alg:SFBF}. Starting with the strategy profile $\bx^{t}$, we choose the initial conditions $\bz^{(t)}_{i,0}=\bx_{i}^{t}$ for all $i\in \scrI$. Then, for each $k\in\{0,1,\ldots,K-1\}$ each player queries the stochastic oracle to obtain the feedback signal
 \begin{equation}\label{eq:hatV}
 \hat{V}^{\eta_{t}}_{i,t,k+1/2}(\bz^{(t)}_{k+1/2})\eqdef\hat{V}_{i}(\bz^{(t)}_{k+1/2},\xi_{i,t,k+1/2})+\eta_{t}\bz_{i,k+1/2}^{(t)}.
 \end{equation}
 Similarly, each player obtains the random information $H^{\delta_{t}}_{\bz^{(t)}_{i,k+1/2}}(\bW_{i,t,k+1/2})$ and $H^{\delta_{t}}_{\bx^{t}_{i}}(\bW_{i,t,k+1/2})$, as defined in \eqref{eq:estimator}. These random variables are used to generate the parallel updates 
 \begin{align*}
& \bz^{(t)}_{i,k+1/2}=\Pi_{\scrX_{i}}[\bz^{(t)}_{i,k}-\gamma_{t}(\bar{V}^{t}_{i}+\eta_{t}\bx_{i}^{t}+H^{\delta_{t},b_{t}}_{\bx_{i}^{t}})],\text{ and }\\
&\bz^{(t)}_{i,k+1}=\bz^{(t)}_{i,k+1/2}-\gamma_{t}\left(\hat{V}^{\eta_{t}}_{i,t,k+1/2}(\bz^{(t)}_{k+1/2})+H^{\delta_{t}}_{\bz^{(t)}_{i,k+1/2}}(\bW_{i,t,k+1/2})-\hat{V}^{\eta_{t}}_{i,t,k+1/2}(\bx^{t})-H^{\delta_{t}}_{\bx^{t}_{i}}(\bW_{i,t,k+1/2})\right)
\end{align*}
for all $i\in\scrI$. These iterations correspond to a stochastic approximation variant of Tseng's forward-backward-forward method \cite{BotMerStaVu21} for solving the time-varying stochastic variational inequality 
$$
0\in V(\bar{\bx})+\nabla h^{\delta_{t}}(\bar{\bx})+\eta_{t}\bar{\bx}+\NC_{\scrX}(\bar{\bx}),
$$
characterized by Tikhonov regularization and smoothing of the implicit function $h_{i}$. 

\begin{algorithm}[t!]
\SetAlgoLined
\caption{$\SFBF(\bar{\bx},\bar{\bv},\bar{H},\gamma,\eta,\delta,K)$}
 \label{alg:SFBF}
\KwResult{Iterate $\bz^{K}$}
Set $\bz^{0}=\bar{\bx}$\;

\For{$k=0,1,\ldots,K-1$}{
  Update $\bz_{k+1/2}=\Pi_{\scrX}[\bz_{k}-\gamma(\bar{\bv}+\eta\bar{\bx}+\bar{H})]$\;
  
 Obtain $\hat{V}^{\eta}_{k+1/2}(\bz_{k+1/2})$ and $\hat{V}^{\eta}_{k+1/2}(\bar{\bx})$ as defined in eq. \eqref{eq:hatV}\;

Draw iid direction vectors $\bW_{k+1/2}=\{\bW_{i,k+1/2}\}_{i\in\scrI}$, with each $\bW_{i,k+1/2}\sim\Uniform(\sphere_{i})$. \; 
  
  Obtain $H^{\delta}_{\bz_{k+1/2}}(\bW_{k+1/2})$ and $H^{\delta}_{\bar{\bx}}(\bW_{k+1/2})$\;
  
  Update 
  $$
  \bz_{k+1}=\bz_{k+1/2}-\gamma\left(\hat{V}^{\eta}_{k+1/2}(\bz_{k+1/2})+H^{\delta}_{\bz_{k+1/2}}(\bW_{k+1/2})-\hat{V}^{\eta}_{k+1/2}(\bar{\bx})-H^{\delta}_{\bar{\bx}}(\bW_{k+1/2})\right).
  $$
  }
\end{algorithm}
\begin{algorithm}[t!]
\SetAlgoLined
\caption{Variance Reduced Hierarchical Game Solver  ($\VRHGS$) }
 \label{alg:VRHGS}
\KwData{$\bx,T,\{\gamma_{t}\}_{t=0}^{T},\{b_{t}\}_{t=0}^{T}$}
Set $\bx^{0}=\bx$.\\
\For{$t=0,1,\ldots,T-1$}{
For each $i\in\scrI$ receive the oracle feedback $\bar{V}^{t}$ defined by $\bar{V}^{t}_{i}\eqdef\frac{1}{b_{t}}\sum_{s=1}^{b_{t}}\hat{V}_{i}(\bx^{t},\xi_{i,t}^{(s)})$.
\; 

For each $i\in\scrI$ construct the estimator $H^{\delta_{t},b_{t}}_{\bx^{t}}$ defined by $H_{i,\bx_{i}}^{\delta_{t},b_{t}}\eqdef \frac{1}{b_{t}}\sum_{s=1}^{b_{t}}H_{i,\bx_{i}}^{\delta_{t}}(\bW_{i,t}^{(s)})$.
\;

Update $\bx^{t+1}=\SFBF(\bx^{t},\bar{V}^{t},H^{\delta_{t},b_{t}}_{\bx^{t}},\gamma_{t},\eta_{t},\delta_{t},K)$.
  }
\end{algorithm}
\paragraph{Discussion}
The key innovation of the scheme $\VRHGS$ lies in the combination of smoothing (to allow for hierarchy), regularization (to contend with ill-posedness), and variance-reduction (to mitigate bias) within a stochastic forward-backward-forward framework. Our double-loop solution strategy mimics the computational architecture of SVRG, which takes a full gradient sample of the finite sum problem ``once in a while", while performing frequent single-sample updates in between. Our method, adapted to general probability spaces, proceeds similarly: The ``shadow sequence" $\bz^{(t)}_{i,k+1/2}$ uses costly mini-batch estimators computed in the outer loop; these are maintained in memory while executing the inner loop (i.e. only ``once in a while" updated). The additional forward steps to obtain the iterates $\bz^{(t)}_{i,k+1}$ make use of fresh one-shot estimators of the payoff gradient and the finite difference estimator. All steps are overlaid by a Tikhonov regularization, while smoothing facilitates accommodation with hierarchical objectives. From a computational perspective, our scheme performs a single projection onto the leaders' feasible set $\scrX_{i}$. This can save considerably on computational time in cases where the projection operator is costly to evaluate, and constitutes a major difference compared to viable alternative algorithmic schemes like the extragradient or optimistic mirror descent. Hence, our method reduces the sample-complexity of recent mini-batch variance reduction techniques for stochastic VIs \citep{BotMerStaVu21}, while concomitantly reducing the computational bottlenecks of double-call algorithms \citep{IusJofOliTho17,KanSha19} by lifting one projection step. Finally, similar rate and complexity statements emerge when allowing for inexact generalizations that allow for $\epsilon$-approximate solutions of lower-level problem. 

%% file: Analysis.tex
%

In this section we state the main results on the asymptotic convergence of scheme $\VRHGS$. All technical and lengthy proofs are collected in Section \ref{sec:proofs}.

In the inner and outer loops of $\VRHGS$, we have two sources of randomness at
each iteration: (i) the sequence of  mini-batches
$\xi^{1:b_{t}}_{t}\eqdef\{\xi^{1:b_{t}}_{i,t}\}_{i\in\scrI}$ and
$\bW^{1:b_{t}}_{t}\eqdef\{\bW_{i,t}^{1:b_{t}}\}_{i\in\scrI}$, which are used to
perform the opening forward-backward step in Algorithm \ref{alg:SFBF}; (ii) the sequences $\xi_{t,k+1/2}=\{\xi_{i,t,k+1/2}\}_{i\in\scrI}$ and
$\bW_{t,k+1/2}=\{\bW_{i,t,k+1/2}\}_{i\in\scrI}$ for $k\in\{0,1,\ldots,K-1\}$,
which are employed in constructing the iterate $\bz^{(t)}_{k+1}$ in Algorithm \ref{alg:SFBF} during the outer epoch $t$.
To keep track of the
information structure of the outer and inner loops, we introduce the filtrations
$\scrF_{t}\eqdef \sigma(\bx^{0},\ldots,\bx^{t})$ for $0\leq t\leq T$, and
$\scrA_{t,0}\eqdef \sigma(\bx^{t},\xi_{t}^{1:b_{t}},\bW_{t}^{1:b_{t}})$ as
well as
$\scrA_{t,k}\eqdef\sigma(\scrA_{t,0}\cup\sigma(\xi_{t,1/2},\bW_{t,1/2},\ldots,\xi_{t,k-1/2},\bW_{t,k-1/2}))$
for $k\in\{1,2,\ldots,K-1\}$. By construction, the iterates $\bz^{(t)}_{k}$ and
$\bz^{(t)}_{k+1/2}$ are both $\scrA_{t,k}$-measurable.

\subsection{Error Structure of the estimators}
We impose a uniform variance bound on the random vector field $\hat{V}$ over the set $\scrX$. Compactness of $\scrX$ implies that such an assumption comes without loss of generality, and the proof of the variance bound in Lemma \ref{lem:Variance} is simple to obtain and thus omitted; See \cite{BotMerStaVu21}.
\begin{lemma}
\label{lem:Variance}
There exists $M_{V}>0$ such that $\Ex_{\xi}[\norm{\hat{V}(\bx,\xi)-V(\bx)}^{2}]\leq M^{2}_{V}$ for all $\bx\in\scrX$. Additionally, let $b_{t}\geq 1$ and $\xi^{1:b_{t}}_{t}=\{\{\xi^{(1)}_{i,t}\}_{i\in\scrI},\ldots,\{\xi^{(b_{t})}_{i,t}\}_{i\in\scrI}\}$ denote an i.i.d sample of the random variable $\xi$. Then, for $\eps_{V}(\xi^{1:b_{t}}_{t})\eqdef \frac{1}{b_{t}}\sum_{s=1}^{b_{t}}\hat{V}(\bx,\xi^{(s)}_{t})-V(\bx)$, we have 
$
\sqrt{\Ex\left[\norm{\eps_{V}(\xi^{1:b_{t}}_{t})}^{2}\right]}\leq \frac{M_{V}}{\sqrt{b_{t}}}.
$
\end{lemma}
Concerning the estimator of the gradient of the smoothed lower level function $h^{\delta}(\cdot)$, we can report the following bounds, which are derived in Appendix \ref{sec:sampling} and proved in Lemma \ref{lem:boundH}.
\begin{lemma}\label{lem:boundHindividual}
 Let Assumptions \ref{ass:LLunique}-\ref{ass:standing} hold. Define $\ce_{\bx_{i}}(\bW_{i}^{1:b})\eqdef H^{\delta,b}_{i,\bx_{i}}-\nabla h^{\delta}_{i}(\bx_{i})$, where $\{\bW^{(s)}_{i}\}_{s=1}^{b_{t}}$ is an i.i.d sample drawn uniformly from the unit sphere $\sphere_{i}$, i.e. $\bW^{1:b}_{i}\sim\Uniform(\sphere_{i})^{\otimes b}$. Then for all $i\in\scrI$,  
\begin{itemize}
\item[(a)] $\Ex_{\bW^{1:b}_{i}\sim\Uniform(\sphere_{i})^{\otimes b}}[\ce_{\bx_{i}}(\bW^{1:b}_{i})]=0$;
\item[(b)] $\norm{H_{i,\bx_{i}}^{\delta}(\bw_{i})}^{2}\leq L_{h_{i}}^{2}n^{2}_{i}$ for all $\bw_{i}\in\sphere_{i}$;
\item[(c)] $\Ex_{\bW^{1:b}_{i}\sim\Uniform(\sphere_{i})^{\otimes b}}[\norm{\ce_{\bx_{i}}(\bW^{1:b}_{i})}^{2}]\leq \frac{n^{2}_{i}L^{2}_{h_{i}}}{b_{i}}.$
\end{itemize}
\end{lemma}

\subsection{Almost sure convergence of the last iterate}

Our analysis of $\VRHGS$ relies on the following energy inequality, proved in Section \ref{app:energy}. 
\begin{lemma}\label{lem:energy}
    Let Assumptions \ref{ass:LLunique},\ref{ass:standing},\ref{ass:implicit} and \ref{ass:SO_V} hold true. Let $\{\bx^{t}\}_{t=0}^{T-1}$ be generated by $\VRHGS$ and denote by $\{\bz^{(t)}_{k}\}_{k\in\{0,1/2,\ldots,K\}}$ the sequence obtained by executing $\SFBF(\bx^{t},\bar{V}^{t},H^{\delta_{t},b_{t}}_{\bx^{t}},\gamma_{t},\eta_{t},\delta_{t},K)$. Set $L_{h}\eqdef\sum_{i\in\scrI}L_{h_{i}}$.
    Then for all $\bx\in\scrX$  and $k\in\{0,1,\ldots,K-1\}$, we have 
\begin{align*}
\norm{\bz^{(t)}_{k+1}-\bx}^{2}&\leq(1-\gamma_{t}\eta_{t})\norm{\bz^{(t)}_{k}-\bx}^{2}-(1-2\gamma_{t}\eta_{t})\norm{\bz^{(t)}_{k+1/2}-\bz^{(t)}_{k}}^{2}\\
&+8\gamma_{t}^{2}\sum_{i\in\scrI}L_{h_{i}}^{2}n^{2}_{i}+4\gamma^{2}_{t}(\scrL_{f}(\xi_{t,k+1/2})^{2}+\eta^{2}_{0})\norm{\bz^{(t)}_{k+1/2}-\bx^{t}}^{2}\\
&-2\gamma_{t}\inner{\hat{V}_{t,k+1/2}(\bz^{(t)}_{k+1/2})+H^{\delta_{t}}_{\bz^{(t)}_{k+1/2}}(\bW_{t,k+1/2})-V(\bz^{(t)}_{k+1/2})-\nabla h^{\delta_{t}}(\bz^{(t)}_{k+1/2}),\bz^{(t)}_{k+1/2}-\bx}\\
&-2\gamma_{t}\inner{\left(V(\bx^{t})+\nabla h^{\delta_{t}}(\bx^{t})\right)-\left(\hat{V}_{t,k+1/2}(\bx^{t})+H^{\delta_{t}}_{\bx^{t}}(\bW_{t,k+1/2})\right),\bz^{(t)}_{k+1/2}-\bx}\\
&-2\gamma_{t}\inner{\eps^{t}_{V}(\xi^{1:b_{t}}_{t})+\eps^{t}_{h}(\bW^{1:b_{t}}_{t}),\bz^{(t)}_{k+1/2}-\bx}\\
&-2\gamma_{t}\left(\inner{V^{\eta_{t}}(\bx),\bz^{(t)}_{k+1/2}-\bx}+h(\bz^{(t)}_{k+1/2})-h(\bx)\right)+2\gamma_{t}\delta_{t}L_{h}.
\end{align*}
\end{lemma}
We next prove a.s. convergence of $\{\bx^{t}\}_{t=0}^{T}$ to the least-norm solution of $\MVI(V,h)$ as $T\to\infty$. The proof rests on a fine comparison between the algorithmic sequence $\{\bx^{t}\}$ and the sequence of solutions of the regularized problems $\MVI(V^{\eta_{t}},h)$, denoted as $\{\bs_{t}\}_{t=0}^{T}$. 
\begin{theorem}\label{th:converge}
Let Assumptions \ref{ass:LLunique},\ref{ass:standing},\ref{ass:implicit}, and  \ref{ass:SO_V} hold. Suppose we are given sequences $\{\gamma_{t}\}_{t\in\N},\{\delta_{t}\}_{t\in\N}$ and $\{\eta_{t}\}_{t\in\N}$, satisfying the following conditions:
\begin{itemize}
\item[(a)] $\lim_{t\to\infty}\frac{\gamma_{t}}{\eta_{t}}=\lim_{t\to\infty}\frac{\delta_{t}}{\eta_{t}}=0$, and $\sum_{t=0}^{\infty}\gamma^{2}_{t}<\infty,\sum_{t=0}^{\infty}\gamma_{t}\eta_{t}=\infty$;
\item[(b)] $\gamma_{t}\eta_{t}\in(0,1/2)$ and $\lim_{t\to\infty}\eta_{t}=0$;
\item[(c)] $\sum_{t=0}^{\infty}\left(\frac{\eta_{t}-\eta_{t-1}}{\eta_{t}}\right)^{2}(1+\frac{1}{\gamma_{t}\eta_{t}})<\infty$ and $\lim_{t\to\infty}\left(\frac{\eta_{t}-\eta_{t-1}}{\eta_{t}}\right)^{2}\left(\frac{1+\frac{1}{\gamma_{t}\eta_{t}}}{\gamma_{t}\eta_{t}}\right)=0$.
\end{itemize}
Then $\Pr(\lim_{t\to\infty}\norm{\bx^{t}-\bx^{\ast}}=0)=1$, where $\bx^{\ast}$ denotes the unique least norm solution of $\MVI(V,h)$; i.e.  $\{(\bx^{t},\by(\bx^{t}))\}_{t\in\N}$ converges almost surely to a hierarchical equilibrium of the game \eqref{eq:Opt}. 
\end{theorem}
The proof of this Theorem can be found in Section \ref{app:proofTheoremConvergence}.
\subsection{Finite-time complexity}
The convergence measure usually employed for MVI$(V,h)$ is the gap function 
\begin{equation}\label{eq:gap}
\Gamma(\bx)\eqdef\sup_{\bz\in\scrX}\left(\inner{V(\bz),\bx-\bz}+h(\bx)-h(\bz)\right),
\end{equation}
Since we work in probabilistic setting, naturally our convergence measure will be based on $\Ex[\Gamma(\bx)]$. Our main finite-time iteration complexity result in terms of this performance measure is the next Theorem, whose proof is detailed in Section \ref{app:proofofComplexity}.
\begin{theorem}\label{th:gap}
Let Assumptions \ref{ass:LLunique},\ref{ass:standing},\ref{ass:implicit}, and  \ref{ass:SO_V} hold and fix $T\in\N$. Consider Algorithm $\VRHGS$ with the inputs $\gamma_{t}=\eta_{t}=\delta_{t}=1/T$, as well as $K=T$ and batch size $b_{t}\geq T^{2}$. Then, 
$
\Ex[\Gamma(\bar{\bz}^{T})]=\scrO\left(\frac{C\sigma}{T}\right),
$
where $\sigma\eqdef\sqrt{2M_{V}^{2}+2\sum_{i\in\scrI}L^{2}_{h_{i}}n^{2}_{i}},C=\max_{i\in\scrI}C_{i}$ (cf. Assumption (\ref{ass:standing}.i)), and 
$\bar{\bz}^{T}\eqdef\frac{\sum_{t=0}^{T-1}\gamma_{t}\bar{\bz}^{t}}{\sum_{t=0}^{T-1}\gamma_{t}}$ for $\bar{\bz}^{t}\eqdef\frac{1}{K}\sum_{k=0}^{K-1}\bz^{(t)}_{k+1/2}$.
\end{theorem}
We next evaluate the oracle complexity of $\VRHGS$. To be precise, let $\OC(T,K,\{b_{t}\}_{t=0}^{T-1})$ the number of random variables method $\VRHGS$ generates in the inner and outer loop until we achieve a solution that pushes the expected gap below a target value $\eps$.
\begin{remark}
We point out that this measure of oracle complexity ignores the computational effort arising from solving the lower level problem attached with player $i\in\scrI$. This is consistent as we assume that the solution map is provided to us in terms of an oracle. A full-fledged complexity analysis can be done, and will appear in a future publication. 
\end{remark}
\begin{proposition}\label{prop:OOC}
Let $\eps>0$ be given, and set $T=\lceil1/\eps\rceil$. If we choose the same sequences as in Theorem \ref{th:gap}, we have 
$
\OC(T,K,\{b_{t}\}_{t=0}^{T-1})=\scrO(2N/\eps^{3}).
$
\end{proposition}
\begin{proof}
The number of random variables generated in each inner loop iteration is $2K\times N$. In each round of the outer loop we sample $2b_{t}\times N$ random variables. Hence, the total oracle complexity is $\OC(T,K,\{b_{t}\}_{t=0}^{T})=2KTN+2N\sum_{t=0}^{T-1}b_{t}.$ For the specific values of $T,K,\gamma_{t},\delta_{t},\eta_{t}$ defined in Theorem \ref{th:gap} and $b_{t}=T^{2}$, it clearly follows $\OC(T,K,\{b_{t}\}_{t=0}^{T-1})=\scrO(2N/\eps^{3}).$
\end{proof}

\subsection{Inexact generalization}
A key shortcoming in the implementation of $\VRHGS$ is the need for exact solutions of the lower-level problem. Naturally, when the solution map $\by_{i}(\cdot)$ corresponds to the solution of a large-scale stochastic optimization/VI problem, this claim is hard to justify. In this section, we allow for an inexact solution $\by_{i}^{\eps}(\bx)$ associated with an error level $\eps$, defined as 
\begin{equation}
\Ex[\norm{\by^{\eps}(\bx)-\by(\bx)}\vert\bx]\leq\eps\quad \text{ a.s.}
\end{equation}
Under the inexact lower level solution $\by_{i}^{\eps}$, we let $h_{i}^{\eps}(\bx_{i})\eqdef g_{i}(\bx_{i},\by_{i}^{\eps}(\bx_{i}))$.
\begin{remark}\label{rem:inexactsolver}
We can obtain the inexact solution $\by_{i}^{\eps}(\bx_{i})$ with rather efficient numerical methods. First, we can parallelize the computation since the problems $\VI(\phi_{i}(\bx_{i},\cdot),\scrY_{i})$ are uncoupled. Second, the mapping $\phi_{i}(\bx_{i},\cdot)$ is assumed to be strongly monotone. Hence, we can solve the VI to $\eps$-accuracy with exponential rate using for instance the method in \cite{Cui:2022ab}.
\end{remark}
 As in the exact regime, we assume that player $i$ has access to an oracle with which she can construct a spherical approximation of the gradient of the implicit function $h_{i}^{\eps}$. Hence, for given $\delta>0$, we let 
$
h^{\eps,\delta}_{i}(\bx_{i})\eqdef\int_{\ball_{n_{i}}}h^{\eps}_{i}(\bx_{i}+\delta\bw)\frac{\dd\bw}{\Vol_{n}(\ball_{n_{i}})}.
$
We denote the resulting estimators by $H^{\delta,\eps}_{i,\bx_{i}}(\bW_{i})=n_{i}\bW_{i}\nabla_{(\bW_{i},\delta)}h^{\eps}_{i}(i,\bx_{i})$, and the mini-batch versions $H^{\delta,\eps,b}_{i,\bx_{i}}\eqdef \frac{1}{b}\sum_{s=1}^{b}H^{\delta,\eps}_{i,\bx_{i}}(\bW_{i}^{(s)}).$ With these concepts in hand, we can adapt $\VRHGS$ to run exactly the same way as described in Algorithm \ref{alg:SFBF} and Algorithm \ref{alg:VRHGS}, replacing the appearance of quantities involving $h_{i}$ with its inexact version $h^{\eps}_{i}$; see Section \ref{sec:analysisinexact} for a precise formulation of the method.
\begin{theorem}\label{th:inexact}
Let Assumptions \ref{ass:standing} hold and fix $T\in\N$. Consider Algorithm I-$\VRHGS$, defined in Section \ref{sec:analysisinexact}, with the sequence $\gamma_{t}=\eta_{t}=\delta_{t}=1/T$, as well as $K=T$, the batch size $b_{t}\geq T^{2}$ and inexactness regime $\eps_{t}=1/T^{2}$. Then
$
\Ex[\Gamma(\bar{\bz}^{T})]=\scrO\left(\frac{C\sigma}{T}\right),
$
 where $\sigma\eqdef\sqrt{2M_{V}^{2}+2\sum_{i\in\scrI}L^{2}_{h_{i}}n^{2}_{i}},C=\max_{i\in\scrI}C_{i}$ (cf. Assumption (\ref{ass:standing}.i)).
\end{theorem}

Notably, tractable resolution of the proposed stochastic hierarchical game is possible in inexact regimes and such practically motivated schemes are not adversely affected in terms of either the rate or complexity guarantees.

%% file: application.tex
In this section we present a  model inspired by Hobbs and
Pang~\cite{hobbs07nash}, but suitably modified to account for uncertainty in
prices and costs, multi-period settings with ramping constraints, and the
incorporation of virtual power plants (VPPs)
(see~\cite{NAVAL2021111393,DOE_VPPLiftOff_2023} for a review of VPPs and power
markets). The model we present below is at this stage an academic example that
demonstrates the modelling power of our hierarchical games approach. In future
studies we aim for numerical implementations of this model. \\

Consider a set of nodes $\Nscr$ of a network and a set of time periods $\Tscr
\, \triangleq \, \left\{\, 1, 2, \cdots, T\, \right\}$. A generation
firm is indexed by $f$, where $f$ belongs to the finite set $\Fscr$ and each
firm is assumed to have an associated VPP.  At a node $i$ in the network, a firm $f$ may
generate $g_{f,i,t}$ units via conventional generation in period $t$ and sell
$s_{f,i,t}$ units  during the same period. In addition, at time period $t$, firm
$f$ may generate $P^{\rm pv,S}_{f,t}+P^{\rm pv,L}_{f,t}$ units
of power via PV capacity, of which  $P^{\rm pv,S}_{f,t}$ is sold and
$P^{\rm pv,L}_{f,t}$ is employed for meeting load.  The total amount of
power sold at node $i$ during period $t$ by all generating firms is represented by $S_{i,t}$, i.e. $S_{i,t}={\displaystyle \sum_{f\in\Fscr}}s_{f,i,t}$.
If the nodal power price at the $i$th node during period $t$ is a random
function given by $p_{i,t}(\bullet,\xi)$, where $p_{i,t}(\bullet,\xi)$ is a
decreasing function of aggregate nodal sales $S_{i,t}$ for any $\xi \in \Xi$. It follows that firm
$f$'s revenue from non-PV power sales at node $i$ during period $t$  under realization $\xi$ is
$p_{i,t}(S_{i,t},\xi) s_{f,i}$. Sales of PV output by firm $f$ at time $t$ is priced
using a function $p^{\rm R}_{f,t}$, earning a revenue given by $p_{f,t}^{\rm
R}\left( P_{f,t}^{\rm pv,S} + P_{f,t}^{\rm pv,S,V} \right)
P^{\rm pv,S}$, where $P_{f,t}^{\rm pv,S,V}$ denotes the sales of
firm $f$'s associated VPP (whose problem is described later in this section).
We observe that renewable power is priced using this price function, distinct
from conventional sources, and is designed to provide incentives for renewable
expansion~\cite{rui2023assessing}.  The costs incurred by firm $f$ at node $i$
during period $t$ are given by the sum of the cost of generating $g_{f,i,t}$
and the cost of transmitting the excess $(s_{f,i,t} -g_{f,i,t})$.  Let the
random cost function of generation associated with firm $f$ at node $i$ be
given by $c_{fi}(\bullet,\zeta)$ while the cost of transmitting power from an
arbitrary node (referred to as the hub) to node $i$ is given by $w_i$. The
constraint set incorporates a balance between aggregate sales, aggregate
generation, and power injection into the VPP  at all nodes for every time
period $t$. In addition, we impose nonnegativity bound on sales and generation
at any time period $t$, enforce a  capacity limit on generation levels, and
introduce ramping constraints on the change in generation levels.   The
resulting problem faced by generating firm $f$, denoted by ({\bf Firm$_f$}),
requires minimizing generation cost  less revenue from conventional and PV
sales  by optimizing sales $s_{f,i,t}$ and generation $g_{f,i,t}$ at every
node $i$ and every time period $t$ as well as load-directed PV output $P^{\rm
pv,L}_{f,t}$ and PV sales $P^{\rm pv,S}_{f,t}$ at time $t$. If $P_{f,t}^{\rm
pv,S,V,\epsilon}(\cdot)$ denotes a component of the single-valued solution map of
the $\epsilon$-regularized problem of the VPP associated with firm $f$, denoted
by  ({\bf VPP}$_f(P_f^{\rm pv,S})$), then firm $f$'s problem is defined as
follows, where $\widehat{\Tscr} = \{1, \cdots, T-1\}$.
\[ 
\begin{array}{ll} 
\displaystyle{
    {\operatornamewithlimits{ \mbox{\bf maximize}}_{s_{f,i,t}, \, g_{f,i,t}, \, P^{\rm pv,S}_{f,t}\, P^{\rm pv,L}_{f,t}}}
} & \mathbb{E}\left[\displaystyle{
    \sum_{t = 1}^T \sum_{i\in \Nscr}
} \, \left( \, p_{i,t} (S_{i,t}, \xi)s_{f,i,t} - c_{fi}(g_{fi,t},\zeta)
        -(s_{fi,t}-g_{fi,t})w_{i,t} \, \right)\right]\\
    &  + {\displaystyle \sum_{t \in \Tscr}}\left( p_{f,t}^{\rm R}\left( P_{f,t}^{\rm pv,S} + P_{f,t}^{\rm pv,S,V,\epsilon}(P_{f,t}^{\rm pv,S}) \right) P^{\rm pv,S} \right)   \\ [0.2in]
\mbox{\bf subject to} &
\left\{ \begin{array}{lll}
0 & \leq & g_{f,i,t} \, \leq \, {\rm cap}_{fi} \\ [5pt]
0 & \leq & s_{f,i,t} \end{array} \right\}, \hspace{0.9in} \forall \, t \in \Tscr, \quad \forall i \in \Nscr \hspace{0.7in} \left( \mbox{\bf Firm}_f(s_{-f},g_{-f})\right)\\ [0.2in]
    &  -\mbox{RR}^{\rm down}_{f,i} \, \leq \, g_{f,i,t} - g_{f,i,t-1} \, \leq \, \mbox{RR}^{\rm up}_{fi}, \quad \forall \, t \in \widehat{\Tscr}, \quad \forall i \in \Nscr \\ [0.2in]
  &  P^{\rm pv,L}_{f,t}+ P^{\rm pv,S}_{f,t} \, \leq \, {\rm cap}^{\rm pv}_{f,t}, \ \qquad \qquad \qquad \forall t \in \Tscr \\ 
 & 0 \, \leq \,   P^{\rm pv,L}_{f,t}, P^{\rm pv,S}_{f,t}, \qquad \qquad \qquad \qquad \quad \forall t \in \Tscr \\[5pt]
\mbox{and} & \displaystyle{
\sum_{i \in \Nscr}
    } \, ( \, s_{f,i,t} - g_{f,i,t} \, ) - P^{\rm pv,L}_{f,t} \, = \, 0.  \ \ \quad \qquad \forall t \in\ \Tscr 
\end{array}
\]
It bears reminding that the last set of constraints specified in $\left( \mbox{\bf Firm}_f(s_{-f},g_{-f})\right)$ are parametrized by rival decisions and can be relaxed with Lagrange multiplier
$\lambda_{f,t}$, leading to the following {\em relaxed} problem $\left(
\mbox{\bf Firm}^{\rm rel}_f(s_{-f},g_{-f})\right)$, defined as follows.
\[ 
\begin{array}{ll} 
\displaystyle{
    {\operatornamewithlimits{ \mbox{\bf maximize}}_{s_{f,i,t}, \, g_{f,i,t}, \, P^{\rm pv,S}_{f,t}\, P^{\rm pv,L}_{f,t}}}
} & \mathbb{E}\left[\displaystyle{
    \sum_{t = 1}^T \sum_{i\in \Nscr}
} \, \left( \, p_{i,t} (S_{i,t}, \xi)s_{f,i,t} - c_{f,i}(g_{f,i,t},\zeta)
        -(s_{f,i,t}-g_{f,i,t})w_{i,t} \, \right)\right]  \,  \\
    &\hspace{-1.2in}  + {\displaystyle \sum_{t \in \Tscr}}\left( p_{f,t}^{\rm R}\left( P_{f,t}^{\rm pv,S} + P_{f,t}^{\rm pv,S,V} \right) P_{f,t}^{\rm pv,S} \right)   - {\displaystyle \sum_{t=1}^T} \lambda_{f,t}^\top\left({\displaystyle \sum_{i \in \Nscr} }\left( (\, s_{f,i,t} - g_{f,i,t} \, ) - P_{f,t}^{\rm pv,L}\right)\right)  \\ [0.2in]
      \mbox{\bf subject to} &
\left\{ \begin{array}{lll}
0 & \leq & g_{fi,t} \, \leq \, {\rm cap}_{fi} \\ [5pt]
0 & \leq & s_{fi,t} \end{array} \right\}, \hspace{0.9in} \forall \, t \in \Tscr, \quad \forall i \in \Nscr \hspace{0.7in} \left( \mbox{\bf Firm}_f(s_{-f},g_{-f})\right)\\ [0.2in]
    &  -\mbox{RR}^{\rm down}_{f,i} \, \leq \, g_{f,i,t} - g_{f,i,t-1} \, \leq \, \mbox{RR}^{\rm up}_{f,i}, \quad \forall \, t \in \widehat{\Tscr}, \quad \forall i \in \Nscr \\ [0.2in]
  &  P^{\rm pv,L}_{f,t}+ P^{\rm pv,S}_{f,t} \, \leq \, {\rm cap}^{\rm pv}_{f,t}, \  \qquad \qquad \qquad\forall t \in \Tscr \\ 
 & 0 \, \leq \,   P^{\rm pv,L}_{f,t}, P^{\rm pv,S}_{f,t}, \qquad \qquad \qquad \qquad \quad \forall t \in \Tscr \end{array}
\]
In addition, we introduce a pricing player $\left(\mbox{\bf Price}\left(s_{f}, g_{f}, P_{f,t}^{\rm pv,L}\right)\right)$ corresponding to the determination of $\lambda_{f,t}$ for $f \in \Fscr$ and  $t \, \in \, \Tscr$, defined as follows.
\[
\begin{array}{ll} 
\displaystyle{
    {\operatornamewithlimits{ \mbox{\bf minimize}}_{\lambda}}
} &  {\displaystyle \sum_{f \in \Fscr} \sum_{t \in \Tscr}} \lambda_{f,t}^\top\left({\displaystyle \sum_{i \in \Nscr}
    } \,  ( \, s_{f,i,t} - g_{f,i,t} \, ) + P_{f,t}^{\rm pv,L}\right). \hspace{1in} \left(\mbox{\bf Price}\left(s, g, P^{\rm pv,L}\right)\right)
\end{array}
\]
Note that, the generating firm sees the transmission fee $w_{i,t}$ and the rival
firms' sales $s_{-fi,t} \equiv \{s_{hi,t}~:~ h \neq f\}$ as exogenous parameters to
its optimization problem even though they are endogenous to the overall
equilibrium model as we will see shortly.  The ISO sees the transmission fees
$w=(w_{i,t})_{i \in \Nscr, t \in \Tscr}$ as exogenous and prescribes flows $y = (y_{i,t})_{i \in
\Nscr, t \in \Tscr}$ as per a solution of the following linear program 
\[ \begin{array}{ll} %
\displaystyle{
    {\operatornamewithlimits{\mbox{\bf maximize}}_{y}} 
} & \displaystyle{
    \sum_{i \in \Nscr} \sum_{t \in \Tscr}
} \, y_{i,t} w_{i,t} \\ [0.2in]
\mbox{\bf subject to} & \displaystyle{
\sum_{i \in \Nscr}
    } \, {\rm PDF}_{ij} y_{i,t} \, \leq \, \hat{T}_j, \qquad \forall \, j \in \Kscr, \ \forall \, t \in \Tscr,\hspace{1.3in} \left( \mbox{\bf ISO}(w)\right)
\end{array} \]
where $\Kscr$ is the set of all arcs or links in the network with node set $\Nscr$, $\hat{T}_j$ denotes the transmission capacity of link $j$, $y_{i,t}$ represents the transfer of power (in MW) by the system operator from a hub node to node node $i$ and  PDF$_{ij}$
denotes the power transfer distribution factor,
which specifies the MW flow through link $j$ as a consequence of
unit MW injection at an arbitrary hub node and a unit withdrawal at node $i$.
Finally,  to clear the market, the transmission flows $y_i$ must must balance the net sales at each node, as specified next.
\begin{align}\label{eq:mc_power}
    y_{i,t} \, = \, \displaystyle{
\sum_{f \in \Fscr}
} \, \left( \, s_{f,i,t} - g_{f,i,t} \, \right), \qquad \forall \, i \in \Nscr, \quad \forall \ t \in \Tscr.
\end{align}
    In fact, this constraint can be recast as a collection of pricing players, denoted by ({\bf Flow}$^{\rm price}(g,s,y)$).
\[ \begin{array}{ll} %
\displaystyle{
    {\operatornamewithlimits{\mbox{\bf minimize}}_{\beta}} 
} &  
    {\displaystyle \sum_{t \in \Tscr} \sum_{i \in \Nscr}} \beta_{i,t} \left(\, y_{i,t} \, - \, \displaystyle{
\sum_{f \in \Fscr}
    } \, \left( \, s_{f,i,t} - g_{f,i,t} \, \right)\, \right) \hspace{2in} (\mbox{\bf Flow}^{\rm price}(g, s, y)) 
\end{array}
    \]
We now extend the scope of the framework of power markets by incorporating virtual power
    plants.  A virtual power plant (VPP) represents a collections of
    distributed energy resources (DERs) (e.g., batteries, smart thermostats,
    controllable water heaters, and rooftop solar) that can be coordinated to
    enhance the reliability and sustainability of the electric grid. To satisfy
    the short-term goals for clean energy technology (CET) deployment, it has
    been estimated that U.S. VPP capacity must triple by 2030, leading to
    potential savings of \$10 billion in annual grid
    costs~\cite{DOE_VPPLiftOff_2023}.   Without loss of generality, we assume that any firm $f \, \in \, \Fscr$
    has a collection of components, which collectively provide ``virtual
    power'' in addition to conventional generation. Before proceeding, we
    model three components in such a VPP, akin to approaches employed in~\cite{jin2017foresee,garifi18stochastic,castillo2019stochastic}.

\medskip

\noindent (a) {\em Battery storage.} Suppose the storage unit associated with firm $f$ has an associated state of charge (SOC) level at time $t$ by SOC$_{f,t}$.
\begin{align}
    \mbox{SOC}_{f,t+1} \, = \, \mbox{SOC}_{f,t}\, + \, \tfrac{\eta^{\rm b,ch}_{f}\Delta t}{Q^{\mathrm{b}}_f}P^{\rm b, ch}_{f,t} \, - \, \tfrac{\Delta t}{\eta^{\rm b,ds}_{f}Q^{\mathrm{b}}_f}P^{\rm b, ds}_{f,t}, \ \forall \, t  \in  \Tscr \label{cons1} 
\end{align}
where $P^{\rm b,ch}_{f,t}$ and  $P^{\rm b,ds}_{f,t}$ represent charging and
    discharging power-levels at time $t$, $\eta^{\rm b,ch}_{f}$ and $\eta^{\rm
    b,ds}_{f}$ represent charging and discharging efficiencies at time $t$,
    while $Q^{\rm b}_f$ and $\Delta t$ denote the battery capacity and time
    interval, respectively. In addition, SOC$_{f,t}$ is bounded between a
    minimum value $\mbox{SOC}^{\rm \min}_f$ and maximum value $\mbox{SOC}^{\rm \max}_f$ while at any time $t$, charging and discharging rates cannot 
    exceed $P^{\rm b,ch,mx}_f$ and $P^{\rm b,ds,mx}_f$, respectively, as captured by the following bounds. 
\begin{align}
    \mbox{SOC}^{\rm \min}_f \, & \leq \, \mbox{SOC}_{f,t} \, \leq \, \mbox{SOC}^{\rm \max}_f, \qquad \forall t  \in  \Tscr  \label{cons2} \\
    0 \, & \le \, \tfrac{P^{\rm b,ch}_{f,t}}{P^{\rm b,ch,mx}_f}\, \le \, 1,\qquad \qquad \quad \forall t  \in  \Tscr  \label{cons3} \\  
    0 \, & \le \, \tfrac{P^{\rm b,ch}_{f,t}}{P^{\rm b,ds,mx}_f} \,  \, \le \, 1. \qquad \qquad \quad \forall t  \in  \Tscr   \label{cons4}
\end{align}

\medskip

    \noindent (b) {\em Intermittent resources.} We now model intermittency by considering a photovoltaic (PV) array associated with firm $f$, where at time $t$, $P^{\rm pv, L,V}_{f,t}$ and $P^{\rm pv,S,V}_{f,t}$ denote the PV output employed for meeting load and for deriving sales revenue, respectively. Further, $P^{\rm
pv,max}_{f}$ represents maximum PV power at time $t$. Consequently,  PV output is modeled as
\begin{align}
    P^{\rm pv,L,V}_{f,t} + P^{\rm pv,S,V}_{f,t} \, & = \, (1-U^{\rm pv}_{f}) P^{\rm pv,max}_{f,t} E_{f,t}\qquad  \qquad \forall t  \in  \Tscr  \label{cons5} \\
    0 \, & \leq \, U_{f,t}^{\rm pv}\, \leq \, 1, \qquad \qquad \qquad \qquad  \, \forall  t  \in  \Tscr  \label{cons6}\\
    0 \, & \leq \, P^{\rm pv, L,V}, P^{\rm pv,S,V} \qquad   \, \forall  t  \in  \Tscr \label{cons6b} 
\end{align}
    where $U^{\rm pv}_{f,t}$ denotes the PV curtailment employed by firm $f$ at time $t$ while $P^{\rm pv,max}_{f,t}$ scales with the solar irradiance at time $t$ as seen by firm $f$, denoted by $E_{f,t}$. We observe that $P^{\rm pv,S,V}_{f,t}, P^{\rm pv,L,V}_{f,t} \, \geq \, 0$ for any $f \, \in \, \Fscr$ and any $t \, \in \, \Tscr.$ 

\medskip

\noindent (c) {\em Thermal onsite generation.} Often VPPs may incorporate
onsite thermal generation that can be employed. For any $f \, \in \, \Fscr$,
suppose the generation capacity is denoted by $\mbox{Cap}^{\rm onsite}_f$ while
the upward and downward ramping rates are given by $\mbox{RR}^{\rm up}_{f}$ and
$\mbox{RR}^{\rm down}_f$, respectively. Consequently, if the generation output
at time $t$ is denoted by $P^{\rm onsite}_{f,t}$, then for any $t \, \in \, \Tscr$, we have 
    \begin{align}
        0 \, \leq \, P^{\rm onsite}_{f,t} \, \leq \, \mbox{Cap}^{\rm onsite}_f, \qquad \forall \, t \, \in \, \Tscr \label{cons7}
    \end{align}
    Furthermore, changes in generation level are bounded by ramping rates, as captured by the following set of two-sided constraints. 
    \begin{align}
        - \mbox{RR}^{\rm down}_f \, \leq \, P^{\rm onsite}_{f,t+1} - P^{\rm onsite}_{f,t+1} \, \leq \, \mbox{RR}^{\rm up}_f, \quad   
        \forall \, t \, \in \, \Tscr \label{cons8}
    \end{align}

 VPPs are characterized by an idiosyncratic load profile that cannot be
 controlled; specifically,  $P^{\rm L,V}_{f,t}$ denotes the load associated
 with VPP $f$ at time $t$. In more comprehensive models, we may incorporate
 HVAC and water heater components that allow for more fine-grained control of
 such loads but for purposes of simplicity, we omit such a discussion here. In
 the current setting,  the effective load
 emerging from managing the VPP associated with firm $f$ and time $t$ is
 given by the sum of the uncontrollable load and the battery load (charging
 less discharging level) less the sum of onsite generation and load-directed  PV output is required to be nonpositive, as specified next.   
\begin{align}
    P^{\rm L,V}_{f,t} + \left(\, P^{\rm b,ch}_{f,t} - P^{\rm b,ds}_{f,t}\, \right)  - P_{f,t}^{\rm onsite}- P^{\rm pv, L,V}_{f,t} \, \le \, 0,\qquad \forall \, t \, \in \, \Tscr \label{cons9} 
\end{align}
Note that the satisfaction of this constraint relies on appropriate sizing of the battery capacity $Q_f^b$ and the onsite generation capacity $\mbox{Cap}_f^{\rm onsite}$. Suppose the decision vector of firm $f$'s VPP is denoted by $y^{\rm vpp}_f$, defined as 
$$y_f^{\rm vpp} \, = \, \left( \, \mbox{SOC}_{f}; P^{\rm b, ch}_{f}; P^{\rm b, ds}_{f};P^{\rm pv,L,V}_{f}; P^{\rm pv,S,V}_{f}; U_{f}^{\rm pv}; P^{\rm onsite}_f \, \right).$$ 
The profit function associated with firm $f$'s VPP is the revenue obtained by sales revenue derived from PV sales less the VPP's operational cost  (given by the sum of the costs of onsite generation and the (converted) cost of PV curtailment), defined as
\begin{align} \notag
    \mbox{r}^{\rm vpp}_f (y^{\rm vpp}_f;P^{\rm pv,S}_f) \,  \triangleq \, \displaystyle{ \sum_{t \in \Tscr }} &  \, \left(\, \underbrace{p_{f,t}^{\rm R} \left(P^{\rm pv,S,V}_{f,t}+P^{\rm pv,S}_{f,t}\right) P^{\rm pv, S,V}_{f,t}}_{\tiny \mbox{VPP revenue from PV sales}}   -  \underbrace{c_{f}^{\rm onsite}(P^{\rm onsite}_{f,t})}_{\tiny \mbox{Cost of onsite gen.}} -  \underbrace{\beta  U^{\rm pv}_{f,t} P^{\rm max}_{f,t}}_{\tiny \mbox{Env. cost of PV curtailment}} \, \right), 
\end{align}
where $p_{f,t}^{\rm R}(\cdot)$ denotes the price function  of renewables seen at firm $f$ at time $t$, while the revenue obtained is given by $p_{f,t}^{\rm R} \left(P^{\rm pv, S,V}_{f,t}+P^{\rm pv,S}_{f,t}\right) P^{\rm pv,S,V}_{f,t}$.  We may then formally define the optimization problem faced by the VPP associated with firm $f$, where the polyhedral constraints are captured by $\left\{ \, y^{\rm vpp}_f \, \mid \, A_f y^{\rm vpp}_f  \le d_f \, \right\}$ where $A_f \in \Real^{m \times n}$ and $d_f \in \Real^m$. 
\[ \begin{array}{ll} %
\displaystyle{
    {\operatornamewithlimits{\mbox{\bf maximize}}_{y^{\rm vpp}_{f}}} 
    } & r^{\rm vpp}_f (y^{\rm vpp}_f; P_f^{\rm pv, S}) \hspace{2.3in} \left( \mbox{\bf VPP}_f(P_f^{\rm pv,S})\right)\\ [0.2in]
    \mbox{\bf subject to} & \eqref{cons1} \, - \, \eqref{cons9} \equiv \left\{ \, y^{\rm vpp}_f \, \mid \, A_f y^{\rm vpp}_f  \le d_f \, \right\} . 
\end{array} \]
We now observe that the resulting equilibrium problem comprises of a collection
of firms, each of which has a single follower as captured by a VPP, in addition
to the ISO and a set of players that determine prices. To facilitate analysis
of the necessary and sufficient equilibrium conditions of this hierarchical
game, we approximate $\left( \mbox{\bf VPP}_f(P_f^{\rm pv,S})\right)$
by employing a smooth (exact) penalized approximation; this latter formulation
is of particular relevance in deriving the concavity of the function
$p_{f,t}^{\rm R}\left( P_{f,t}^{\rm pv,S} + P_{f,t}^{\rm
pv,S,V,\epsilon} (P_{f,t}^{\rm pv,S}) \right) P_{f,t}^{\rm
pv,S}$ in $P_{f,t}^{\rm pv,S}$, where $P_{f,t}^{\rm
pv,S,V,\epsilon} (P_{f,t}^{\rm pv,S})$ is a component of the
single-valued solution map  $y^{\rm vpp,\epsilon}_f(P_{f,t}^{\rm
pv,S})$, a solution of the $\epsilon$-regularized and the $\epsilon$-smoothed
(exact) penalized approximation of $\left( \mbox{\bf VPP}_f(P_f^{\rm pv,S
})\right)$. To this end, we define the
exact penalty function $\varphi$ 
and its smoothed counterpart $\varphi_{\epsilon}$ 
as  
\begin{align}
    \varphi(A_fy^{\rm vpp}_f  - d_f) \, &\triangleq \, \sum_{i=1}^m  \max\{a_{fi}^\top y^{\rm vpp}_f  - d_{fi},0\}, \quad
    \varphi_{\epsilon}(A_fy^{\rm vpp}_f  - d_f) \, \triangleq \, \sum_{i=1}^m  \psi_{\epsilon}(a_{fi}^\top y^{\rm vpp}_f  - d_{fi}), \\
    \mbox{ and } \psi_{\epsilon}(t) & \triangleq \begin{cases} 0, & \mbox{ if } t \, \le \, 0\\
                            \tfrac{\mu t^2}{2\epsilon}, & \mbox{ if } 0 \, \le \, t \, \le \, \epsilon \\
                            \mu(t-\tfrac{\epsilon}{2}). & \mbox{ if } t \, \ge \, \epsilon
                        \end{cases}
\end{align}
%
This penalty function and its smoothed counterpart are employed in formally defining the exact penalty reformulation of (VPP$_f$) and its smoothed counterpart. 

\begin{align}\notag
    & \ \left\{ \begin{aligned} 
        \max_{y^{\rm vpp}_f} & \ r^{\rm vpp}(y^{\rm vpp}_f, P_f^{\rm pv,S}) \\
        \mbox{subject to} & \ A y^{\rm vpp}_f  \, \leq \, d_f  
    \end{aligned} \right\}   \ \equiv \, 
\left\{ \begin{aligned} 
    \max_{y^{\rm vpp}_f} & \ \left(r^{\rm vpp}(y^{\rm vpp}_f, P_f^{\rm pv,S}) + \mu \varphi(Ay^{\rm vpp}_f - d_f)\right)
    \end{aligned} \right\} \\
    & \hspace{0.2in}\ \approx \, \tag{\bf VPP$_f^\epsilon(P^{\rm pv,S}_f)$}
\left\{ \begin{aligned} 
    \max_{y^{\rm vpp}_f} & \left(r^{\rm vpp}(y^{\rm vpp}_f, P_f^{\rm pv,S}) + \mu \varphi_{\epsilon}(Ay^{\rm vpp}_f  - d_f) + \tfrac{\epsilon}{2} \| y^{\rm vpp}_f\|^2\right). 
    \end{aligned} \right\} 
\end{align}
%
%

    We observe that the resulting game can be viewed as a noncooperative
hierarchical game, defined by upper-level player problems given by firm players
$\left( \mbox{\bf Firm}^{\rm rel}_f(s_{-f},g_{-f},\lambda_f)\right)_{f \, \in
\, \Fscr}$, pricing players $\left(\mbox{\bf Price}\left(s, g, P^{\rm
pv,firm-load}\right)\right)$, the ISO as denoted by ({\bf ISO}), and the
transmission pricing player $(\mbox{\bf Flow}^{\rm price}(g, s, y))$. In
addition, the set of regularized lower-level VPP problems is given by
$\left(\mbox{\bf VPP$_f^\epsilon(P^{\rm pv,S}_f)$} \right)_{f \in
\Fscr}$. We succinctly represent this noncooperative game as an $N+2$ player
game, in which the first $N$  players correspond to firm $f$'s problem for $f
\in \Fscr$ while the last two correspond to pricing players. 
\begin{align} &
    \notag
\left\{\begin{aligned}
        \min_{\bz^1 \in \bZ^1} \, \left(h_1(\bz^1, \bz^{-1})+ g_1(\bz^1,\bu^1(\bz^1))\right)
\end{aligned} \right\} \quad
    \cdots\quad
\left\{\begin{aligned}
        \min_{\bz^{N} \in \bZ^N} \, \left(h_{N}(\bz^{N}, \bz^{-N})+ g_{N}(\bz^{N},\bu^{N}(\bz^{N}))\right)
\end{aligned}\right\},\\
    & \tag{Hier-game$^{\rm vpp}$}
\left\{\begin{aligned}
    \min_{\bz^{N+1} \in \bZ^{N+1}} \, \left(h_{N+1}(\bz^{N+1}, \bz^{-(N+1)})\right)
\end{aligned}\right\} \quad
\left\{\begin{aligned}
    \min_{\bz^{N+2} \in \bZ^{N+2}} \, \left(h_{N+2}(\bz^{N+2}, \bz^{-(N+2)})\right)
\end{aligned}\right\}.
\end{align}
Note that the first $N$ players' objectives are characterized by hierarchical
terms; specifically, the hierarchical terms  $g_j(\bz^j,\bu^j(\bz^j))$ for any
$j \, \in \{1,\ldots, N\}$  correspond to the hierarchical terms in firm $f$'s
problem given by ${\displaystyle \sum_{t \in \Tscr}}\left( p_{f,t}^{\rm
R}\left( P_{f,t}^{\rm pv,S} + P_{f,t}^{\rm pv,S,V} \right)
P_{f,t}^{\rm pv,S} \right)$ for any $f \, \in \, \Fscr$.

\medskip
 While convexity of player problems follows in a straightforward fashion from the definition of firm problems and suitable convexity requirements on the cost functions as well as affineness requirements on the price functions. Single-valuedness of the solution map $\bu^j(\bullet)$ follows from the observation that the regularized VPP profit function is strongly concave. Additionally, convexity of $g_j(\bullet,\bu^j(\bullet))$ is a consequence of analogous analysis for Stackelberg leadership (cf.~\cite{xu2005mpcc,DeMiguel:2009aa}). Finally,  monotonicity of $F$ can be derived in a fashion similar to that considered in~\cite{hobbs01linear}. Existence of an equilibrium can then be derived in a fashion similar to that employed in ~\cite{DeMiguel:2009aa}. A comprehensive analysis of this model is left to future work.

%% file: Appendix_Analysis.tex
%
The proof on the finite time-complexity estimate starts by a technical derivation of an energy-type inequality that gives us an upper bound on the change of the energy function $\frac{1}{2}\norm{\bz^{(t)}_{k+1}-\bx}^{2}$, computed within an arbitrary inner loop evaluation, and for an arbitrary anchor point $\bx\in\scrX$. Via a sequence of tedious, but otherwise straightforward, manipulations we arrive out our first main result, Lemma \ref{lem:energy}. From there, we proceed as in the standard analysis of stochastic approximation schemes \citep{NJLS09}, and derive an upper bound on the gap function of the mixed variational inequality. 

\subsubsection{Proof of Lemma \ref{lem:energy}}
\label{app:energy}
To simplify notation we omit the dependence on the outer iteration loop $t$, and thus simply write $\bz_{k}$ for $\bz^{(t)}_{k}$. The same notational simplification will be used in all variables that are computed within the inner loop executed in the $t$-round of the outer loop procedure. With the hope that the reader agrees that this reduces notational complexity a bit, we proceed to derive the the postulated energy inequality. To start, we observe that for each $\bx\in\scrX$ we have
\begin{align*}
\norm{\bz_{k+1}-\bx}^{2}&=\norm{\bz_{k+1}-\bz_{k+1/2}+\bz_{k+1/2}-\bz_{k}+\bz_{k}-\bx}^{2}\\
&=\norm{\bz_{k+1}-\bz_{k+1/2}}^{2}-\norm{\bz_{k+1/2}-\bz_{k}}^{2}+\norm{\bz_{k}-\bx}^{2}+2\inner{\bz_{k+1}-\bz_{k+1/2},\bz_{k+1/2}-\bx}\\
&+2\inner{\bz_{k+1/2}-\bz_{k},\bz_{k+1/2}-\bx}\\
&=\norm{\bz_{k}-\bx}^{2}-\norm{\bz_{k+1/2}-\bz_{k}}^{2}\\
&+\norm{\gamma_{t}\left(\hat{V}^{\eta_{t}}_{k+1/2}(\bz_{k+1/2})-\hat{V}^{\eta_{t}}_{k+1/2}(\bx^{t})+H^{\delta_{t}}_{\bz_{k+1/2}}(\bW_{k+1/2})-H^{\delta_{t}}_{\bx^{t}}(\bW_{k+1/2})\right)}^{2}\\
&-2\gamma_{t}\inner{\hat{V}^{\eta_{t}}_{k+1/2}(\bz_{k+1/2})-\hat{V}^{\eta_{t}}_{k+1/2}(\bx^{t})+H^{\delta_{t}}_{\bz_{k+1/2}}(\bW_{k+1/2})-H^{\delta_{t}}_{\bx^{t}}(\bW_{k+1/2}),\bz_{k+1/2}-\bx}\\
&+2\inner{\bz_{k+1/2}-\left(\bz_{k}-\gamma_{t}(\bar{V}^{t}+H^{\delta_{t},b_{t}}_{\bx^{t}}+\eta_{t}\bx^{t})\right),\bz_{k+1/2}-\bx}\\
&-2\gamma_{t}\inner{\bar{V}^{t}+H^{\delta_{t},b_{t}}_{\bx^{t}}+\eta_{t}\bx^{t},\bz_{k+1/2}-\bx}
\end{align*}
Lemma \ref{lem:projector}(i) gives 
$
2\inner{\bz_{k+1/2}-\left(\bz_{k}-\gamma_{t}(\bar{V}^{t}+H^{\delta_{t},b_{t}}_{\bx^{t}}+\eta_{t}\bx^{t}\right),\bz_{k+1/2}-\bx}\leq 0.
$
Additionally, Assumption \ref{ass:SO_V} and Lemma \ref{lem:boundH}(b) gives
\begin{align*}
&\norm{\gamma_{t}\left(\hat{V}^{\eta_{t}}_{k+1/2}(\bz_{k+1/2})-\hat{V}^{\eta_{t}}_{k+1/2}(\bx^{t})+H^{\delta_{t}}_{\bz_{k+1/2}}(\bW_{k+1/2})-H^{\delta_{t}}_{\bx^{t}}(\bW_{k+1/2})\right)}^{2}\\
&\leq 2\gamma^{2}_{t}\norm{\hat{V}^{\eta_{t}}_{k+1/2}(\bz_{k+1/2})-\hat{V}^{\eta_{t}}_{k+1/2}(\bx^{t})}^{2}+2\gamma^{2}_{t}\norm{H^{\delta_{t}}_{\bz_{k+1/2}}(\bW_{k+1/2})-H^{\delta_{t}}_{\bx^{t}}(\bW_{k+1/2})}^{2}\\
&\leq 4\gamma^{2}_{t}\norm{\hat{V}_{k+1/2}(\bz_{k+1/2})-\hat{V}_{k+1/2}(\bx^{t})}^{2}+4\eta^{2}_{t}\gamma^{2}_{t}\norm{\bz_{k+1/2}-\bx^{t}}^{2}\\
&+4\gamma^{2}_{t}\left(\norm{H^{\delta_{t}}_{\bz_{k+1/2}}(\bW_{k+1/2})}^{2}+\norm{H^{\delta_{t}}_{\bx^{t}}(\bW_{k+1/2})}^{2}\right)\\
&\leq 4\gamma^{2}_{t}\left(\scrL_{f}(\xi_{k+1/2})^{2}+\eta^{2}_{t}\right)\norm{\bz_{k+1/2}-\bx^{t}}^{2}+8\gamma^{2}_{t}\sum_{i\in\scrI}L_{h_{i}}^{2}n_{i}^{2}.
\end{align*}
It follows 
\begin{align*}
\norm{\bz_{k+1}-\bx}^{2}&\leq \norm{\bz_{k}-\bx}^{2}-\norm{\bz_{k+1/2}-\bz_{k}}^{2}+8\gamma_{t}^{2}\sum_{i\in\scrI}L_{h_{i}}^{2}n^{2}_{i}\\
&+4\gamma^{2}_{t}(\scrL_{f}(\xi_{k+1/2})^{2}+\eta^{2}_{0})\norm{\bz_{k+1/2}-\bx^{t}}^{2}\\
&-2\gamma_{t}\inner{\hat{V}^{\eta_{t}}_{k+1/2}(\bz_{k+1/2})+H^{\delta_{t}}_{\bz_{k+1/2}}(\bW_{k+1/2})-\hat{V}^{\eta_{t}}_{k+1/2}(\bx^{t})-H^{\delta_{t}}_{\bx^{t}}(\bW_{k+1/2}),z_{k+1/2}-\bx}\\
&-2\gamma_{t}\inner{\bar{V}^{t}+\eta_{t}\bx^{t}+H^{\delta_{t},b_{t}}_{\bx^{t}},\bz_{k+1/2}-\bx}\\
&=\norm{\bz_{k}-\bx}^{2}-\norm{\bz_{k+1/2}-\bz_{k}}^{2}+8\gamma_{t}^{2}\sum_{i\in\scrI}L_{h_{i}}^{2}n^{2}_{i}+4\gamma^{2}_{t}(\scrL_{f}(\xi_{k+1/2})^{2}+\eta^{2}_{0})\norm{\bz_{k+1/2}-\bx^{t}}^{2}\\
&-2\gamma_{t}\inner{\hat{V}_{k+1/2}(\bz_{k+1/2})+H^{\delta_{t}}_{\bz_{k+1/2}}(\bW_{k+1/2})-V(\bz_{k+1/2})-\nabla h^{\delta_{t}}(\bz_{k+1/2}),\bz_{k+1/2}-\bx}\\
&+2\gamma_{t}\inner{\hat{V}_{k+1/2}(\bx^{t})+H^{\delta_{t}}_{\bx^{t}}(\bW_{k+1/2})-\bar{V}^{t}-H^{\delta_{t},b_{t}}_{\bx^{t}},\bz_{k+1/2}-\bx}\\
&-2\gamma_{t}\inner{V(\bz_{k+1/2})+\nabla h^{\delta_{t}}(\bz_{k+1/2})+\eta_{t}\bz_{k+1/2},\bz_{k+1/2}-\bx}\\
&=\norm{\bz_{k}-\bx}^{2}-\norm{\bz_{k+1/2}-\bz_{k}}^{2}+8\gamma_{t}^{2}\sum_{i\in\scrI}L_{h_{i}}^{2}n^{2}_{i}+4\gamma^{2}_{t}(\scrL_{f}(\xi_{k+1/2})^{2}+\eta^{2}_{0})\norm{\bz_{k+1/2}-\bx^{t}}^{2}\\
&-2\gamma_{t}\inner{\hat{V}_{k+1/2}(\bz_{k+1/2})+H^{\delta_{t}}_{\bz_{k+1/2}}(\bW_{k+1/2})-V(\bz_{k+1/2})-\nabla h^{\delta_{t}}(\bz_{k+1/2}),\bz_{k+1/2}-\bx}\\
&+2\gamma_{t}\inner{\hat{V}_{k+1/2}(\bx^{t})+H^{\delta_{t}}_{\bx^{t}}(\bW_{k+1/2})-\bar{V}^{t}-H^{\delta_{t},b_{t}}_{\bx^{t}},\bz_{k+1/2}-\bx}\\
&-2\gamma_{t}\inner{V^{\eta_{t}}(\bz_{k+1/2})+\nabla h^{\delta_{t}}(\bz_{k+1/2})-V^{\eta_{t}}(\bx)-\nabla h^{\delta_{t}}(\bx),\bz_{k+1/2}-\bx}\\
&-2\gamma_{t}\inner{V^{\eta_{t}}(\bx)+\nabla h^{\delta_{t}}(\bx),\bz_{k+1/2}-\bx}
\end{align*}
Since the operator $\bx\mapsto V^{\eta_{t}}(\bx)+\nabla h^{\delta_{t}}(\bx)$ is $\eta_{t}$-strongly monotone, we can further bound the expression above as 
\begin{align*}
\norm{\bz_{k+1}-\bx}^{2}&\leq\norm{\bz_{k}-\bx}^{2}-\norm{\bz_{k+1/2}-\bz_{k}}^{2}+8\gamma_{t}^{2}\sum_{i\in\scrI}L_{h_{i}}^{2}n^{2}_{i}+4\gamma^{2}_{t}(\scrL_{f}(\xi_{k+1/2})^{2}+\eta^{2}_{0})\norm{\bz_{k+1/2}-\bx^{t}}^{2}\\
&-2\gamma_{t}\inner{\hat{V}_{k+1/2}(\bz_{k+1/2})+H^{\delta_{t}}_{\bz_{k+1/2}}(\bW_{k+1/2})-V(\bz_{k+1/2})-\nabla h^{\delta_{t}}(\bz_{k+1/2}),\bz_{k+1/2}-\bx}\\
&+2\gamma_{t}\inner{\hat{V}_{k+1/2}(\bx^{t})+H^{\delta_{t}}_{\bx^{t}}(\bW_{k+1/2})-\bar{V}^{t}-H^{\delta_{t},b_{t}}_{\bx^{t}},\bz_{k+1/2}-\bx}\\
&-2\gamma_{t}\eta_{t}\norm{\bz_{k+1/2}-\bx}^{2}\\
&-2\gamma_{t}\inner{V^{\eta_{t}}(\bx)+\nabla h^{\delta_{t}}(\bx),\bz_{k+1/2}-\bx}
\end{align*}
Next, we split the mini-batch estimator $\bar{V}^{t}$ into its mean component and its error component so that 
\begin{equation}\label{eq:errV}
\bar{V}^{t}=V(\bx^{t})+\eps^{t}_{V}(\xi^{1:b_{t}}_{t}).
\end{equation}
Similarly, we write 
\begin{equation}\label{eq:errH}
\eps^{t}_{h}(\bW^{1:b_{t}}_{t})\eqdef H^{\delta_{t},b_{t}}_{\bx^{t}}-\nabla h^{\delta_{t}}(\bx^{t}).
\end{equation}
Using these error terms, we may further bound the right hand side of the penultimate display as 
\begin{align*}
\norm{\bz_{k+1}-\bx}^{2}&\leq\norm{\bz_{k}-\bx}^{2}-\norm{\bz_{k+1/2}-\bz_{k}}^{2}+8\gamma_{t}^{2}\sum_{i\in\scrI}L_{h_{i}}^{2}n^{2}_{i}+4\gamma^{2}_{t}(\scrL_{f}(\xi_{k+1/2})^{2}+\eta^{2}_{0})\norm{\bz_{k+1/2}-\bx^{t}}^{2}\\
&-2\gamma_{t}\inner{\hat{V}_{k+1/2}(\bz_{k+1/2})+H^{\delta_{t}}_{\bz_{k+1/2}}(\bW_{k+1/2})-V(\bz_{k+1/2})-\nabla h^{\delta_{t}}(\bz_{k+1/2}),\bz_{k+1/2}-\bx}\\
&-2\gamma_{t}\inner{\left(V(\bx^{t})+\nabla h^{\delta_{t}}(\bx^{t})\right)-\left(\hat{V}_{k+1/2}(\bx^{t})+H^{\delta_{t}}_{\bx^{t}}(\bW_{k+1/2})\right),\bz_{k+1/2}-\bx}\\
&-2\gamma_{t}\eta_{t}\norm{\bz_{k+1/2}-\bx}^{2}-2\gamma_{t}\inner{\eps^{t}_{V}(\xi^{1:b_{t}}_{t})+\eps^{t}_{h}(\bW^{1:b_{t}}_{t}),\bz_{k+1/2}-\bx}\\
&-2\gamma_{t}\inner{V^{\eta_{t}}(\bx)+\nabla h^{\delta_{t}}(\bx),\bz_{k+1/2}-\bx}
\end{align*}
A simple application of the triangle inequality shows
$$
-2\gamma_{t}\eta_{t}\norm{\bz_{k+1/2}-\bx}^{2}\leq 2\gamma_{t}\eta_{t}\norm{\bz_{k+1/2}-\bz_{k}}^{2}-\gamma_{t}\eta_{t}\norm{\bz_{k}-\bx}^{2}.
$$
Using this bound, we continue with the derivations above to arrive at 
\begin{align*}
\norm{\bz_{k+1}-\bx}^{2}&\leq(1-\gamma_{t}\eta_{t})\norm{\bz_{k}-\bx}^{2}-(1-2\gamma_{t}\eta_{t})\norm{\bz_{k+1/2}-\bz_{k}}^{2}\\
&+8\gamma_{t}^{2}\sum_{i\in\scrI}L_{h_{i}}^{2}n^{2}_{i}+4\gamma^{2}_{t}(\scrL_{f}(\xi_{k+1/2})^{2}+\eta^{2}_{0})\norm{\bz_{k+1/2}-\bx^{t}}^{2}\\
&-2\gamma_{t}\inner{\hat{V}_{k+1/2}(\bz_{k+1/2})+H^{\delta_{t}}_{\bz_{k+1/2}}(\bW_{k+1/2})-V(\bz_{k+1/2})-\nabla h^{\delta_{t}}(\bz_{k+1/2}),\bz_{k+1/2}-\bx}\\
&-2\gamma_{t}\inner{\left(V(\bx^{t})+\nabla h^{\delta_{t}}(\bx^{t})\right)-\left(\hat{V}_{k+1/2}(\bx^{t})+H^{\delta_{t}}_{\bx^{t}}(\bW_{k+1/2})\right),\bz_{k+1/2}-\bx}\\
&-2\gamma_{t}\inner{\eps^{t}_{V}(\xi^{1:b_{t}}_{t})+\eps^{t}_{h}(\bW^{1:b_{t}}_{t}),\bz_{k+1/2}-\bx}\\
&-2\gamma_{t}\inner{V^{\eta_{t}}(\bx)+\nabla h^{\delta_{t}}(\bx),\bz_{k+1/2}-\bx}
\end{align*}
By convexity of the application $\bx\mapsto h^{\delta_{t}}(\bx)$, we have 
$$
h^{\delta_{t}}(\bz_{k+1/2})\geq h^{\delta_{t}}(\bx)+\inner{\nabla h^{\delta_{t}}(\bx),\bz_{k+1/2}-\bx}.
$$
Hence, the penultimate display turns into 
\begin{align*}
\norm{\bz_{k+1}-\bx}^{2}&\leq(1-\gamma_{t}\eta_{t})\norm{\bz_{k}-\bx}^{2}-(1-2\gamma_{t}\eta_{t})\norm{\bz_{k+1/2}-\bz_{k}}^{2}\\
&+8\gamma_{t}^{2}\sum_{i\in\scrI}L_{h_{i}}^{2}n^{2}_{i}+4\gamma^{2}_{t}(\scrL_{f}(\xi_{k+1/2})^{2}+\eta^{2}_{0})\norm{\bz_{k+1/2}-\bx^{t}}^{2}\\
&-2\gamma_{t}\inner{\hat{V}_{k+1/2}(\bz_{k+1/2})+H^{\delta_{t}}_{\bz_{k+1/2}}(\bW_{k+1/2})-V(\bz_{k+1/2})-\nabla h^{\delta_{t}}(\bz_{k+1/2}),\bz_{k+1/2}-\bx}\\
&-2\gamma_{t}\inner{\left(V(\bx^{t})+\nabla h^{\delta_{t}}(\bx^{t})\right)-\left(\hat{V}_{k+1/2}(\bx^{t})+H^{\delta_{t}}_{\bx^{t}}(\bW_{k+1/2})\right),\bz_{k+1/2}-\bx}\\
&-2\gamma_{t}\inner{\eps^{t}_{V}(\xi^{1:b_{t}}_{t})+\eps^{t}_{h}(\bW^{1:b_{t}}_{t}),\bz_{k+1/2}-\bx}\\
&-2\gamma_{t}\left(\inner{V^{\eta_{t}}(\bx),\bz_{k+1/2}-\bx}+h^{\delta_{t}}(\bz_{k+1/2})-h^{\delta_{t}}(\bx)\right).
\end{align*}
From Lemma 2 in \cite{YousNedSha10}, we know for $L_{h}\eqdef\sum_{i\in\scrI}L_{h_{i}}$ that
$$
h^{\delta_{t}}(\bz_{k+1/2})-h^{\delta_{t}}(\bx)\geq h(\bz_{k+1/2})-h(\bx)-\delta_{t}L_{h}.
$$
Therefore, 
\begin{align*}
\norm{\bz_{k+1}-\bx}^{2}&\leq(1-\gamma_{t}\eta_{t})\norm{\bz_{k}-\bx}^{2}-(1-2\gamma_{t}\eta_{t})\norm{\bz_{k+1/2}-\bz_{k}}^{2}\\
&+8\gamma_{t}^{2}\sum_{i\in\scrI}L_{h_{i}}^{2}n^{2}_{i}+4\gamma^{2}_{t}(\scrL_{f}(\xi_{k+1/2})^{2}+\eta^{2}_{0})\norm{\bz_{k+1/2}-\bx^{t}}^{2}\\
&-2\gamma_{t}\inner{\hat{V}_{k+1/2}(\bz_{k+1/2})+H^{\delta_{t}}_{\bz_{k+1/2}}(\bW_{k+1/2})-V(\bz_{k+1/2})-\nabla h^{\delta_{t}}(\bz_{k+1/2}),\bz_{k+1/2}-\bx}\\
&-2\gamma_{t}\inner{\left(V(\bx^{t})+\nabla h^{\delta_{t}}(\bx^{t})\right)-\left(\hat{V}_{k+1/2}(\bx^{t})+H^{\delta_{t}}_{\bx^{t}}(\bW_{k+1/2})\right),\bz_{k+1/2}-\bx}\\
&-2\gamma_{t}\inner{\eps^{t}_{V}(\xi^{1:b_{t}}_{t})+\eps^{t}_{h}(\bW^{1:b_{t}}_{t}),\bz_{k+1/2}-\bx}\\
&-2\gamma_{t}\left(\inner{V^{\eta_{t}}(\bx),\bz_{k+1/2}-\bx}+h(\bz_{k+1/2})-h(\bx)\right)+2\gamma_{t}\delta_{t}L_{h}
\end{align*}
which is what has been claimed. 

\subsubsection{Proof of Theorem \ref{th:converge}}\label{app:proofTheoremConvergence}
Let $\{\eta_{t}\}_{t\in\N}$ be a positive sequence with $\eta_{t}\downarrow 0$. Let $\bs_{t}\equiv\bs(\eta_{t})$ denote the corresponding sequence of solutions to $\MVI(V^{\eta_{t}},h)$. Set $q_{t}\equiv 1-\gamma_{t}\eta_{t}\in(0,1/2)$. Then, iterating the energy inequality established in Lemma \ref{lem:energy}, we have for $\bx=\bs_{t}$:
\begin{align*}
&\norm{\bz^{(t)}_{K}-\bs_{t}}^{2}\leq q_{t}^{K}\norm{\bz^{(t)}_{0}-\bs_{t}}^{2}\\
&+\sum_{k=0}^{K-1}q_{t}^{K-k+1}\gamma_{t}^{2}\left( 8\left(\sum_{i\in\scrI}L_{h_{i}}^{2}n^{2}_{i}\right)+4(\scrL_{f}(\xi_{t,k+1/2})^{2}+\eta^{2}_{0})\norm{\bz^{(t)}_{k+1/2}-\bx^{t}}^{2}\right)\\
&-2\gamma_{t}\sum_{k=0}^{K-1}q_{t}^{K-k+1}\inner{\hat{V}_{t,k+1/2}(\bz^{(t)}_{k+1/2})+H^{\delta_{t}}_{\bz^{(t)}_{k+1/2}}(\bW_{t,k+1/2})-V(\bz^{(t)}_{k+1/2})-\nabla h^{\delta_{t}}(\bz^{(t)}_{k+1/2}),\bz^{(t)}_{k+1/2}-\bs_{t}}\\
&-2\gamma_{t}\sum_{k=0}^{K-1}q_{t}^{K-k+1}\inner{\left(V(\bx^{t})+\nabla h^{\delta_{t}}(\bx^{t})\right)-\left(\hat{V}_{t,k+1/2}(\bx^{t})+H^{\delta_{t}}_{\bx^{t}}(\bW_{t,k+1/2})\right),\bz^{(t)}_{k+1/2}-\bs_{t}}\\
&-2\gamma_{t}\sum_{k=0}^{K-1}q_{t}^{K-k+1}\inner{\eps^{t}_{V}(\xi^{1:b_{t}}_{t})+\eps^{t}_{h}(\bW^{1:b_{t}}_{t}),\bz^{(t)}_{k+1/2}-\bs_{t}}\\
&-2\gamma_{t}\sum_{k=0}^{K-1}q^{K-k+1}_{t}\left(\inner{V^{\eta_{t}}(\bs_{t}),\bz^{(t)}_{k+1/2}-\bs_{t}}+h(\bz^{(t)}_{k+1/2})-h(\bs_{t})-2\delta_{t}L_{h}\right).
\end{align*}
By definition of the point $\bs_{t}$, we have $\inner{V^{\eta_{t}}(\bs_{t}),\bz^{(t)}_{k+1/2}-\bs_{t}}+h(\bz^{(t)}_{k+1/2})-h(\bs_{t})\geq 0$. Furthermore, the estimators involved are unbiased, which means 
\begin{align*}
&\Ex[\inner{\hat{V}_{t,k+1/2}(\bz^{(t)}_{k+1/2})+H^{\delta_{t}}_{\bz^{(t)}_{k+1/2}}(\bW_{t,k+1/2})-V(\bz^{(t)}_{k+1/2})-\nabla h^{\delta_{t}}(\bz^{(t)}_{k+1/2}),\bz^{(t)}_{k+1/2}-\bs_{t}}\vert\scrF_{t}]=0,\\
&\Ex[\left(\inner{V(\bx^{t})+\nabla h^{\delta_{t}}(\bx^{t})\right)-\left(\hat{V}_{t,k+1/2}(\bx^{t})+H^{\delta_{t}}_{\bx^{t}}(\bW_{t,k+1/2})\right),\bz^{(t)}_{k+1/2}-\bs_{t}}\vert\scrF_{t}]=0,\\
&\Ex[\inner{\eps^{t}_{V}(\xi^{1:b_{t}}_{t})+\eps^{t}_{h}(\bW^{1:b_{t}}_{t}),\bz^{(t)}_{k+1/2}-\bs_{t}}\vert\scrF_{t}]=0.
\end{align*}
To wit, let us focus on the first line of the above display. Using the law of iterated expectations, we have 
\begin{align*}
&\Ex[\inner{\hat{V}_{t,k+1/2}(\bz^{(t)}_{k+1/2})+H^{\delta_{t}}_{\bz^{(t)}_{k+1/2}}(\bW_{t,k+1/2})-V(\bz^{(t)}_{k+1/2})-\nabla h^{\delta_{t}}(\bz^{(t)}_{k+1/2}),\bz^{(t)}_{k+1/2}-\bs_{t}}\vert\scrF_{t}]=\\
&\Ex\left[\Ex\left(\inner{\hat{V}_{t,k+1/2}(\bz^{(t)}_{k+1/2})+H^{\delta_{t}}_{\bz^{(t)}_{k+1/2}}(\bW_{t,k+1/2})-V(\bz^{(t)}_{k+1/2})-\nabla h^{\delta_{t}}(\bz^{(t)}_{k+1/2}),\bz^{(t)}_{k+1/2}-\bs_{t}}\vert\scrA_{t,k}\right)\vert\scrF_{t}\right]\\
&=0
\end{align*}
This being true because $\bz_{k+1/2}$ is $\scrA_{t,k}$-measurable. The remaining two equalities can be demonstrated in the same way. Since $\bz_{K}=\bx^{t+1}$ and $\bz_{0}=\bx^{t}$, and using the results above, we are left with the estimate
$$
\Ex[\norm{\bx^{t+1}-\bs_{t}}^{2}\vert\scrF_{t}]\leq q_{t}^{K}\norm{\bx^{t}-\bs_{t}}^{2}+K\gamma_{t}^{2}(L_{f}^{2}+\eta^{2}_{0})C^{2}+8\gamma_{t}^{2}\left(\sum_{i\in\scrI}L^{2}_{h_{i}}n^{2}_{i}\right)+2K\gamma_{t}\delta_{t}L_{h}.
$$
From Proposition \ref{prop:Tikhonov}(d), we obtain the estimate 
\begin{align*}
\norm{\bx^{t}-\bs_{t}}^{2}&\leq (1+\gamma_{t}\eta_{t})\norm{\bx^{t}-\bs_{t-1}}^{2}+(1+\frac{1}{\gamma_{t}\eta_{t}})\norm{\bs_{t}-\bs_{t-1}}^{2}\\
&\leq (1+\gamma_{t}\eta_{t})\norm{\bx^{t}-\bs_{t-1}}^{2}+(1+\gamma_{t}\eta_{t})\left(\frac{\eta_{t}-\eta_{t-1}}{\eta_{t}}\right)^{2}\bA^{2}_{x}
\end{align*}
where $\bA_{x}$ is a constant upper bound of $\inf_{\bx\in\SOL(V,h)}\norm{\bx}$ (cf. Proposition \ref{prop:Tikconverge}). Moreover, 
\begin{align*}
(1+\gamma_{t}\eta_{t})q_{t}^{K}=(1+\gamma_{t}\eta_{t})(1-\gamma_{t}\eta_{t})q_{t}^{K-1}=(1-\gamma^{2}_{t}\eta_{t}^{2})q_{t}^{K-1}<q_{t}.
\end{align*}
This allows us to conclude 
\begin{align*}
\Ex[\norm{\bx^{t+1}-\bs_{t}}^{2}\vert\scrF_{t}]&\leq q_{t}\norm{\bx^{t}-\bs_{t-1}}^{2}+q_{t}^{K}(1+\frac{1}{\gamma_{t}\eta_{t}})\bA^{2}_{x}\left(\frac{\eta_{t}-\eta_{t-1}}{\eta_{t}}\right)^{2}\\
&+K\gamma^{2}_{t}(L^{2}_{f}+\eta^{2}_{0})C^{2}+8\sum_{i\in\scrI}L^{2}_{h_{i}}n^{2}_{i}+2KL_{h}\gamma_{t}\delta_{t}.
\end{align*}
Define 
$$
\ca_{t}\eqdef q_{t}^{K}(1+\frac{1}{\gamma_{t}\eta_{t}})\bA^{2}_{x}\left(\frac{\eta_{t}-\eta_{t-1}}{\eta_{t}}\right)^{2}+K\gamma^{2}_{t}(L^{2}_{f}+\eta^{2}_{0})C^{2}+8\gamma^{2}_{t}\sum_{i\in\scrI}L^{2}_{h_{i}}n^{2}_{i}+2KL_{h}\gamma_{t}\delta_{t},
$$
and $\psi_{t}\eqdef \Ex[\norm{\bx^{t}-\bs_{t-1}}^{2}]$, so that we obtain the recursion 
$$
\psi_{t+1}\leq q_{t}\psi_{t}+\ca_{t}.
$$
Under the assumptions stated in Theorem \ref{th:converge}, we have $\sum_{t=0}^{\infty}\gamma_{t}\eta_{t}=\infty$ and $\lim_{t\to\infty}\frac{\ca_{t}}{\gamma_{t}\eta_{t}}=0$. Using Lemma 3 in \cite{Pol87}, it follows $\lim_{t\to\infty}\psi_{t}=0$, and therefore $\lim_{t\to\infty}\norm{\bx^{t}-\bs_{t}}=0$ almost surely. Now, let $\bw\eqdef\inf_{\bx\in\SOL(V,h)}\norm{\bx}$. Then, using the triangle inequality we conclude 
$$
\norm{\bx^{t+1}-\bw}\leq \norm{\bx^{t+1}-\bs_{t}}+\norm{\bs_{t}-\bw}\to 0\quad\text{ as }t\to\infty,\text{ a.s.}
$$
Since $\by(\cdot)$ is Lipschitz continuous (cf. Fact \ref{fact:hierarchicalVI}, Appendix \ref{app:VIs}), the claim follows.

\subsubsection{Proof of Theorem \ref{th:gap}}\label{app:proofofComplexity}
We use the energy estimate formulated in Lemma \ref{lem:energy}  to deduce a bound on the gap function \eqref{eq:gap} relative to a suitably constructed ergodic average.
\begin{lemma}\label{lem:boundgap}
For any $t\in\{0,1,\ldots,T-1\}$, define $\bar{\bz}^{t}\eqdef\frac{1}{K}\sum_{k=0}^{K-1}\bz^{(t)}_{k+1/2}$. Let $\{\gamma_{t}\}_{t},\{\eta_{t}\}_{t}$ be positive sequences satisfying $0<\gamma_{t}\eta_{t}<1/2$. For $k\in\{0,1,\ldots,K-1\}$ define 
\begin{align*}
\bY^{1}_{t,k}&\eqdef \hat{V}_{t,k+1/2}(\bz^{(t)}_{k+1/2})+H^{\delta_{t}}_{\bz^{(t)}_{k+1/2}}(\bW_{t,k+1/2})-V(\bz^{(t)}_{k+1/2})-\nabla h^{\delta_{t}}(\bz^{(t)}_{k+1/2}),\text{ and }\\
\bY^{2}_{t,k}&\eqdef V(\bx^{t})+\nabla h^{\delta_{t}}(\bx^{t})-\hat{V}_{t,k+1/2}(\bx^{t})-H^{\delta_{t}}_{\bx^{t}}(\bW_{t,k+1/2}).
\end{align*}
Under the same Assumptions as in Lemma \ref{lem:energy}, we have for all $\bx\in\scrX$:
\begin{equation}\label{eq:lastbound}
\begin{split}
\gamma_{t}&\left(\inner{V(\bx),\bar{\bz}^{t}-\bx}+h(\bar{\bz}^{t})-h(\bx)\right)\leq\frac{1}{2K}\left(\norm{\bx^{t}-\bx}^{2}-\norm{\bx^{t+1}-\bx}^{2}\right)\\
&+2\frac{\gamma^{2}_{t}}{K}(\scrL_{f}(\xi_{t,k+1/2})^{2}+\eta^{2}_{0})\sum_{k=0}^{K-1}\norm{\bz^{(t)}_{k+1/2}-\bx}^{2}+4\gamma^{2}_{t}\sum_{i\in\scrI}L^{2}_{h_{i}}n^{2}_{i}\\
&-\gamma_{t}\inner{\eps^{t}_{V}(\xi^{1:b_{t}}_{t})+\eps^{t}_{h}(\bW^{1:b_{t}}_{t}),\bar{\bz}^{t}-\bx}+\gamma_{t}\delta_{t}L_{h}-\gamma_{t}\eta_{t}\inner{\bx,\bar{\bz}^{t}-\bx}\\
&-\frac{\gamma_{t}}{K}\sum_{k=0}^{K-1}\inner{\bY^{1}_{t,k},\bz^{(t)}_{k+1/2}-\bx}-\frac{\gamma_{t}}{K}\sum_{k=0}^{K-1}\inner{\bY^{2}_{t,k},\bz^{(t)}_{k+1/2}-\bx}.
\end{split}
\end{equation}
\end{lemma}
\begin{proof}
We depart from the energy bound in Lemma \ref{lem:energy}. Rearranging this inequality and using $\gamma_{t}\eta_{t}\in(0,1/2)$, it follows 
\begin{align*}
2&\gamma_{t}\left(\inner{V(\bx),\bz^{(t)}_{k+1/2}-\bx}+h(\bz^{(t)}_{k+1/2})-h(\bx)\right)\leq \norm{\bz^{(t)}_{k}-\bx}^{2}-\norm{\bz^{(t)}_{k+1}-\bx}^{2}\\
&+8\gamma_{t}^{2}\sum_{i\in\scrI}L_{h_{i}}^{2}n^{2}_{i}+4\gamma^{2}_{t}(\scrL_{f}(\xi_{t,k+1/2})^{2}+\eta^{2}_{0})\norm{\bz^{(t)}_{k+1/2}-\bx^{t}}^{2}\\
&-2\gamma_{t}\inner{\hat{V}_{t,k+1/2}(\bz^{(t)}_{k+1/2})+H^{\delta_{t}}_{\bz^{(t)}_{k+1/2}}(\bW_{k+1/2})-V(\bz^{(t)}_{k+1/2})-\nabla h^{\delta_{t}}(\bz^{(t)}_{k+1/2}),\bz^{(t)}_{k+1/2}-\bx}\\
&-2\gamma_{t}\inner{\left(V(\bx^{t})+\nabla h^{\delta_{t}}(\bx^{t})\right)-\left(\hat{V}_{t,k+1/2}(\bx^{t})+H^{\delta_{t}}_{\bx^{t}}(\bW_{t,k+1/2})\right),\bz^{(t)}_{k+1/2}-\bx}\\
&-2\gamma_{t}\inner{\eps^{t}_{V}(\xi^{1:b_{t}}_{t})+\eps^{t}_{h}(\bW^{1:b_{t}}_{t}),\bz^{(t)}_{k+1/2}-\bx}+2\gamma_{t}\delta_{t}L_{h}-2\gamma_{t}\eta_{t}\inner{\bx,\bz^{(t)}_{k+1/2}-\bx}.
\end{align*}
Summing from $k=0,\ldots,K-1$ and calling $\bar{\bz}^{t}\eqdef\frac{1}{K}\sum_{k=0}^{K-1}\bz^{(t)}_{k+1/2}$, we get first from Jensen's inequality 
$$
2\frac{\gamma_{t}}{K}\sum_{k=0}^{K-1}\left(\inner{V(\bx),\bz^{(t)}_{k+1/2}-\bx}+h(\bz^{(t)}_{k+1/2})-h(\bx)\right)\geq 2\gamma_{t}\left(\inner{V(\bx),\bar{\bz}^{t}-\bx}+h(\bar{\bz}^{t})-h(\bx)\right).
$$
Second, telescoping the expression in the penultimate display and using the definitions of the process $\{\bY^{\nu}_{t,k}\}_{k=0}^{K-1}$, we deduce the bound
\begin{align*}
\gamma_{t}\left(\inner{V(\bx),\bar{\bz}^{t}-\bx}+h(\bar{\bz}^{t})-h(\bx)\right)&\leq\frac{1}{2K}\left(\norm{\bx^{t}-\bx}^{2}-\norm{\bx^{t+1}-\bx}^{2}\right)\\
&+2\frac{\gamma^{2}_{t}}{K}\sum_{k=0}^{K-1}(\scrL_{f}(\xi_{t,k+1/2})^{2}+\eta^{2}_{0})\norm{\bz_{k+1/2}-\bx}^{2}+4\gamma^{2}_{t}\sum_{i\in\scrI}L^{2}_{h_{i}}n^{2}_{i}\\
&-\gamma_{t}\inner{\eps^{t}_{V}(\xi^{1:b_{t}}_{t})+\eps^{t}_{h}(\bW^{1:b_{t}}_{t}),\bar{\bz}^{t}-\bx}+\gamma_{t}\delta_{t}L_{h}-\gamma_{t}\eta_{t}\inner{\bx,\bar{\bz}^{t}-\bx}\\
&-\frac{\gamma_{t}}{K}\sum_{k=0}^{K-1}\inner{\bY^{1}_{t,k},\bz^{(t)}_{k+1/2}-\bx}-\frac{\gamma_{t}}{K}\sum_{k=0}^{K-1}\inner{\bY^{2}_{t,k},\bz^{(t)}_{k+1/2}-\bx}.
\end{align*}
\end{proof}
We can now give the proof of Theorem \ref{th:gap}. Let us introduce the auxiliary processes $\{\bu^{\nu}_{t,k}\}_{k=0}^{K-1}$ for $\nu=1,2$ defined recursively as 
\begin{equation}
\bu^{\nu}_{t,k+1}=\Pi_{\scrX}(\bu^{\nu}_{t,k}-\gamma_{t}\bY^{\nu}_{t,k}),\quad\bu^{\nu}_{t,0}=\bx^{t}.
\end{equation}
The definition of the auxiliary sequence gives for $\nu=1,2$ (see e.g. \cite{NJLS09})
\begin{align*}
\norm{\bu^{\nu}_{t,k+1}-\bx}^{2}&\leq \norm{\bu^{\nu}_{t,k}-\bx}^{2}-2\gamma_{t}\inner{\bY^{\nu}_{t,k},\bu^{\nu}_{t,k}-\bx}+\gamma^{2}_{t}\norm{\bY^{\nu}_{t,k}}^{2}\\
&=\norm{\bu^{\nu}_{t,k}-\bx}^{2}+2\gamma_{t}\inner{\bY^{\nu}_{t,k},\bz^{(t)}_{k+1/2}-\bx}-2\gamma_{t}\inner{\bY^{\nu}_{t,k},\bu^{\nu}_{t,k}-\bz^{(t)}_{k+1/2}}+\gamma^{2}_{t}\norm{\bY^{\nu}_{t,k}}^{2}.
\end{align*}
Rearranging and telescoping shows 
\begin{align*}
-2\gamma_{t}\sum_{k=0}^{K-1}\inner{\bY^{\nu}_{t,k},\bz^{(t)}_{k+1/2}-\bx}&\leq \norm{\bu^{\nu}_{t,0}-\bx}^{2}-\norm{\bu^{\nu}_{t,K}-\bx}^{2}+\gamma^{2}_{t}\sum_{k=0}^{K-1}\norm{\bY^{\nu}_{t,k}}^{2}\\
&-2\gamma_{t}\sum_{k=0}^{K-1}\inner{\bY^{\nu}_{t,k},\bz^{(t)}_{k+1/2}-\bu^{\nu}_{t,k}}.
\end{align*}
Plugging this into eq. \eqref{eq:lastbound}, we get 
\begin{align*}
\gamma_{t}&\left(\inner{V(\bx),\bar{\bz}^{t}-\bx}+h(\bar{\bz}^{t})-h(\bx)\right)\leq \frac{1}{2K}\left(\norm{\bx^{t}-\bx}^{2}-\norm{\bx^{t+1}-\bx}^{2}\right)\\
&+2\frac{\gamma^{2}_{t}}{K}\sum_{k=0}^{K-1}(\scrL_{f}(\xi_{t,k+1/2})^{2}+\eta^{2}_{0})\norm{\bz^{(t)}_{k+1/2}-\bx}^{2}+4\gamma^{2}_{t}\sum_{i\in\scrI}L^{2}_{h_{i}}n^{2}_{i}\\
&-\gamma_{t}\inner{\eps^{t}_{V}(\xi^{1:b_{t}}_{t})+\eps^{t}_{h}(\bW^{1:b_{t}}_{t}),\bar{\bz}^{t}-\bx}+\gamma_{t}\delta_{t}L_{h}-\gamma_{t}\eta_{t}\inner{\bx,\bar{\bz}^{t}-\bx}\\
&+\frac{\gamma_{t}}{K}\left(\norm{\bu^{1}_{t,0}-\bx}^{2}-\norm{\bu^{1}_{t,K}-\bx}^{2}+\gamma^{2}_{t}\sum_{k=0}^{K-1}\norm{\bY^{1}_{t,k}}^{2}-2\gamma_{t}\sum_{k=0}^{K-1}\inner{\bY^{1}_{t,k},\bz^{(t)}_{k+1/2}-\bu^{1}_{t,k}}\right)\\
&+\frac{\gamma_{t}}{K}\left(\norm{\bu^{2}_{t,0}-\bx}^{2}-\norm{\bu^{2}_{t,K}-\bx}^{2}+\gamma^{2}_{t}\sum_{k=0}^{K-1}\norm{\bY^{2}_{t,k}}^{2}-2\gamma_{t}\sum_{k=0}^{K-1}\inner{\bY^{2}_{t,k},\bz^{(t)}_{k+1/2}-\bu^{2}_{t,k}}\right)
\end{align*}
Summing this expression over the outer-iteration loop and introduce the averaged iterate 
$$
\bar{\bz}^{T}\eqdef\frac{\sum_{t=0}^{T-1}\gamma_{t}\bar{\bz}^{t}}{\sum_{t=0}^{T-1}\gamma_{t}}.
$$
Jensen's inequality readily implies
$$
\sum_{t=0}^{T-1}\left(\inner{V(\bx),\bar{\bz}^{t}-\bx}+h(\bar{\bz}^{t})-h(\bx)\right)\geq \left(\sum_{t=0}^{T-1}\gamma_{t}\right)\left(\inner{V(\bx),\bar{\bz}^{T}-\bx}+h(\bar{\bz}^{T})-h(\bx)\right).
$$
Recall that $C=\max_{i\in\scrI}C_{i}$ is the upper bound on the diameter of the set $\scrX_{i}$ (cf. Assumption \ref{ass:standing}.(i)). This assumed compactness of the set $\scrX$, we derive the a-priori bounds 
\begin{align*}
&\norm{\bx^{t}-\bx}\leq C\qquad \forall t=0,1,\ldots,T-1,\\
&\norm{\bz^{(t)}_{k+1/2}-\bx}\leq C\qquad \forall k=0,1,\ldots,K-1,\text{ and }\norm{\bar{\bz}^{t}-\bx}\leq C\qquad \forall t=0,1,\ldots,T-1.
\end{align*}
Using these bounds, we conclude 
\begin{align*}
\inner{V(\bx),\bar{\bz}^{T}-\bx}+h(\bar{\bz}^{T})-h(\bx)&\leq \frac{1}{2K\sum_{t=0}^{T-1}\gamma_{t}}\left(\norm{\bx^{0}-\bx}^{2}-\norm{\bx^{T}-\bx}^{2}\right)\\
&+\frac{2\sum_{t=0}^{T-1}\gamma^{2}_{t}\sum_{k=0}^{K-1}(\scrL_{f}(\xi_{t,k+1/2})^{2}+\eta^{2}_{0})\norm{\bz^{(t)}_{k+1/2}-\bx}^{2}}{K\sum_{t=0}^{T-1}\gamma_{t}}\\
&+\frac{\sum_{t=0}^{T-1}4\gamma^{2}_{t}\sum_{i\in\scrI}L^{2}_{h_{i}}n^{2}_{i}+L_{h}\sum_{t=0}^{T-1}\delta_{t}\gamma_{t}}{\sum_{t=0}^{T-1}\gamma_{t}}\\
&+\frac{1}{K\sum_{s=0}^{T-1}\gamma_{s}}\sum_{t=0}^{T-1}\gamma_{t}\left(\norm{\bu^{1}_{t,0}-\bx}^{2}+\norm{\bu^{2}_{t,0}-\bx}^{2}\right)\\
&+\frac{\sum_{t=0}^{T-1}\gamma^{3}_{t}}{K\sum_{s=0}^{T-1}\gamma_{s}}\sum_{k=0}^{K-1}\left(\norm{\bY^{1}_{t,k}}^{2}+\norm{\bY^{2}_{t,k}}^{2}\right)\\
&+\frac{2\sum_{t=0}^{T-1}\gamma^{2}_{t}}{K\sum_{s=0}^{T-1}\gamma_{s}}\sum_{k=0}^{K-1}\left(\inner{\bY^{1}_{t,k},\bu^{1}_{t,k}-\bz^{(t)}_{k+1/2}}+\inner{\bY^{2}_{t,k},\bu_{t,k}^{2}-\bz^{(t)}_{k+1/2}}\right)\\
&+\frac{\sum_{t=0}^{T-1}\gamma_{t}\inner{\eps^{t}_{V}(\xi^{1:b_{t}}_{t})+\eps^{t}_{h}(\bW^{1:b_{t}}_{t}),\bx-\bar{\bz}^{t}}}{\sum_{t=0}^{T-1}\gamma_{t}}\\
&+\frac{\sum_{t=0}^{T}\gamma_{t}\eta_{t}\inner{\bx,\bx-\bar{\bz}^{t}}}{\sum_{t=0}^{T-1}\gamma_{t}}
\end{align*}
We bound each of the terms above individually as follows: 
\begin{enumerate}
\item $\frac{1}{2K\sum_{t=0}^{T-1}\gamma_{t}}\left(\norm{\bx^{0}-\bx}^{2}-\norm{\bx^{T}-\bx}^{2}\right)\leq \frac{C^{2}}{2K\sum_{t=0}^{T-1}\gamma_{t}},$
\item $\frac{2\sum_{t=0}^{T-1}\gamma^{2}_{t}\sum_{k=0}^{K-1}(\scrL_{f}(\xi_{t,k+1/2})^{2}+\eta^{2}_{0})\norm{\bz^{(t)}_{k+1/2}-\bx}^{2}}{K\sum_{t=0}^{T-1}\gamma_{t}}\leq 2C^{2}\frac{\sum_{t=0}^{T-1}\gamma^{2}_{t}\sum_{k=0}^{K-1}(\scrL_{f}(\xi_{t,k+1/2})^{2}+\eta^{2}_{0})}{K\sum_{s=0}^{T-1}\gamma_{s}},$
\item $\frac{1}{K}\left(\norm{\bu^{1}_{t,0}-\bx}^{2}+\norm{\bu^{2}_{t,0}-\bx}^{2}\right)=\frac{2}{K}\norm{\bx^{t}-\bx}^{2}\leq \frac{2C^{2}}{K},$ 
\item $\frac{\sum_{t=0}^{T}\gamma_{t}\eta_{t}\inner{\bx,\bx-\bar{\bz}^{t}}}{\sum_{t=0}^{T-1}\gamma_{t}}\leq \frac{3C^{2}\sum_{t=0}^{T-1}\gamma_{t}\eta_{t}}{2\sum_{t=0}^{T-1}\gamma_{t}},$
\item $\frac{\sum_{t=0}^{T-1}\gamma_{t}\inner{\eps^{t}_{V}(\xi^{1:b_{t}}_{t})+\eps^{t}_{h}(\bW^{1:b_{t}}_{t}),\bx-\bar{\bz}^{t}}}{\sum_{t=0}^{T-1}\gamma_{t}}\leq \frac{\sum_{t=0}^{T-1}\gamma_{t}(C\norm{\eps^{t}_{V}(\xi^{1:b_{t}}_{t})}+C\norm{\eps^{t}_{h}(\bW^{1:b_{t}}_{t})})}{\sum_{t=0}^{T-1}\gamma_{t}}$
\end{enumerate}
Plugging all these bounds into the penultimate display gives 
\begin{align*}
\inner{V(\bx),\bar{\bz}^{T}-\bx}+h(\bar{\bz}^{T})-h(\bx)&\leq  \frac{C^{2}}{2K\sum_{t=0}^{T-1}\gamma_{t}}+2C^{2}\frac{\sum_{t=0}^{T-1}\gamma^{2}_{t}\sum_{k=0}^{K-1}(\scrL_{f}(\xi_{t,k+1/2})^{2}+\eta^{2}_{0})}{K\sum_{s=0}^{T-1}\gamma_{s}}\\
&+\frac{\sum_{t=0}^{T-1}4\gamma^{2}_{t}\sum_{i\in\scrI}L^{2}_{h_{i}}n^{2}_{i}+L_{h}\sum_{t=0}^{T-1}\delta_{t}\gamma_{t}}{\sum_{t=0}^{T-1}\gamma_{t}}\\
&+\frac{2C^{2}}{K}+\frac{\sum_{t=0}^{T-1}\gamma_{t}(C\norm{\eps^{t}_{V}(\xi^{1:b_{t}}_{t})}+C\norm{\eps^{t}_{h}(\bW^{1:b_{t}}_{t})})}{\sum_{t=0}^{T-1}\gamma_{t}}\\
&+\frac{\sum_{t=0}^{T-1}\gamma^{3}_{t}}{K\sum_{t=0}^{T-1}\gamma_{t}}\sum_{k=0}^{K-1}\left(\norm{\bY^{1}_{t,k}}^{2}+\norm{\bY^{2}_{t,k}}^{2}\right)\\
&+\frac{2\sum_{t=0}^{T-1}\gamma^{2}_{t}}{K\sum_{t=0}^{T-1}\gamma_{t}}\sum_{k=0}^{K-1}\left(\inner{\bY^{1}_{t,k},\bu^{1}_{t,k}-\bz^{(t)}_{k+1/2}}+\inner{\bY^{2}_{t,k},\bu_{t,k}^{2}-\bz_{k+1/2}}\right)\\
&+\frac{3C^{2}\sum_{t=0}^{T-1}\gamma_{t}\eta_{t}}{2\sum_{t=0}^{T-1}\gamma_{t}}.
\end{align*}
Let $L_{f}^{2}\eqdef\Ex_{\xi}[\scrL_{f}(\xi)^{2}]$, and define 
\begin{equation}\label{eq:constantfactor}
\begin{split}
\scrD_{K,T}&\eqdef  \frac{C^{2}}{2K\sum_{t=0}^{T-1}\gamma_{t}}+2C^{2}(L_{f}^{2}+\eta^{2}_{0})\frac{\sum_{t=0}^{T-1}\gamma^{2}_{t}}{\sum_{t=0}^{T-1}\gamma_{t}}+\frac{3C^{2}\sum_{t=0}^{T-1}\gamma_{t}\eta_{t}}{2\sum_{t=0}^{T-1}\gamma_{t}}\\
&+\frac{\sum_{t=0}^{T-1}4\gamma^{2}_{t}\sum_{i\in\scrI}L^{2}_{h_{i}}n^{2}_{i}+L_{h}\sum_{t=0}^{T-1}\delta_{t}\gamma_{t}}{\sum_{t=0}^{T-1}\gamma_{t}}+\frac{2C^{2}}{K}.
\end{split}
\end{equation}
Hence, using the definition of the gap function \eqref{eq:gap}, we see 
\begin{align*}
\Ex[\Gamma(\bar{\bz}^{T})]&\leq \scrD_{K,T}+\Ex\left[\frac{\sum_{t=0}^{T-1}\gamma^{3}_{t}}{K\sum_{t=0}^{T-1}\gamma_{t}}\sum_{k=0}^{K-1}\left(\norm{\bY^{1}_{t,k}}^{2}+\norm{\bY^{2}_{t,k}}^{2}\right)\right]\\
&+\Ex\left[\frac{2\sum_{t=0}^{T-1}\gamma^{2}_{t}}{K\sum_{t=0}^{T-1}\gamma_{t}}\sum_{k=0}^{K-1}\left(\inner{\bY^{1}_{t,k},\bu^{1}_{t,k}-\bz^{(t)}_{k+1/2}}+\inner{\bY^{2}_{t,k},\bu_{t,k}^{2}-\bz^{(t)}_{k+1/2}}\right)\right]\\
&+\Ex\left[\frac{\sum_{t=0}^{T-1}\gamma_{t}(C\norm{\eps^{t}_{V}(\xi^{1:b_{t}}_{t})}+C\norm{\eps^{t}_{h}(\bW^{1:b_{t}}_{t})})}{\sum_{t=0}^{T-1}\gamma_{t}}\right].
\end{align*}
Next, observe that 
\begin{align*}
&\norm{\bY^{1}_{t,k}}\leq \norm{\hat{V}_{t,k+1/2}(\bz^{(t)}_{k+1/2})-V(\bz_{k+1/2})}+\norm{H^{\delta_{t}}_{\bz^{(t)}_{k+1/2}}(\bW_{t,k+1/2})-\nabla h^{\delta_{t}}(\bz^{(t)}_{k+1/2})},\text{ and }\\
&\norm{\bY^{2}_{t,k}}\leq \norm{\hat{V}_{t,k+1/2}(\bx^{t})-V(\bx^{t})}+\norm{H^{\delta_{t}}_{\bx^{t}}(\bW_{t,k+1/2})-\nabla h^{\delta_{t}}(\bx^{t})}.
\end{align*}
Moreover, using compactness of $\scrX$, 
\begin{equation}
\inner{\bY^{\nu}_{t,k},\bu^{\nu}_{t,k}-\bz^{(t)}_{k+1/2}}\leq C\norm{\bY^{\nu}_{t,k}}\quad \forall \nu=1,2.
\end{equation}
Lemma \ref{lem:boundH}(c) (for $b=1$) gives 
$$
\Ex_{\bW\sim\Uniform(\sphere_{n})}\left[\norm{H^{\delta_{t}}_{\bx^{t}}(\bW_{t,k+1/2})-\nabla h^{\delta_{t}}(\bx^{t})}^{2}\vert\scrA_{t,k}\right]\leq \sum_{i\in\scrI}n^{2}_{i}L^{2}_{h_{i}}.
$$
Lemma \ref{lem:Variance} in turn implies 
$$
\Ex_{\xi}\left[ \norm{\hat{V}_{t,k+1/2}(\bz^{(t)}_{k+1/2})-V(\bz^{(t)}_{k+1/2})}^{2}\vert\scrA_{t,k}\right]\leq M^{2}_{V}.
$$
By Jensen's inequality in tandem with Lemma \ref{lem:Variance} and Lemma \ref{lem:boundHindividual}, we conclude from \eqref{eq:errV} and \eqref{eq:errH}
\begin{align*}
\Ex\left[\norm{\eps_{V}^{t}(\xi^{1:b_{t}}_{t})}\vert\scrF_{t}\right]&\leq \sqrt{\Ex\left[\norm{\eps_{V}^{t}(\xi^{1:b_{t}}_{t})}^{2}\vert\scrF_{t}\right]}\leq\frac{M_{V}}{\sqrt{b_{t}}},\text{ and }\\
\Ex\left[\norm{\eps_{h}^{t}(\bW^{1:b_{t}}_{t})}\vert\scrF_{t}\right]&\leq \sqrt{\Ex\left[\norm{\eps_{h}^{t}(\bW^{1:b_{t}}_{t})}^{2}\vert\scrF_{t}\right]}\leq\frac{\left(\sum_{i\in\scrI}L^{2}_{h_{i}}n^{2}_{i}\right)^{1/2}}{\sqrt{b_{t}}}.
\end{align*}
This implies $\Ex[\norm{\bY^{\nu}_{t,k}}^{2}\vert\scrA_{t,k}]\leq 2M^{2}_{V}+2\sum_{i\in\scrI}L^{2}_{h_{i}}n^{2}_{i}\equiv\sigma^{2}$ for all $\nu=1,2$, as well as 
$$
\Ex[\norm{\bY^{\nu}_{t,k}}\vert\scrA_{t,k}]\leq \sqrt{\Ex[\norm{\bY^{\nu}_{t,k}}^{2}\vert\scrA_{t,k}]}\leq \sigma.
$$
We conclude, via a repeated application of the law of iterated expectations, that
\begin{align*}
\Ex[\Gamma(\bar{\bz}^{T})]&\leq \scrD_{K,T}+\frac{2\sigma^{2}\sum_{t=0}^{T-1}\gamma^{3}_{t}}{\sum_{t=0}^{T-1}\gamma_{t}}+\frac{4C\sigma\sum_{t=0}^{T-1}\gamma^{2}_{t}}{\sum_{t=0}^{T-1}\gamma_{t}}\\
&+\frac{\sum_{t=0}^{T-1}\frac{\gamma_{t}}{\sqrt{b_{t}}}\left(CM_{V}+C(\sum_{i\in\scrI}L_{h_{i}}^{2}n_{i}^{2})^{1/2}\right)}{\sum_{t=0}^{T-1}\gamma_{t}}.
\end{align*}
Using the specification $\gamma_{t}=1/T=\eta_{t}=\delta_{t}$, as well as $K=T$ and $b_{t}\geq T^{2}$ gives 
\begin{align*}
\scrD_{K,T}\leq \frac{5C^{2}}{T}+\frac{2C^{2}(L^{2}_{f}+1/T^{2})}{T}+\frac{3C^{2}}{2T}+\frac{4(\sum_{i\in\scrI}L^{2}_{h_{i}}n^{2}_{i})+L_{h}}{T}\equiv \const_{T}
\end{align*}
and consequently, 
$$
\Ex[\Gamma(\bar{\bz}^{T})]\leq \const_{T}+\frac{2\sigma^{2}}{T^{2}}+\frac{C\sigma}{T}+\frac{C(M_{V}+(\sum_{i\in\scrI}L^{2}_{h_{i}}n^{2}_{i})^{1/2})}{T}=\scrO(C\sigma/T) 
$$

%% file: Appendix_Inexact.tex
%

The inexact version of our method $\VRHGS$ is obtained by replacing the estimates for the implicit function using the inexact solution map $\by_{i}^{\eps}$. The precise implementation is summarized in Alg. \ref{alg:inexactSFBF} and Alg. \ref{alg:Inexact-VRHGS}. 
\begin{algorithm}[t!]
\SetAlgoLined
\caption{$\ISFBF(\bar{\bx},\bar{\bv},\bar{H},\gamma,\eta,\delta,\eps,K)$}
 \label{alg:inexactSFBF}
\KwResult{Iterate $\bz^{K}$}
Set $\bz^{0}=\bar{\bx}$\;

\For{$k=0,1,\ldots,K-1$}{
  Update $\bz_{k+1/2}=\Pi_{\scrX}[\bz_{k}-\gamma(\bar{\bv}+\eta\bar{\bx}+\bar{H})]$,\;
  
   Obtain $\hat{V}^{\eta}_{k+1/2}(\bz_{k+1/2})$ and $\hat{V}^{\eta}_{k+1/2}(\bar{\bx})$ as defined in eq. \eqref{eq:hatV}\;

Draw iid direction vectors $\bW_{k+1/2}=\{\bW_{i,k+1/2}\}_{i\in\scrI}$, with each $\bW_{i,k+1/2}\sim\Uniform(\sphere_{i})$. \; 
    
    Obtain $H^{\delta,\eps}_{\bz_{k+1/2}}(\bW_{k+1/2})$ and $H^{\delta,\eps}_{\bar{\bx}}(\bW_{k+1/2})$\;
  
  Update 
      \vspace{-0.2in}
  $$
  \bz_{k+1}=\bz_{k+1/2}-\gamma\left(\hat{V}^{\eta}_{k+1/2}(\bz_{k+1/2})+H^{\delta,\eps}_{\bz_{k+1/2}}(\bW_{k+1/2})-\hat{V}^{\eta}_{k+1/2}(\bar{\bx})-H^{\delta,\eps}_{\bar{\bx}}(\bW_{k+1/2})\right).
  $$
      \vspace{-0.2in}
  \;
  }
\end{algorithm}
\begin{algorithm}[t!]
\SetAlgoLined
    \caption{Inexact Variance Reduced Hierarchical Game Solver 
    (I-$\VRHGS$) } \label{alg:Inexact-VRHGS}
    \KwData{$\bx,T,\{\gamma_{t}\}_{t=0}^{T},\{b_{t}\}_{t=0}^{T},\{\eta_{t}\}_{t=0}^{T},\{\eps_{t}\}_{t=0}^{T}$}
Set $\bx^{0}=\bx$.\\
\For{$t=0,1,\ldots,T-1$}{
For each $i\in\scrI$ receive the oracle feedback $\bar{V}^{t}$ defined by $\bar{V}^{t}_{i}\eqdef\frac{1}{b_{t}}\sum_{s=1}^{b_{t}}\hat{V}_{i}(\bx^{t},\xi_{i,t}^{(s)})$.
\; 

For each $i\in\scrI$ construct the estimator $H^{\delta_{t},b_{t}}_{\bx^{t}}$ defined by $H_{i,\bx_{i}}^{\delta_{t},\eps_{t},b_{t}}\eqdef \frac{1}{b_{t}}\sum_{s=1}^{b_{t}}H_{i,\bx_{i}}^{\delta_{t},\eps_{t}}(\bW_{i,t}^{(s)})$.
\;

Update $\bx^{t+1}=\ISFBF(\bx^{t},\bar{V}^{t},H^{\delta_{t},b_{t},\eps_{t}}_{\bx^{t}},\gamma_{t},\eta_{t},\delta_{t},\eps_{t},b_{t},K)$
  }
\end{algorithm}
The proof of Theorem \ref{th:inexact} is analogous to the one of Theorem \ref{th:gap}, with the simple modification due to inexact feedback from the follower's problem. We state the main changes here, leaving the straightforward derivations to the reader. We begin with the modified energy inequality, similar to Lemma \ref{lem:energy}. Here, we also follow the same notational simplification by suppressing the outer iteration counter $t$ from the variables. 

\begin{lemma}\label{lem:energy-inexact}
Let Assumptions \ref{ass:LLunique},\ref{ass:standing},\ref{ass:implicit} and \ref{ass:SO_V} hold true. Then, for all $t\in\{0,1,\ldots,T-1\}$ and all anchor points $\bx\in\scrX$, we have 
\begin{align*}
\norm{\bz_{k+1}-\bx}^{2}&\leq(1-\gamma_{t}\eta_{t})\norm{\bz_{k}-\bx}^{2}-(1-2\gamma_{t}\eta_{t})\norm{\bz_{k+1/2}-\bz_{k}}^{2}\\
&+8\gamma_{t}^{2}\left(\norm{H^{\delta_{t},\eps_{t}}_{\bz_{k+1/2}}(\bW_{k+1/2})}^{2}+\norm{H^{\delta_{t},\eps_{t}}_{\bx^{t}}(\bW_{k+1/2})}^{2}\right)\\
&+4\gamma^{2}_{t}(\scrL_{f}(\xi_{k+1/2})^{2}+\eta^{2}_{0})\norm{\bz_{k+1/2}-\bx^{t}}^{2}\\
&-2\gamma_{t}\inner{\hat{V}_{k+1/2}(\bz_{k+1/2})+H^{\delta_{t}}_{\bz_{k+1/2}}(\bW_{k+1/2})-V(\bz_{k+1/2})-\nabla h^{\delta_{t}}(\bz_{k+1/2}),\bz_{k+1/2}-\bx}\\
&-2\gamma_{t}\inner{\left(V(\bx^{t})+\nabla h^{\delta_{t}}(\bx^{t})\right)-\left(\hat{V}_{k+1/2}(\bx^{t})+H^{\delta_{t}}_{\bx^{t}}(\bW_{k+1/2})\right),\bz_{k+1/2}-\bx}\\
&-2\gamma_{t}\inner{\eps^{t}_{V}(\xi^{1:b_{t}})+\eps^{t}_{h}(\bW^{1:b_{t}}),\bz_{k+1/2}-\bx}\\
&-2\gamma_{t}\inner{H^{\delta_{t},b_{t}}_{\bx^{t}}-H^{\delta_{t},b_{t},\eps_{t}}_{\bx^{t}}+H^{\delta_{t},\eps_{t}}_{\bx^{t}}(\bW_{k+1/2})-H^{\delta_{t}}_{\bx^{t}}(\bW_{k+1/2}),\bz_{k+1/2}-\bx}\\
&-2\gamma_{t}\inner{H^{\delta_{t},\eps_{t}}_{\bz_{k+1/2}}(\bW_{k+1/2})-H^{\delta_{t}}_{\bz_{k+1/2}}(\bW_{k+1/2}),\bz_{k+1/2}-\bx}\\
&-2\gamma_{t}\left(\inner{V^{\eta_{t}}(\bx),\bz_{k+1/2}-\bx}+h(\bz_{k+1/2})-h(\bx)\right)+2\gamma_{t}\delta_{t}L_{h},
\end{align*}
where $L_{h}\eqdef\sum_{i\in\scrI}L_{h_{i}}$.
\end{lemma}
The inexact version of Lemma \ref{lem:boundgap} reads then as follows.

\begin{lemma}\label{lem:boundgap-inexact}
For any $t\in\{0,1,\ldots,T-1\}$, define $\bar{\bz}^{t}\eqdef\frac{1}{K}\sum_{k=0}^{K-1}\bz^{(t)}_{k+1/2}$. Let $\{\gamma_{t}\}_{t},\{\eta_{t}\}_{t}$ be positive sequences satisfying $0<\gamma_{t}\eta_{t}<1/2$. For $k\in\{0,1,\ldots,K-1\}$ define the process $\{\bY^{\nu}_{t,k}\}_{k=0}^{K-1},\nu\in\{1,2\}$ as in Lemma \ref{lem:boundgap}. Then, we have for all $\bx\in\scrX$:
\begin{align*}
\gamma_{t}&\left(\inner{V(\bx),\bar{\bz}^{t}-\bx}+h(\bar{\bz}^{t})-h(\bx)\right)\leq\frac{1}{2K}\left(\norm{\bx^{t}-\bx}^{2}-\norm{\bx^{t+1}-\bx}^{2}\right)\\
&+\frac{2\gamma^{2}_{t}}{K}\sum_{k=0}^{K-1}(\scrL_{f}(\xi_{t,k+1/2})^{2}+\eta^{2}_{0})\norm{\bz^{(t)}_{k+1/2}-\bx}^{2}-\gamma_{t}\eta_{t}\inner{\bx,\bar{\bz}^{t}-\bx}\\
&-\gamma_{t}\inner{\eps^{t}_{V}(\xi^{1:b_{t}}_{t})+\eps^{t}_{h}(\bW^{1:b_{t}}_{t}),\bar{\bz}^{t}-\bx}+\gamma_{t}\delta_{t}L_{h}\\
&-\frac{\gamma_{t}}{K}\sum_{k=0}^{K-1}\inner{\bY^{1}_{t,k},\bz^{(t)}_{k+1/2}-\bx}-\frac{\gamma_{t}}{K}\sum_{k=0}^{K-1}\inner{\bY^{2}_{t,k},\bz^{(t)}_{k+1/2}-\bx}\\
&+\frac{2\gamma^{2}_{t}}{K}\sum_{k=0}^{K-1}\left(\norm{H^{\delta_{t},\eps_{t}}_{\bz^{(t)}_{k+1/2}}(\bW_{t,k+1/2})}^{2}+\norm{H^{\delta_{t},\eps_{t}}_{\bx^{t}}(\bW_{t,k+1/2})}^{2}\right)\\
&-\frac{\gamma_{t}}{K}\sum_{k=0}^{K-1}\inner{H^{\delta_{t},b_{t}}_{\bx^{t}}-H^{\delta_{t},b_{t},\eps_{t}}_{\bx^{t}}+H^{\delta_{t},\eps_{t}}_{\bx^{t}}(\bW_{t,k+1/2})-H^{\delta_{t}}_{\bx^{t}}(\bW_{t,k+1/2}),\bz^{(t)}_{k+1/2}-\bx}\\
&-\frac{\gamma_{t}}{K}\sum_{k=0}^{K-1}\inner{H^{\delta_{t},\eps_{t}}_{\bz^{(t)}_{k+1/2}}(\bW_{t,k+1/2})-H^{\delta_{t}}_{\bz^{(t)}_{k+1/2}}(\bW_{t,k+1/2}),\bz^{(t)}_{k+1/2}-\bx}.
\end{align*}
\end{lemma}

Using this bound, we conclude in the same way as in the analysis of the exact scheme that 
\begin{align*}
&\inner{V(\bx),\bar{\bz}^{T}-\bx}+h(\bar{\bz}^{T})-h(\bx)\leq \frac{1}{2K\sum_{t=0}^{T-1}\gamma_{t}}\norm{\bx^{0}-\bx}^{2}+\frac{\sum_{t=0}^{T}\gamma_{t}\eta_{t}\inner{\bx,\bx-\bar{\bz}^{t}}}{\sum_{t=0}^{T-1}\gamma_{t}}+\frac{L_{h}\sum_{t=0}^{T-1}\delta_{t}\gamma_{t}}{\sum_{t=0}^{T-1}\gamma_{t}}\\
&+\frac{2\sum_{t=0}^{T-1}\gamma^{2}_{t}\sum_{k=0}^{K-1}(\scrL_{f}(\xi_{t,k+1/2})^{2}+\eta^{2}_{0})\norm{\bz^{(t)}_{k+1/2}-\bx}^{2}}{K\sum_{s=0}^{T-1}\gamma_{s}}\\
&+\frac{1}{K\sum_{s=0}^{T-1}\gamma_{s}}\sum_{t=0}^{T-1}\gamma_{t}\left(\norm{\bu^{1}_{t,0}-\bx}^{2}+\norm{\bu^{2}_{t,0}-\bx}^{2}\right)\\
&+\frac{\sum_{t=0}^{T-1}\gamma^{3}_{t}}{K\sum_{t=0}^{T-1}\gamma_{t}}\sum_{k=0}^{K-1}\left(\norm{\bY^{1}_{t,k}}^{2}+\norm{\bY^{2}_{t,k}}^{2}\right)\\
&+\frac{2\sum_{t=0}^{T-1}\gamma^{2}_{t}}{K\sum_{t=0}^{T-1}\gamma_{t}}\sum_{k=0}^{K-1}\left(\inner{\bY^{1}_{t,k},\bu^{1}_{t,k}-\bz^{(t)}_{k+1/2}}+\inner{\bY^{2}_{t,k},\bu_{t,k}^{2}-\bz^{(t)}_{k+1/2}}\right)\\
&+\frac{\sum_{t=0}^{T-1}\gamma_{t}\inner{\eps^{t}_{V}(\xi^{1:b_{t}}_{t})+\eps^{t}_{h}(\bW^{1:b_{t}}_{t}),\bx-\bar{\bz}^{t}}}{\sum_{t=0}^{T-1}\gamma_{t}}\\
&+\sum_{t=0}^{T-1}\frac{2\gamma^{2}_{t}}{K\sum_{s=0}^{T-1}\gamma_{s}}\sum_{k=0}^{K-1}\left(\norm{H^{\delta_{t},\eps_{t}}_{\bz^{(t)}_{k+1/2}}(\bW_{t,k+1/2})}^{2}+\norm{H^{\delta_{t},\eps_{t}}_{\bx^{t}}(\bW_{t,k+1/2})}^{2}\right)\\
&-\sum_{t=0}^{T-1}\frac{\gamma_{t}}{K\sum_{s=0}^{T-1}\gamma_{s}}\sum_{k=0}^{K-1}\inner{H^{\delta_{t},b_{t}}_{\bx^{t}}-H^{\delta_{t},b_{t},\eps_{t}}_{\bx^{t}}+H^{\delta_{t},\eps_{t}}_{\bx^{t}}(\bW_{t,k+1/2})-H^{\delta_{t}}_{\bx^{t}}(\bW_{t,k+1/2}),\bz^{(t)}_{k+1/2}-\bx}\\
&-\sum_{t=0}^{T-1}\frac{\gamma_{t}}{K\sum_{s=0}^{T-1}\gamma_{s}}\sum_{k=0}^{K-1}\inner{H^{\delta_{t},\eps_{t}}_{\bz^{(t)}_{k+1/2}}(\bW_{t,k+1/2})-H^{\delta_{t}}_{\bz^{(t)}_{k+1/2}}(\bW_{t,k+1/2}),\bz^{(t)}_{k+1/2}-\bx}.
\end{align*}
Performing the same bounding steps as done in the exact case, we readily arrive at the expression
\begin{align*}
&\inner{V(\bx),\bar{\bz}^{T}-\bx}+h(\bar{\bz}^{T})-h(\bx)\leq  \frac{C^{2}}{2K\sum_{t=0}^{T-1}\gamma_{t}}+2C^{2}\frac{\sum_{t=0}^{T-1}\gamma^{2}_{t}\sum_{k=0}^{K-1}(\scrL_{f}(\xi_{t,k+1/2})^{2}+\eta^{2}_{0})}{K\sum_{t=0}^{T-1}\gamma_{t}}\\
&+\frac{3C^{2}\sum_{t=0}^{T-1}\gamma_{t}\eta_{t}}{2\sum_{t=0}^{T-1}\gamma_{t}}+\frac{2C^{2}}{K}+\frac{L_{h}\sum_{t=0}^{T-1}\delta_{t}\gamma_{t}}{\sum_{t=0}^{T-1}\gamma_{t}}+\frac{\sum_{t=0}^{T-1}\gamma_{t}\left(C\norm{\eps^{t}_{V}(\xi^{1:b_{t}}_{t})}+C\norm{\eps^{t}_{h}(\bW^{1:b_{t}}_{t})}\right)}{\sum_{t=0}^{T-1}\gamma_{t}}\\
&+\frac{\sum_{t=0}^{T-1}\gamma^{3}_{t}}{K\sum_{t=0}^{T-1}\gamma_{t}}\sum_{k=0}^{K-1}\left(\norm{\bY^{1}_{t,k}}^{2}+\norm{\bY^{2}_{t,k}}^{2}\right)+\frac{2\sum_{t=0}^{T-1}\gamma^{2}_{t}}{K\sum_{t=0}^{T-1}\gamma_{t}}\sum_{k=0}^{K-1}\left(\inner{\bY^{1}_{t,k},\bu^{1}_{t,k}-\bz^{(t)}_{k+1/2}}+\inner{\bY^{2}_{t,k},\bu_{t,k}^{2}-\bz^{(t)}_{k+1/2}}\right)\\
&+\sum_{t=0}^{T-1}\frac{2\gamma^{2}_{t}}{K\sum_{s=0}^{T-1}\gamma_{s}}\sum_{k=0}^{K-1}\left(\norm{H^{\delta_{t},\eps_{t}}_{\bz^{(t)}_{k+1/2}}(\bW_{t,k+1/2})}^{2}+\norm{H^{\delta_{t},\eps_{t}}_{\bx^{t}}(\bW_{t,k+1/2})}^{2}\right)\\
&+\sum_{t=0}^{T-1}\frac{C\gamma_{t}}{K\sum_{s=0}^{T-1}\gamma_{s}}\sum_{k=0}^{K-1}\norm{H^{\delta_{t},b_{t}}_{\bx^{t}}-H^{\delta_{t},b_{t},\eps_{t}}_{\bx^{t}}+H^{\delta_{t},\eps_{t}}_{\bx^{t}}(\bW_{t,k+1/2})-H^{\delta_{t}}_{\bx^{t}}(\bW_{t,k+1/2})}\\
&+\sum_{t=0}^{T-1}\frac{C\gamma_{t}}{K\sum_{s=0}^{T-1}\gamma_{s}}\sum_{k=0}^{K-1}\norm{H^{\delta_{t},\eps_{t}}_{\bz_{k+1/2}}(\bW_{t,k+1/2})-H^{\delta_{t}}_{\bz^{(t)}_{k+1/2}}(\bW_{t,k+1/2})}
\end{align*}
We next estimate the error terms appearing because of the inexact feedback map in the coupling function. Lemma \ref{lem:boundvarianceH-inexact} yields 
$$
\Ex\left[\norm{H^{\delta_{t},\eps_{t}}_{\bz_{k+1/2}}(\bW_{t,k+1/2})}^{2}+\norm{H^{\delta_{t},\eps_{t}}_{\bx^{t}}(\bW_{t,k+1/2})}^{2}\vert\scrA_{t,k}\right]\leq 6\sum_{i\in\scrI}\left(\frac{n_{i}}{\delta_{t}}\right)^{2}(2L^{2}_{2,i}\eps^{2}_{t}+L_{1,i}^{2}\delta^{2}_{t})\equiv \alpha^{(1)}_{t}.
$$
Furthermore, the triangle inequality, Jensen's inequality, and \eqref{eq:varianceHeps} gives 
\begin{align*}
&\Ex\left[\norm{H^{\delta_{t},b_{t}}_{\bx^{t}}-H^{\delta_{t},b_{t},\eps_{t}}_{\bx^{t}}+H^{\delta_{t},\eps_{t}}_{\bx^{t}}(\bW_{k+1/2})-H^{\delta_{t}}_{\bx^{t}}(\bW_{k+1/2})}\vert\scrA_{t,k}\right]\leq \Ex\left[\norm{H^{\delta_{t},b_{t}}_{\bx^{t}}-H^{\delta_{t},b_{t},\eps_{t}}_{\bx^{t}}}\vert\scrA_{t,k}\right]\\
&+\Ex\left[\norm{H^{\delta_{t},\eps_{t}}_{\bx^{t}}(\bW_{k+1/2})-H^{\delta_{t}}_{\bx^{t}}(\bW_{k+1/2})}\vert\scrA_{t,k}\right]\\
&\leq \sqrt{\Ex\left[\norm{H^{\delta_{t},b_{t}}_{\bx^{t}}-H^{\delta_{t},b_{t},\eps_{t}}_{\bx^{t}}}^{2}\vert\scrA_{t,k}\right]}+\sqrt{\Ex\left[\norm{H^{\delta_{t},\eps_{t}}_{\bx^{t}}(\bW_{k+1/2})-H^{\delta_{t}}_{\bx^{t}}(\bW_{k+1/2})}^{2}\vert\scrA_{t,k}\right]}\\
&\leq \frac{2\eps_{t}}{\delta_{t}}\left(\frac{1}{\sqrt{b_{t}}}+1\right)\sqrt{\sum_{i\in\scrI}L_{2,i}^{2}n_{i}^{2}}\equiv\alpha^{(2)}_{t}.
\end{align*}
Lastly, we bound 
$$
\Ex\left[\norm{H^{\delta_{t},\eps_{t}}_{\bz_{k+1/2}}(\bW_{k+1/2})-H^{\delta_{t}}_{\bz_{k+1/2}}(\bW_{k+1/2})}\vert\scrA_{t,k}\right]\leq \frac{2\eps_{t}}{\delta_{t}}\sqrt{\sum_{i\in\scrI}L^{2}_{2,i}n^{2}_{i}}\equiv\alpha^{(3)}_{t}.
$$
We now set
\begin{align*} 
\scrD_{K,T}&\eqdef \frac{C^{2}}{2K\sum_{t=0}^{T-1}\gamma_{t}}+2C^{2}(L_{f}^{2}+\eta^{2}_{0})\frac{\sum_{t=0}^{T-1}\gamma^{2}_{t}}{\sum_{t=0}^{T-1}\gamma_{t}}\\
&+\frac{3C^{2}\sum_{t=0}^{T-1}\gamma_{t}\eta_{t}}{2\sum_{t=0}^{T-1}\gamma_{t}}+\frac{2C^{2}}{K}+\frac{L_{h}\sum_{t=0}^{T-1}\delta_{t}\gamma_{t}}{\sum_{t=0}^{T-1}\gamma_{t}}.
\end{align*}
Using the definition of the gap function \eqref{eq:gap}, we deduce 
\begin{align*}
\Ex[\Gamma(\bar{\bz}^{T})]&\leq \scrD_{K,T}+\Ex\left[\frac{\sum_{t=0}^{T-1}\gamma^{3}_{t}}{K\sum_{t=0}^{T-1}\gamma_{t}}\sum_{k=0}^{K-1}\left(\norm{\bY^{1}_{t,k}}^{2}+\norm{\bY^{2}_{t,k}}^{2}\right)\right]\\
&+\Ex\left[\frac{2\sum_{t=0}^{T-1}\gamma^{2}_{t}}{K\sum_{t=0}^{T-1}\gamma_{t}}\sum_{k=0}^{K-1}\left(\inner{\bY^{1}_{t,k},\bu^{1}_{t,k}-\bz^{(t)}_{k+1/2}}+\inner{\bY^{2}_{t,k},\bu_{t,k}^{2}-\bz^{(t)}_{k+1/2}}\right)\right]\\
&+\Ex\left[\frac{\sum_{t=0}^{T-1}\gamma_{t}(C\norm{\eps^{t}_{V}(\xi^{1:b_{t}}_{t})}+C\norm{\eps^{t}_{h}(\bW^{1:b_{t}}_{t})})}{\sum_{t=0}^{T-1}\gamma_{t}}\right]\\
&+\Ex\left[\sum_{t=0}^{T-1}\frac{2\gamma^{2}_{t}}{K\sum_{s=0}^{T-1}\gamma_{s}}\sum_{k=0}^{K-1}\left(\norm{H^{\delta_{t},\eps_{t}}_{\bz_{k+1/2}}(\bW_{t,k+1/2})}^{2}+\norm{H^{\delta_{t},\eps_{t}}_{\bx^{t}}(\bW_{t,k+1/2})}^{2}\right)\right]\\
&+\Ex\left[\sum_{t=0}^{T-1}\frac{C\gamma_{t}}{K\sum_{s=0}^{T-1}\gamma_{s}}\sum_{k=0}^{K-1}\norm{H^{\delta_{t},b_{t}}_{\bx^{t}}-H^{\delta_{t},b_{t},\eps_{t}}_{\bx^{t}}+H^{\delta_{t},\eps_{t}}_{\bx^{t}}(\bW_{t,k+1/2})-H^{\delta_{t}}_{\bx^{t}}(\bW_{t,k+1/2})}\right]\\
&+\Ex\left[\sum_{t=0}^{T-1}\frac{C\gamma_{t}}{K\sum_{s=0}^{T-1}\gamma_{s}}\sum_{k=0}^{K-1}\norm{H^{\delta_{t},\eps_{t}}_{\bz^{(t)}_{k+1/2}}(\bW_{t,k+1/2})-H^{\delta_{t}}_{\bz^{(t)}_{k+1/2}}(\bW_{t,k+1/2})}\right]\\
&\leq \scrD_{K,T}+\frac{2\sigma^{2}\sum_{t=0}^{T-1}\gamma^{3}_{t}}{\sum_{t=0}^{T-1}\gamma_{t}}+\frac{4C\sigma\sum_{t=0}^{T-1}\gamma^{2}_{t}}{\sum_{t=0}^{T-1}\gamma_{t}}+\frac{\sum_{t=0}^{T-1}\frac{\gamma_{t}}{\sqrt{b_{t}}}C(M_{V}+(\sum_{i\in\scrI}L^{2}_{h_{i}}n_{i}^{2})^{1/2}}{\sum_{t=0}^{T-1}\gamma_{t}}\\
&+\sum_{t=0}^{T-1}\frac{2\gamma^{2}_{t}\alpha^{(1)}_{t}}{\sum_{s=0}^{T-1}\gamma_{s}}+\sum_{t=0}^{T-1}\frac{C\alpha^{(2)}_{t}\gamma_{t}}{\sum_{s=0}^{T-1}\gamma_{s}}+\sum_{t=0}^{T-1}\frac{C\gamma_{t}\alpha^{(3)}_{t}}{\sum_{s=0}^{T-1}\gamma_{s}}
\end{align*}
Making the choice $\gamma_{t}=\eta_{t}=\delta_{t}=1/T$ as well as $K=T,b_{t}\geq T^{2}$ and $\eps_{t}=1/T^{2}$, we see that $\alpha^{(1)}_{t}=\scrO(1/T^{2}),\alpha^{(2)}_{t}=\scrO(1/T)$ and $\alpha^{(3)}_{t}=\scrO(1/T)$. It follows $\Ex[\Gamma(\bar{\bz}^{T})]=\scrO(C\sigma/T)$, which completes the proof of Theorem \ref{th:inexact}.

%% file: Conclusion.tex
%
In this work, we proposed a new solution approach to solve a fairly large class of stochastic hierarchical games. Using a combination of smoothing, zeroth-order gradient approximation, and iterative regularization, we develop a novel variance reduction method for stochastic VIs affected by general stochastic noise. We demonstrate consistency of the method by proving that solution trajectory converges almost surely to a particular equilibrium of the game and derive a $\scrO(1/T)$ convergence rate in terms of the expected gap function, using a suitably defined averaged trajectory. This rate result is robust to inexact solutions of the lower level problem of the follower and aligns with state-of-the-art variance reduction methods tailored to finite-sum problems. Furthermore, our approach is based on Tseng's splitting technique, which shares the same number of function calls as the popular extragradient method, but saves on one projection step. This implies that our scheme reduces the oracle complexity relative to vanilla mini-batch approaches and at the same time reduces the computational bottlenecks in every single iteration. This leaves open the door for many future investigations, involving bias and non-convexities, that we leave for future research.

%% file: Appendix_General.tex
%

Given a closed convex set $\scrX\subset\Rn$, we denote by $\Pi_{\scrX}:\Rn\to\scrX$ the orthogonal projector defined as 
$$
\Pi_{\scrX}(\bw):=\argmin_{\bx\in\scrX}\frac{1}{2}\norm{\bx-\bw}^{2}.
$$
This is the solution map of a strongly convex optimization problem with the following well-known properties.
\begin{lemma}\label{lem:projector}
Let $\scrX\subset\Rn$ be a nonempty closed convex set. Then:
\begin{itemize}
\item[(i)] $\Pi_{\scrX}(\bw)$ is the unique point satisfying $\inner{\bw-\Pi_{\scrX}(\bw),\bx-\Pi_{\scrX}(\bw)}\leq 0$ for all $\bx\in\scrX$;
\item[(ii)] For all $\bw\in\Rn$ and $\bx\in\scrX$, we have $\norm{\Pi_{\scrX}(\bw)-\bx}^{2}+\norm{\Pi_{\scrX}(\bw)-\bw}^{2}\leq\norm{\bw-\bx}^{2}$; 
\item[(iii)] For all $\bw,\bv\in\Rn$, $\norm{\Pi_{\scrX}(\bw)-\Pi_{\scrX}(\bv)}^{2}\leq\norm{\bw-\bv}^{2}$; 
\end{itemize}
\end{lemma}
Let $(\Omega,\scrF,\Pr)$ be a given probability space carrying a filtration $\F=\{\scrF_{k}\}_{k\geq 0}$. We call the tuple $(\Omega,\scrF,\F,\Pr)$ a discrete stochastic basis. Given a vector space $\scrK\subseteq\Rn$ with Borel $\sigma$-algebra $\scrB(\scrK)$, a $\scrK$-valued random variable is a $(\scrF,\scrB(\scrK))$-measurable map $f:\Omega\to\scrK$; we write $f\in L^{0}(\Omega,\scrF,\Pr;\scrK)$. For every $p\in[1,\infty]$, define the equivalence class of random variables $f\in L^{0}(\Omega,\scrF,\Pr;\scrK)$ with $\Ex[\norm{f}^{p}]^{1/p}<\infty$ as $f\in L^{p}(\Omega,\scrF,\Pr;\scrK)$. For $f_{1},\ldots,f_{k}\in L^{p}(\Omega,\scrF,\Pr;\scrK)$, we denote the sigma-algebra generated by these random variables by $\sigma(f_{1},\ldots,f_{k})$. We denote by $\ell^{0}_{+}(\F)$ the set of non-negative random variables $\{\xi_{k}\}_{k\geq 0}$ such that for each $k\geq 0$, we have $\xi_{k}\in L^{0}(\Omega,\scrF_{k},\Pr;\R_{+})$. For $p\geq 1$, we set 
$$
\ell^{p}_{+}(\F)=\{\{\xi_{k}\}_{k\geq 0}\in\ell^{0}_{+}(\F)\vert\sum_{k\geq 0}\abs{\xi_{i}}^{p}<\infty\quad\Pr-\text{a.s.}\}.
$$

%% file: Appendix-VI.tex
%

In this appendix we summarize the essential parts from the theory of finite-dimensional variational inequalities we use in the paper. A complete treatment can be found in \cite{FacPan03}. 

The data of a variational inequality problem consist of mappings $\phi:\Rn\to\Rn$ and $r:\Rn\to(-\infty,+\infty]$ a proper, convex and lower semi-continuous function. Denote by $\dom(r)=\{\bx\in\Rn\vert r(\bx)<\infty\}$. The mixed variational inequality problems associated with $(\phi,r)$ is 
\begin{equation}\label{eq:MVI}\tag{$\MVI(\phi,r)$}
\text{find }\bx\in\Rn\text{ such that }\inner{\phi(\bx),\by-\bx}+r(\by)-r(\bx)\geq 0\qquad\forall\by\in\Rn. 
\end{equation}
When $r=\delta_{\scrK}$ for a closed convex set $\scrK\subset\Rn$, the problem $\MVI(\phi,r)$ reduces to the classical variational inequality $\VI(\phi,\scrK)$: 
\begin{equation}\label{eq:VI}\tag{$\VI(\phi,\scrK)$}
\text{find }\bx\in\scrK\text{ such that }\inner{\phi(\bx),\by-\bx}\geq 0\qquad\forall\by\in\scrK. 
\end{equation}
We note in passing that if $r$ is proper, convex and lower semi-continuous, then problem $\MVI(\phi,r)$ is equivalent to the generalized equation 
\begin{equation}\label{eq:MI}
0\in \phi(\bx)+\partial r(\bx),
\end{equation}
where $\partial r(\bx)\eqdef\{p\in\Rn\vert r(\bx')\geq r(\bx)+\inner{p,\bx'-\bx}\quad\forall \bx'\in\Rn\}$ is the subgradient of $r$ at $\bx$.

In order to measure the distance of a candidate point to the solution set we introduce as a merit function for $\MVI(\phi,r)$ the gap function
$$
\Gamma_{\scrK}(\bx)\eqdef\sup_{\bz\in\scrK}\left(\inner{\phi(\bz),\bx-\bz}+r(\bx)-r(\bz)\right),
$$
where $\scrK\subset\Rn$ is a compact subset to handle the possibility of unboundedness of $\dom(r)$. As proven in \cite{Nes07}, this restricted version of the gap function is a valid measure as long as $\scrK$ contains any solution of $\MVI(\phi,r)$.\\

For existence and uniqueness questions of variational problems, we usually rely on monotonicity and continuity properties of the map $\phi$. 
\begin{definition}
A mapping $\phi:\Rn\to\Rn$ is said to be $\mu$-monotone if there exists $\mu\geq 0$ such that 
$$
\inner{\phi(\bx)-\phi(\by),\bx-\by}\geq \mu\norm{\bx-\by}^{2}\qquad \forall \bx,\by\in\Rn.
$$
A $0$-monotone mapping is called monotone. 
\end{definition}
\begin{fact}[Solution Convexity of Monotone VIs]
Consider the problem $\VI(\phi,\scrK)$, where $\phi:\Rn\to\Rn$ is monotone on $\scrK\subset\dom(\phi)$, and $\scrK$ is a closed convex set. Then, the solution set 
$$
\scrS=\{\bx\in\scrK\vert \inner{\phi(\bx),\by-\bx}\geq 0\qquad\forall\by\in\scrK\}
$$
is closed and convex. If $\phi$ is $\mu$-monotone with $\mu>0$ on $\scrK$, then $\scrS$ is a singleton. 
\end{fact}
See \cite[Theorem 2F.1, 2F.6]{Dontchev:2009tp} for a proof of this Fact. 
We now extend the scope of variational inequalities and introduce parameters into the problem data. This is effectively the lower level problem solved by the followers in our hierarchical game model. Let $\scrX\subset\Rn$ be a nonempty compact convex set, and $\scrY\subset\R^{m}$ a closed convex set. The object of study is the parameterized generalized equation 
\begin{equation}\label{eq:ParaGE}
0\in \phi(\bx,\by)+\NC_{\scrY}(\by) 
\end{equation}
for a given function $\phi:\Rn\times\R^{m}\to\R^{m}$ and the normal cone 
\begin{equation}\label{eq:NC}
\NC_{\scrY}(\by)=\left\{\begin{array}{ll} 
\emptyset & \text{if }\by\notin\scrY,\\
\{\xi\in\R^{m}\vert\sup_{\bz\in\scrY}\inner{\xi,\bz-\by}\leq 0\quad \bz\in\scrY\} &  \text{if }\by\in\scrY. 
\end{array}
\right.
\end{equation}
Specifically, we are interested in understanding the properties of the solution mapping 
 \begin{equation}
 \scrS(\bx)\eqdef\{\by\in\R^{m}\vert 0\in \phi(\bx,\by)+\NC_{\scrY}(\by)\}
 \end{equation}
 This is a subclass of classical problems, thoroughly summarized in \cite{Dontchev:2009tp}, and dating back to the landmark paper \cite{Robinson:1980aa}. The interested reader can find proofs of the facts stated below, as well as much more information on this topic in these references. We point out that many of the strong assumption made below can be relaxed, at the price of more complicated verification steps. Our aim is to present a simple and not entirely unrealistic set of verifiable conditions under which our model assumptions provably hold; A more general result can be found in \cite[][Lemma 2.2]{XuSIOPT06}.
\begin{fact}\label{fact:hierarchicalVI}
Consider problem \eqref{eq:ParaGE} with the following assumptions on the problem data: 
\begin{itemize}
\item $\scrX\subset\Rn$ is compact and convex,
\item $\scrY\subset\R^{m}$ is closed convex,
\item $\phi:\Rn\times\R^{m}\to\R^{m}$ is strictly differentiable on $\scrK\times\scrY$, where $\scrK$ is an open set containing $\scrX$;
\item $\phi(x,\cdot)$ is strongly monotone for every $\bx\in\scrX$. 
\end{itemize}
Then $\scrS(\bx)=\{\by(\bx)\}$, and $\by(\cdot)$ is Lipschitz continuous on $\scrX$.  
\end{fact}

\subsection{Tikhonov Regularization}
\label{app:Tik}
Tikhonov regularization is a classical method in numerical analysis aiming for introducing additional stability into a computational scheme. Given problem $\MVI(\phi,r)$ and $\eta>0$, we define the Tikhonov regularized mixed variational inequality problem as $\MVI(\phi^{\eta},r)$, in which the operator is defined as $\phi^{\eta}(\bx)\eqdef\phi(\bx)+\eta\bx$. It is easy to see that if $\phi$ is $0$-monotone, then $\phi^{\eta}$ is $\eta$-monotone (i.e. strongly monotone). Hence, for every $\eta>0$, problem $\MVI(\phi^{\eta},r)$ has a unique solution $\bx^{\ast}(\eta)$. The first result we are going to demonstrate is that the net $\{\bx^{\ast}(\eta)\}_{\eta\geq 0}$ is bounded. 
\begin{proposition}
\label{prop:Tikhonovbounded}
Consider problem $\MVI(\phi,r)$ admitting a nonempty solution set $\SOL(\phi,r)$. Then, for all $\eta>0$, we have 
$$
\norm{\bx^{\ast}(\eta)}\leq \inf_{\bx\in\SOL(\phi,r)}\norm{\bx}.
$$
\end{proposition}
\begin{proof}
Using the characterization of a point $\bx^{\ast}\in\SOL(\phi,r)$ as a solution of a monotone inclusion, we have 
$$
-\phi^{\eta}(\bx^{\ast}(\eta))\in \partial r(\bx^{\ast}(\eta)) \text{ and }-\phi(\bx^{\ast})\in\partial r(\bx^{\ast}).
$$
Since $\partial r$ is maximally monotone, it follows 
\begin{align*}
&\inner{\phi^{\eta}(\bx^{\ast}(\eta))-\phi(\bx^{\ast}),\bx^{\ast}-\bx^{\ast}(\eta)}\geq 0\\
&\iff \inner{\phi(\bx^{\ast})-\phi(\bx^{\ast}(\eta)),\bx^{\ast}-\bx^{\ast}(\eta)}\leq \eta\inner{\bx^{\ast}(\eta),\bx^{\ast}-\bx^{\ast}(\eta)}.
\end{align*}
Since $\phi(\cdot)$ is monotone, it follows $\inner{\phi(\bx^{\ast})-\phi(\bx^{\ast}(\eta)),\bx^{\ast}-\bx^{\ast}(\eta)}\geq 0$, so that 
\begin{align*}
0\leq\inner{\bx^{\ast}(\eta),\bx^{\ast}-\bx^{\ast}(\eta)}=\inner{\bx^{\ast}(\eta),\bx^{\ast}}-\norm{\bx^{\ast}(\eta)}^{2}.
\end{align*}
The Cauchy-Schwarz inequality implies $\norm{\bx^{\ast}(\eta)}\leq\norm{\bx^{\ast}}$. Since $\bx^{\ast}$ has been chosen arbitrarily, the claim follows. 
\end{proof}

We next study the asymptotic regime in which $\eta\to 0^{+}$. Since the net $\{\bx^{\ast}(\eta)\}_{\eta>0}$ is bounded, the Bolzano-Weierstrass theorem guarantees the existence of a converging subsequence $\eta_{t}\to 0$ such that $\bx^{\ast}(t)\equiv \bx^{\ast}(\eta_{t})\to \hat{\bx}$. Since $\partial r$ is maximally monotone, the set $\gr(\partial r)$ is closed in the product topology \citep{BauCom16}. Hence, $\hat{\bx}\in\dom(\partial r)$. Moreover, for all $t$ 
$$
(\bx^{\ast}(t),-\phi(\bx^{\ast}(t))-\eta_{t}\bx^{\ast}(t))\in\gr(\partial r)\qquad\forall t>0.
$$
Continuity and Proposition \ref{prop:Tikhonovbounded}, together with the just mentioned closed graph property, yields for $t\to\infty$,
$$
(\hat{\bx},-\phi(\hat{\bx}))\in\gr(\partial r)\qquad\forall t>0.
$$
The next claim follows.
\begin{proposition}\label{prop:accumulationTik}
Every accumulation point of the Tikhonov sequence $\{\bx^{\ast}(\eta)\}_{\eta>0}$ defines a solution of the problem $\MVI(\phi,r)$. 
\end{proposition}
We next deduce a non-asymptotic estimate of the Tikhonov sequence. Let $\{\eta_{t}\}_{t\in\N}$ be a positive sequence of regularization parameters satisfying $\eta_{t}\downarrow 0$. Exploiting again the variational characterization of the unique solutions $\bx^{\ast}_{t}\equiv\bx^{\ast}(\eta_{t})$, we have 
$$
-\phi^{\eta_{t-1}}(\bx^{\ast}_{t-1})\in\partial r(\bx^{\ast}_{t-1})\text{ and }-\phi^{\eta_{t}}(\bx^{\ast}_{t})\in\partial r(\bx^{\ast}_{t}).
$$
By monotonicity of $\partial r$, we obtain 
$$
\inner{\phi^{\eta_{t-1}}(\bx^{\ast}_{t-1})-\phi^{\eta_{t}}(\bx^{\ast}_{t}),\bx^{\ast}_{t}-\bx^{\ast}_{t-1}}\geq 0.
$$
Hence, by monotonicity of $\phi$, it follows
\begin{align*}
0&\geq \inner{\phi(\bx^{\ast}_{t-1})-\phi(\bx^{\ast}_{t}),\bx^{\ast}_{t}-\bx^{\ast}_{t-1}}\geq \inner{\eta_{t}\bx^{\ast}_{t}-\eta_{t-1}\bx^{\ast}_{t-1},\bx^{\ast}_{t}-\bx^{\ast}_{t-1}}\\
&=\eta_{t}\inner{\bx^{\ast}_{t}-\bx^{\ast}_{t-1},\bx^{\ast}_{t}-\bx^{\ast}_{t-1}}+(\eta_{t}-\eta_{t-1})\inner{\bx_{t-1}^{\ast},\bx^{\ast}_{t}-\bx^{\ast}_{t-1}}.
\end{align*}
Whence, 
$$
\eta_{t}\norm{\bx^{\ast}_{t}-\bx^{\ast}_{t-1}}^{2}\leq (\eta_{t}-\eta_{t-1})\inner{\bx^{\ast}_{t-1},\bx^{\ast}_{t-1}-\bx^{\ast}_{t}}\leq (\eta_{t}-\eta_{t-1})\norm{\bx^{\ast}_{t-1}}\cdot\norm{\bx^{\ast}_{t}-\bx^{\ast}_{t-1}}.
$$
The next claim follows:
\begin{proposition}
For any monotonically decreasing sequence $\{\eta_{t}\}_{t\in\N}\subset(0,\infty)$ satisfying $\eta_{t}\downarrow 0$ we have 
$$
\left(\frac{\eta_{t}-\eta_{t-1}}{\eta_{t}}\right)\inf_{\bx\in\SOL(\phi,r)}\norm{\bx}\geq\norm{\bx^{\ast}(\eta_{t})-\bx^{\ast}(\eta_{t-1})}.
$$
\end{proposition}
Lastly, we provide an exact localization result on the Tikhonov sequence.
\begin{proposition}\label{prop:Tikconverge}
Let $\eta_{t}\downarrow 0$ and $\bx^{\ast}_{t}\equiv\bx^{\ast}(\eta_{t})$ the corresponding sequence of solutions to the regularized problem $\MVI(\phi^{\eta_{t}},r)$. Then, $\inf_{\bx\in\SOL(\phi,r)}\norm{\bx}$ exists and is uniquely attained and $\bx^{\ast}_{t}\to \arg\min_{\bx\in\SOL(\phi,r)}\norm{\bx}$.
\end{proposition}
\begin{proof}
The set $\SOL(\phi,r)$ agrees with the zeros of the monotone inclusion problem \eqref{eq:MI}. Since $\partial r$ is maximally monotone, and $\phi$ is continuous and monotone, it follows from Corollary 24.4 in \cite{BauCom16} that the set of zeros is closed and convex. Hence, the problem $\inf_{\bx\in\SOL(\phi,r)}\norm{\bx}$ admits a unique solution, proving the first part of the Proposition. For the second part, let $\eta_{t}\downarrow 0$ and $\bx^{\ast}_{t}\equiv\bx^{\ast}(\eta_{t})$ the corresponding sequence of unique solution of $\MVI(\phi^{\eta_{t}},r)$. Since $\{\bx^{\ast}_{t}\}_{t}$ is bounded (Proposition \ref{prop:Tikhonovbounded}), we can pass to a converging subsequence. By an abuse of notation, omitting the relabeling, let us take the full sequence to be converging with limit point $\hat{\bx}$. By Proposition \ref{prop:accumulationTik}, we know $\hat{\bx}\in\SOL(\phi,r)$. In particular, $\norm{\hat{\bx}}\geq \inf_{\bx\in\SOL(\phi,r)}\norm{\bx}$. But then, in view of Proposition \ref{prop:Tikhonovbounded}, it follows $\norm{\hat{\bx}}= \inf_{\bx\in\SOL(\phi,r)}\norm{\bx}$. Since the accumulation point $\hat{\bx}$ is arbitrary, the entire sequence $\{\bx^{\ast}_{t}\}_{t}$ converges with limit $\arg\min_{\bx\in\SOL(\phi,r)}\norm{\bx}$. \qed
\end{proof}

%% file: Smoothing.tex
%

We let $\ball_{n}\eqdef\{\bx\in\Rn\vert\norm{\bx}\leq 1\}$ denote the unit ball in $\Rn$. The unit sphere is denoted by $\sphere_{n}=\{\bx\in\Rn\vert\norm{\bx}=1\}$. The volume of the unit ball with radius $\delta$ with respect to $n$-dimensional Lebesgue measure is $\Vol_{n}(\delta \ball_{n})=\delta^{n}\frac{\pi^{n/2}}{\Gamma(\frac{n}{2}+1)}$, where $\Gamma(\cdot)$ is the Gamma function. Therefore, the measure 
$\dd\mu_{n}(\bf{u})\eqdef \1_{\{\bf{u}\in\ball_{n}\}}\frac{\dd\bf{u}}{\Vol_{n}(\ball_{n})}$ defines a uniform distribution on the unit ball in $\Rn$. Recall that $\Vol_{n-1}(\delta \sphere_{n})=\frac{n}{\delta}\Vol_{n}(\delta\ball_{n})$ for all $\delta>0$. Given $\delta>0$, and $\scrX$ be a closed convex set in $\Rn$. We define the set $\scrX_{\delta}\eqdef\scrX+\delta\ball_{n}$.
\begin{definition}
Let $h:\Rn\to\R$ be a continuous function. The spherical smoothing of $h$ is defined by
\begin{equation}\label{eq:smoothing}
h^{\delta}(\bx)\eqdef \frac{1}{\Vol_{n}(\delta\ball_{n})} \int_{\delta \ball_{n}}h(\bx+\bu)\dd \bu=\int_{\ball_{n}}h(\bx+\delta\bw)\frac{\dd\bw}{\Vol_{n}(\ball_{n})}.
\end{equation}
\end{definition}
The following properties of the spherical smoothing can be deduced from \cite[Section 9.3.2]{NY83}; see also \cite[Lemma 1]{Cui:2022aa}. 
\begin{fact}
Let $h:\Rn\to\R$ be a continuous function that is $L_{h}$-Lipschitz continuous on $\scrX_{\delta}$. Then, for all $\bx,\by\in\scrX$, we have
\begin{align}
&\abs{h^{\delta}(\bx)-h^{\delta}(\by)}\leq L_{h}\norm{\bx-\by},\\
&\abs{h^{\delta}(\bx)-h(\bx)}\leq L_{h}\delta.
\end{align}
\end{fact}
Using Stoke's Theorem, one can easily show that 
\begin{equation}\label{eq:gradsmoothing}
\begin{split}
\nabla h^{\delta}(\bx)&=\frac{n}{\delta}\int_{\sphere_{n}}h(\bx+\delta \bv)\bv\frac{\dd\bv}{\Vol_{n-1}(\sphere_{n})}\\
&=\frac{n}{\delta}\Ex_{\bW\sim\Uniform(\sphere_{n})}[\bW h(\bx+\delta\bW)]\\
&=\frac{n}{\delta}\Ex_{\bW\sim\Uniform(\sphere_{n})}[\bW\left( h(\bx+\delta\bW)-h(\bx)\right)],
\end{split}
\end{equation}
where $\bW\sim\Uniform(\sphere_{n})$ means that $\bW$ is uniformly distributed on $\sphere_{n}$. A simple application of Jensen's inequality shows then that the spherical smoothing admits a Lipschitz continuous gradient, whose modulus depends on the smoothing parameter $\delta$. 
\begin{fact}
Let $h:\Rn\to\R$ be a continuous function that is $L_{h}$-Lipschitz continuous on $\scrX_{\delta}$. Then, for all $\bx,\by\in\scrX$, we have
\begin{equation}\label{eq:boundgradsmooth}
\norm{\nabla h^{\delta}(\bx)-\nabla h^{\delta}(\by)}\leq \frac{L_{h}n}{\delta}\norm{\bx-\by}.
\end{equation}
\end{fact}
The smoothed function and its gradient is used to construct a variance reduced gradient estimator for the implicit cost function of player $i$ in our hierarchical game problem. 

\subsection{Random sampling}
\label{sec:sampling}
In this section we explain how to construct a random oracle to sample the gradient of the smoothed implicit function $h_{i}^{\delta}$. Let $b\in\N$ denote the batch size. In each round of the algorithm, agent $i$ enters the outer loop procedure, which asks this agent to construct a $n_{i}\times b$ matrix $\bW_{i}^{1:b}=[\bW_{i}^{(1)};\ldots;\bW_{i}^{(b)}]$ satisfying $\norm{\bW_{i}^{(s)}}=1$ for all $1\leq s\leq b$. To construct the uniformly distributed unit vector $\bW_{i}^{(s)}$, we generate $n_{i}$ random numbers $w_{i}^{s}(k)\sim\Normal(0,1)$, and then compute  
$$
\bW_{i}^{(s)}=\frac{1}{\sqrt{\sum_{k=1}^{n_{i}}w_{i}^{s}(k)^{2}}}[w_{i}^{(s)}(1);\ldots;w_{i}^{(s)}(n_{i})]^{\top}\quad s=1,\ldots,b.
$$
The outcome of this procedure is a $n_{i}\times b$ random matrix $\bW^{1:b}_{i}=[\bW_{i}^{(1)};\ldots;\bW^{(b)}_{i}]\in\R^{n_{i}\times b}$ with $\norm{\bW_{i}^{(s)}}=1$ for all $1\leq s\leq b$. In fact, since the Gaussian is spherical, the columns of this matrix will be iid uniformly distributed on $\sphere_{n_{i}}$. Having constructed these random vectors, each agent constructs a gradient estimator involving the finite-difference approximation of the directional derivative 
$$
H^{\delta}_{\bx_{i}}(\bW_{i})\eqdef n_{i}\bW_{i}\nabla_{(\bW_{i},\delta)}h_{i}(\bx_{i}),
,\quad \nabla_{(\bw,\delta)}h_{i}(\bx_{i})\eqdef\frac{h_{i}(\bx_{i}+\delta\bw)-h_{i}(\bx_{i})}{\delta},
$$
as well as its Monte-Carlo variant (with some abuse of notation)
$$
H^{\delta,b}_{\bx_{i}}\eqdef \frac{1}{b}\sum_{s=1}^{b}H^{\delta}_{\bx_{i}}(\bW_{i}^{(s)}).
$$
Note that, since $h_{i}(\cdot)$ is $L_{1,i}$-Lipschitz and directionally differentiable on the convex compact set $\scrX$, we have 
$$
\lim_{\delta\to 0^{+}}\nabla_{(\bw,\delta)}h_{i}(\bx)=h^{\circ}_{i}(\bx_{i},\bw) 
$$
as well $\abs{h^{\circ}_{i}(\bx,\bw)}\leq L_{1,{i}}$ for all $\bw\in\sphere_{n_{i}}$. To understand the statistical properties of this estimator, we need the next Lemma. To simplify the notation, we omit the index of player $i$. 
\begin{lemma}\label{lem:boundH}
Suppose $h$ is $L_{h}$-Lipschitz continuous on $\scrX_{\delta}\eqdef\scrX+\delta\ball$. Define $\ce_{\bx}(\bW^{1:b})\eqdef H^{\delta,b}_{\bx}-\nabla h^{\delta}(\bx)$, where $\bW^{(i)}$ is an i.i.d sample drawn uniformly from the unit sphere $\sphere_{n}$, i.e. $\bW^{1:b}\sim\Uniform(\sphere_{n})^{\otimes b}$. Then 
\begin{itemize}
\item[(a)] $\Ex_{\bW^{1:b}\sim\Uniform(\sphere_{n})^{\otimes b}}[\ce_{\bx}(\bW^{1:b})]=0$;
\item[(b)] $\norm{H_{\bx}^{\delta}(\bw)}^{2}\leq L_{h}^{2}n^{2}$ for all $\bw\in\sphere_{n}$;
\item[(c)] $\Ex_{\bW^{1:b}\sim\Uniform(\sphere_{n})^{\otimes b}}[\norm{\ce_{\bx}(\bW^{1:b})}^{2}]\leq \frac{n^{2}L^{2}_{h}}{b}.$
\end{itemize}
\end{lemma}
\begin{proof}
By linearity of the expectation operator and independence, we see 
\begin{align*}
\Ex_{\bW^{1:b}\sim\Uniform(\sphere_{n})^{\otimes b}}[H^{\delta}_{\bx}(\bW^{(1)},\ldots,\bW^{(b)})]&=\frac{1}{b}\sum_{s=1}^{b}\Ex_{\bW\sim\Uniform(\sphere_{n})}[H^{\delta}_{\bx}(\bW)]=\Ex_{\bW\sim\Uniform(\sphere_{n})}[H^{\delta}_{\bx}(\bW)]\\
&\stackrel{\eqref{eq:gradsmoothing}}{=}\nabla h^{\delta}(\bx).
\end{align*}
This proves part (a). Part (b) is a simple consequence of the following Lipschitz argument:
$$
\norm{H^{\delta}_{\bx}(\bw)}^{2}=\left(\frac{n}{\delta}\right)^{2}\norm{\bw}^{2}\abs{h(\bx+\delta\bw)-h(\bx)}^{2}\leq L^{2}_{h}n^{2},
$$
using $\norm{\bw}=1$. For part (c), observe that for the random i.i.d. sample $\bW^{1:b}=\{\bW^{(1)},\ldots,\bW^{(b)}\}$ taking values in $\sphere_{n}$, we have
\begin{align*}
\norm{\ce_{\bx}(\bW^{1:b})}^{2}=\frac{1}{b^{2}}\norm{\sum_{i=1}^{b}\left(\frac{n}{\delta}H_{\bx}^{\delta}(\bW^{(i)})-\nabla h^{\delta}(\bx)\right)}^{2}=\frac{1}{b^{2}}\norm{\sum_{i=1}^{b}X_{i}}^{2},
\end{align*}
where $X_{i}\eqdef \frac{n}{\delta}H_{\bx}^{\delta}(\bW^{(i)})-\nabla h^{\delta}(\bx)$ are i.i.d zero-mean random variables, almost surely bounded in squared norm. By independence, we have $\Ex[\inner{X_{i},X_{j}}]=0$ for $i\neq j$, so that 
\begin{align*}
\Ex_{\bW^{1:b}\sim\Uniform(\sphere_{n})^{\otimes b}}\left[\norm{\ce_{\bx}(\bW^{1:b})}^{2}\right]&=\frac{1}{b}\left(\Ex_{\bW\sim\Uniform(\sphere_{n})}\left[\norm{H_{\bx}^{\delta}(\bW)}^{2}\right]-\norm{\nabla h^{\delta}(\bx)}^{2}\right)\\
&\leq \frac{1}{b}\Ex_{\bW\sim\Uniform(\sphere_{n})}\left[\norm{H^{\delta}_{\bx}(\bW)}^{2}\right]\leq \frac{n^{2}L_{h}^{2}}{b}.
\end{align*}
\qed
\end{proof}

%% file: inexactImplementation.tex
%

In the main text we have assumed that the solution of the lower level problem is available exactly. In practice, this is difficult to guarantee, particularly when the lower-level problem is large and possible stochastic. Motivated by this concern, we  outline a modification of our hierarchical game solver in this section, reliant on access to an $\eps$-inexact solution of the lower-level problem.
\begin{definition}
\label{def:inexact}
    Let $\delta_{0}>0$ be given and set $\scrX_{i,\delta_{0}}=\scrX_{i}+\delta_{0}\ball_{n_{i}}$ for all $i\in\scrI$. Given $\eps>0$ and $i \in \scrI$, we call $\by^{\eps}_{i}:\scrX_{i,\delta}\to\scrY_{i}$ an $\eps$-solution of the lower level problem $\VI(\phi_{i}(\bx_{i},\cdot),\scrY_{i})$ if 
$$
\Ex[\norm{\by_{i}^{\eps}(\bx_{i})-\by_{i}(\bx_{i})}\vert\bx_{i}]\leq\eps\quad \quad\text{a.s. } 
$$
\end{definition}

We note that such a solution is immediately available by employing $\mathcal{O}(1/\eps^2)$ steps  of a single-sample stochastic approximation scheme for resolving VI$(\phi_i(\bx_i,\cdot),\scrY_i)$. Similarly, if projection onto $\scrY_i$ is a computationally costly operation, then a geometrically increasing mini-batch scheme provides a similar oracle complexity but requires only $\scrO(\ln(1/\eps^2))$ steps. (cf.~\cite{lei22distributed}).

Under the inexact lower level solution $\by_{i}^{\eps}$, we let $h_{i}^{\eps}(\bx_{i})=g_{i}(\bx_{i},\by_{i}^{\eps}(\bx_{i}))$ denote the resulting implicit function coupling leader $i$ and the associated follower. As in the exact regime, we assume that player $i$ has access to an oracle with which she can construct a spherical approximation of the gradient of the implicit function $h_{i}^{\eps}$. Hence, for given $\delta>0$, we let 
$$
h^{\eps,\delta}_{i}(\bx_{i})\eqdef\int_{\ball_{n_{i}}}h^{\eps}_{i}(\bx_{i}+\delta\bw)\frac{\dd\bw}{\Vol_{n}(\ball_{n_{i}})}.
$$
Using the notation for Section \ref{sec:sampling}, we denote the resulting estimators by $H^{\delta,\eps}_{i,\bx_{i}}(\bW_{i})\eqdef n_{i}\bW_{i}\nabla_{(\bW_{i},\delta)}h^{\eps}_{i}(\bx_{i})$, while the corresponding mini-batch counterpart as 
$$
H^{\delta,\eps,b}_{i,\bx_{i}}\eqdef \frac{1}{b}\sum_{s=1}^{b}H^{\delta,\eps}_{i,\bx_{i}}(\bW_{i}^{(s)}),
$$
Next, we derive some bounds of the thus constructed estimator. To reduce notational clutter, we omit the label of player $i$ in the next Lemma. 
\begin{lemma}
\label{lem:boundvarianceH-inexact}
    Let Assumption \ref{ass:standing} hold true. Then, for all $\delta,\eps>0$ and $\bx\in\scrX$, we have  
\begin{align}
&\Ex\left[\norm{H^{\delta,\eps}_{\bx}(\bW)}^{2}\right\vert\bx]\leq 3\left(\frac{n}{\delta}\right)^{2}(2L^{2}_{2}\eps^{2}+L_{1}^{2}\delta^{2})\label{eq:boundHeps} \text{ and }\\
&\Ex\left[\norm{H^{\delta,\eps}_{\bx}(\bW)-H^{\delta}_{\bx}(\bW)}^{2}\vert\bx\right]\leq \left(\frac{2L_{2,h}n\eps}{\delta}\right)^{2}\label{eq:varianceHeps}
\end{align}
almost surely.
\end{lemma}
\begin{proof}
We have 
\begin{align*}
&\norm{H^{\delta,\eps}_{\bx}(\bW)}^{2}=\left(\frac{n}{\delta}\right)^{2}\norm{\bW\left[h^{\eps}(\bx+\delta\bW)-h(\bx+\delta\bW)+h(\bx+\delta\bW)+h(\bx)-h(\bx)-h^{\eps}(\bx)\right]}^{2}\\
&\leq 3\left(\frac{n}{\delta}\right)^{2}\left[\norm{\bW(h^{\eps}(\bx+\delta\bW)-h(\bx+\delta\bW)}^{2}+\norm{\bW(h^{\eps}(\bx)-h(\bx))}^{2}+\norm{\bW(h(\bx+\delta\bW)-h(\bx))}^{2}\right].
\end{align*}
We bound each of the three terms separately. First, by Assumption \ref{ass:standing}.(iv), we note
\begin{align*}
\norm{\bW(h^{\eps}(\bx+\delta\bW)-h(\bx+\delta\bW)}^{2}&\leq \norm{g(\bx+\delta\bW,\by^{\eps}(\bx+\delta\bW))-g(\bx+\delta\bW,\by(\bx+\delta\bW))}^{2}\\
&\leq L^{2}_{2}\norm{\by^{\eps}(\bx+\delta\bW)-\by(\bx+\delta\bW)}^{2}.
\end{align*}
Consequently, using Definition \ref{def:inexact}, we obtain 
$$
\Ex\left[\norm{\bW(h^{\eps}(\bx+\delta\bW)-h(\bx+\delta\bW)}^{2}\vert\bx\right]\leq L^{2}_{2}\eps^{2}.
$$
Second,
\begin{align*}
\norm{\bW(h^{\eps}(\bx)-h(\bx))}^{2}\leq \norm{g(\bx,\by^{\eps}(\bx))-g(\bx,\by(\bx)}^{2}\leq L^{2}_{2}\norm{\by(\bx)-\by^{\eps}(\bx)}^{2}.
\end{align*}
Invoking again Definition \ref{def:inexact}, it follows 
$$
\Ex\left[\norm{\bW(h^{\eps}(\bx)-h(\bx))}^{2}\vert\bx\right]\leq L^{2}_{2}\eps^{2}.
$$
Third, by Assumption \ref{ass:standing}.(ii) and since $\bW\in\sphere$, we conclude
\begin{align*}
\norm{\bW(h(\bx+\delta\bW)-h(\bx))}^{2}\leq L_{1}^{2}\norm{\delta\bW}^{2}=L^{2}_{1}\delta^{2}.
\end{align*}
Summarizing all these bounds, we obtain \eqref{eq:boundHeps}. To show \eqref{eq:varianceHeps}, we first note 
\begin{align*}
\norm{H^{\delta,\eps}_{\bx}(\bW)-H^{\delta}_{\bx}(\bW)}&\leq n\norm{\bW\frac{g(\bx+\delta\bW,\by^{\eps}(\bx+\delta\bW))-g(\bx+\delta\bW,\by(\bx+\delta\bW))}{\delta}}\\
&+n\norm{\bW\frac{g(\bx,\by^{\eps}(\bx))-g(\bx,\by(\bx))}{\delta}}\\
&\leq \frac{L_{2,h}n}{\delta}\norm{\by^{\eps}(\bx+\delta\bW)-\by(\bx+\delta\bW)}+\frac{nL_{2,h}}{\delta}\norm{\by^{\eps}(\bx)-\by(\bx)}.
\end{align*}
Taking expectations on both sides, it follows 
$$
\Ex\left[\norm{H^{\delta,\eps}_{\bx}(\bW)-H^{\delta}_{\bx}(\bW)}\vert\bx\right]\leq \frac{2L_{2,h}n\eps}{\delta}.
$$
\qed
\end{proof}